\documentclass[11pt,english,reqno]{amsart}
\usepackage[T1]{fontenc}
\usepackage[latin9]{inputenc}
\usepackage{enumitem}
\usepackage{geometry}
\usepackage{verbatim}
\usepackage{mathrsfs}
\usepackage{amstext}
\usepackage{amsthm}
\usepackage{amssymb}
\usepackage{hyperref}
\hypersetup{
     colorlinks   = true,
     citecolor    = blue,
     linkcolor    = blue
}
\usepackage{pdfsync}
\usepackage{verbatim}

\newcommand{\beq}{\begin{equation}}
\newcommand{\eeq}{\end{equation}}

\newcommand{\bx}{\mathbf{x}}
\newcommand{\by}{\mathbf{y}}
\newcommand{\bz}{\mathbf{z}}

\newcommand{\cac}{\mathcal C}

\newcommand{\cl}{\mathcal L}

\newcommand{\RR}{\mathbb{R}}

\newcommand{\al}{\alpha}

\newcommand{\ga}{\gamma}

\newcommand{\lp}{\left(}
\newcommand{\rp}{\right)}
\newcommand{\lc}{\left[}
\newcommand{\rc}{\right]}

\makeatletter
\numberwithin{equation}{section}
\numberwithin{figure}{section}
\theoremstyle{plain}
\newtheorem{thm}{\protect\theoremname}
\theoremstyle{definition}
\newtheorem{defn}[thm]{\protect\definitionname}
\theoremstyle{plain}
\newtheorem{prop}[thm]{\protect\propositionname}
\theoremstyle{remark}
\newtheorem{rem}[thm]{\protect\remarkname}
\theoremstyle{plain}
\newtheorem{lem}[thm]{\protect\lemmaname}
\theoremstyle{plain}

\makeatother

\usepackage{babel}
\providecommand{\corollaryname}{Corollary}
\providecommand{\definitionname}{Definition}
\providecommand{\lemmaname}{Lemma}
\providecommand{\propositionname}{Proposition}
\providecommand{\remarkname}{Remark}
\providecommand{\theoremname}{Theorem}

\begin{document}

\title[Volterra equations driven by rough signals]{Volterra equations driven by rough signals}

\author[F. Harang \and S. Tindel]{Fabian A. Harang \and Samy Tindel}
\begin{abstract}
This article is devoted to the extension of the theory of rough paths in the context of Volterra equations with possibly singular kernels. We begin to describe a class of two parameter functions defined on the simplex called Volterra paths. These paths are used to construct a so-called Volterra-signature, analogously to the signature used in Lyon's theory of rough paths. 
We  provide a detailed algebraic and analytic description of this object. Interestingly, the Volterra
signature does not have a multiplicative property similar to the classical
signature, and we introduce an integral product behaving like a convolution 
extending the classical tensor product.  We show that this convolution product is well defined for  a large class of Volterra paths, and we provide an analogue of the extension theorem from the theory of rough paths (which guarantees in particular the existence of a Volterra signature). Moreover  the concept of convolution product is essential in the construction of  Volterra controlled  paths, which is the natural class of processes to be integrated with respect to the driving noise in our situation. This leads to a rough integral given as a functional of the Volterra signature and the Volterra controlled paths, combined  through the convolution product. The rough integral is then used in the construction of solutions to Volterra equations driven by H\"older noises with singular kernels. An example concerning Brownian noises and a singular kernel is treated. 
\end{abstract}


\thanks{
\emph{ Acknowledgments}: F. Harang is supported by the Research Council of Norway (RCN). Project STORM, project number: 274410. S. Tindel is supported by the NSF grant  DMS-1613163.}

\address{Fabian A. Harang: Department of Mathematics, University of Oslo, P.O. box 1053, Blindern, 0316, OSLO, Norway}

\address{Samy Tindel: Department of Mathematics, Purdue University, 150 N. University Street, West Lafayette, IN 47907, United States}

\maketitle
{
\hypersetup{linkcolor=black}
 \tableofcontents 
}

\section{Introduction and main results}

Consider a Volterra equation of the second kind written as
\begin{equation}
u_{t}=f_{t}+\int_{0}^{t}k_{1}\left(t,r\right)b\left(u_{r}\right)dr+\int_{0}^{t}k_{2}\left(t,r\right)\sigma\left(u_{r}\right)dx_{r},
\label{eq:motiv: Volterra equations}
\end{equation}
 where $f$ is some initial condition, $b$ and $\sigma$ are sufficiently
smooth functions, $k_{1}$ and $k_{2}$ are possibly singular kernels,
and $x$ is an irregular signal (typically a (fractional) Brownian
motion). Integral equations on this form have several applications
to physics, biology or even finance. For example in physics, such
equations are used to model viscoelastic materials \cite{DietFred},
or in biology these equations may be used to model the spread of epidemics \cite{Crom}. Volterra equations also play a crucial role
in renewal theory \cite{Fell}, and is frequently used in stochastic volatility modelling where the driving noise is typically frequently chosen to be a rough fractional Brownian motion \cite{BFGMS19,ER19}.

\smallskip

From a mathematical point of view, Volterra equations have been studied for a long time. At a heuristic level,  in order to obtain existence and uniqueness of \eqref{eq:motiv: Volterra equations} one is typically confronted  with the regularity assumption of $b$ and $\sigma$, and the regularity of the initial data $f$ as well as the driving noise $x$. Additionally one needs some type of regularity on the kernels $k_1$ and $k_2$.  Although the conditions on $b,\sigma$ and $f$ ensuring existence and uniqueness in \eqref{eq:motiv: Volterra equations} are generally similar to the case of classical ODEs, the assumption on the noise $x$ and the kernels $k_1$ and $k_2$ are more challenging objects to analyse in this context. Typically, one searches for the most general conditions on $k_1,k_2$ and $x$ in order to still obtain existence and uniqueness for equation \eqref{eq:motiv: Volterra equations}. 

\smallskip

  The introduction of irregular controls in terms of a random or irregular path $x$, as illustrated in the second integral term in \eqref{eq:motiv: Volterra equations}, has been investigated in stochastic analysis for decades. Most of the early analysis in this field has been done under the assumption that $x$ is a semi-martingale (see e.g. \cite{BergMiz,Pr}), and high regularity of $k_2$ (i.e. non-singular cases). During the 1990's these equations received much attention from the perspective of white noise theory, see for example \cite{OksZha} for the case of linear equations with non-singular kernels $k_2$, and \cite{CoGeLeSo} for the case of linear equations with singular $k_2$.
\\
  
 A new direction in stochastic differential equations, called rough path theory, has been initiated in the late 1990's by Terry Lyons (see in particular \cite{LY}) .
   In contrast to white noise analysis, the theory of rough paths  gives a completely path-wise perspective on  differential equations driven by irregular signals. In fact Terry Lyons showed (see \cite{LyonsLevy}) that given an irregular H\"older continuous path $x$, the construction of a differential calculus with respect to $x$ relied mostly on the ability to define iterated integrals of $x$. In particular the solution to an ordinary differential equation of the form
   \begin{equation}\label{simple ODE ex}
   \dot{y}_t =\sigma(y_t)\dot{x}_t \qquad y_0=\xi,
   \end{equation}
is obtained as a continuous functional of the noise $x$, together with its iterated integrals and the initial data $\xi$. That is, if we let $\bx=(x,\bx^2,\ldots,\bx^n)$ for some $n\geq 1$ be the collection of the path $x$ together with its iterated integrals, then the solution $y$ can be viewed as $y_t=\mathcal{I}(\bx,\xi)_t$ where $\mathcal{I}$ is a Lipschitz continuous functional in both arguments.
Therefore the theory of rough paths not only opens up  the analysis of  stochastic differential equations to a vast new class of driving stochastic processes, but it also provides simple stability results with respect to that noise. 
  The cost of the improved analytical tractability of the solutions is that non-linear functions in the diffusion term (corresponding to  $\sigma$ in \eqref{simple ODE ex}) need to be better behaved than in It\^o's theory. Typically one requires coefficients which are at least $C^2$ and bounded in order to get existence and uniqueness of solutions for rough differential equations driven by Brownian motion. This is in contrast to the Lipschitz and linear growth assumptions well known from classical It\^{o} stochastic analysis. 
\smallskip  
  
In order to test the robustness of rough path theory, a natural endeavor has been to explore more general differential systems than the ordinary differential equation~\eqref{simple ODE ex}. One can think for example of delay equations \cite{NNT} and cases of stochastic PDEs \cite{DD,GT}, culminating in the theory of regularity structures \cite{Hairer}.
During the years 2009-2011,  A. Deya and the second author of this paper provided in \cite{DeyaTIndel1} and \cite{DeyaTIndel2} a rough path perspective on Volterra equations driven by irregular signals $x$. In particular they proved that existence and uniqueness hold whenever $k_2$ is sufficiently regular (i.e. non-singular) and the driving rough path is H\"older continuous with H\"older exponent greater than $1/3$.   Notice that in these papers, the authors also discussed the challenges of extending the theory of rough paths in order to include singular kernels in equation~\eqref{eq:motiv: Volterra equations}. 
This remained an open question until late 2018, when Pr\"omel and Trabs \cite{TrabsPromel} gave a para-controlled perspective on Volterra equations driven by irregular signals. Highly influenced by the theory of rough paths, the theory of para-controlled distributions developed by Gubinelli, Imkeller and Perkowski \cite{GubPerImk} gives a path wise perspective on SDEs and SPDEs through Paley-Littlewood para-controlled calculus, and Bony's para-product. 
Although the result of Pr\"omel and Trabs is very interesting in itself, it seems to be currently limited in the same way as for the theory of para-controlled calculus. Namely one has to assume that the regularity of the noise $x$, minus the order of the singularity of the kernel $k_2$ must be greater than $\frac{1}{3}$. Thus a full rough path "picture" in terms of Lyons' theory is not available at this time through the paracontrolled methodology. It should also be mentioned that paracontrolled distributions are mostly expressed through Fourier modes, which is usually not the natural way to handle nonlinear Volterra type equations.
\smallskip  

With the above preliminary considerations in mind, this article is devoted to a complete and comprehensive picture of the theory of rough paths in a Volterra setting with singular kernels. The main idea in order to achieve this goal is to extend the concept of a path $t\mapsto z_t$ to a two variable object $(t,\tau)\mapsto z_t ^\tau$ for $(t,\tau)\in \Delta_2$, where $\Delta_2$ is a simplex of two variables. This extension of the notion of path is motivated from the generic form of a Volterra integral 
\begin{equation}\label{Vp intro} 
z^\tau_t=\int_{0}^t k(\tau,r)dx_r,
\end{equation} 
for some (possibly singular) kernel $k$ and a H\"older continuous function $x$. Note that by considering the mapping $t\mapsto z^t_t$ we recover the classical well known Volterra integral. However, the main advantage with the splitting of the variables into one variable coming from the kernel and the other coming from the integration limit is the following: the regularity of the mapping $\tau\mapsto z^\tau_t$ is then completely determined by the regularity/singularity of the kernel $k$, while on the other hand the mapping $t\mapsto z^\tau_t$ is completely determined by the regularity/singularity of the driving noise. 
 While it is the composition of these regularities which yields the regularity of $t\mapsto z^t_t$, the separation of the two arguments allows us to give a framework for Volterra rough paths, similar to the classical rough path framework. More specifically, consider a two parameter $E$-valued path $z$ as defined in \eqref{Vp intro}.  Our main assumption will be the existence of a $n-$tuple of the form
 \begin{equation}\label{def of nstep sign}
\bz=(\bz^1,\bz^2,\ldots,\bz^n):\Delta_3 \rightarrow \bigotimes _{i=1}^n E^{\otimes i}
\end{equation} 
with $\bz^1=z$, and 
satisfying a modified Chen type relation 
\begin{equation}\label{che}
\bz_{ts}^{\tau}=\bz^{\tau}_{tu}\ast \bz^{\cdot}_{us}.
\end{equation}
Notice that in \eqref{def of nstep sign} the classical tensor product $\otimes$ used in rough path theory is replaced by a bilinear convolution operation  $\ast$. We will go back to this convolution product (which is one of our main ingredients) below. For the time being, let us just notice that it can be defined as a component-wise operation similarly to the classical tensor algebra, i.e.
\begin{equation}
\bz^{m,\tau}_{ts}=\sum_{i=0}^m \bz^{m-i,\tau}_{tu}\ast \bz^{i,\cdot}_{us}.
\end{equation}
With the Volterra structure for $(t,\tau)\mapsto z_t ^\tau$ and the proper definition of the convolution product $\ast$, we will argue that the solution to  a $V$-valued Volterra equation 
\begin{equation}\label{simple Vpeq}
y_t =\xi +\int_0^t k(t,r)\sigma(y_r)dx_r,\qquad \xi \in V
\end{equation} 
can be viewed as a continuous functional of the noise $\bz$ and the initial data $\xi \in V$. That is, the solution $y$ is given by $y=\mathcal{I}(\bz,\xi)$ where $\mathcal{I}$ is Lipschitz continuous in both arguments. It is worth noting that for $k\neq 1$, the element $\bz\in \bigotimes _{i=1}^n E^{\otimes i}$ given as in \eqref{def of nstep sign} is fundamentally different from the classical  iterated integrals in the theory of rough paths, both algebraically and analytically. 
\smallskip  

Let us go back to our first goal, namely the path-wise construction of the Volterra paths in \eqref{Vp intro} as well as the algebraic and analytical properties of the the associated Volterra-signature (as generalized from the concept of signatures in the theory of rough paths). We begin  to show that given an $\alpha-$H\"older continuous path $x$ and a singular kernel $k$ such that $|k(t,s)|\lesssim |t-s|^{-\gamma}$ and $\alpha-\gamma>0$, then the path $(t,\tau)\mapsto z^\tau_t$ is well defined and is contained in a space of two-variable Volterra-H\"older paths which will be specified later. Starting from this object, we will prove that the convolution product given in \eqref{che} is well defined for any two Volterra paths $z$ and $\tilde{z}$ built from Volterra kernels $k$ and $\tilde{k}$ and driving noise $x$ and $\tilde{x}$ respectively. In fact, intuitively one can think of this operation between $z$ and $\tilde{z}$ as 
\begin{equation}\label{eq:convol-intro}
z^{\tau}_{tu}\ast \tilde{z}^{\cdot}_{us}=\int_{u}^t dz^\tau_r \otimes \tilde{z}^r_{us},
\end{equation} 
where the increment $z^{\tau}_{tu}$ is defined by $z^{\tau}_{tu}=z^{\tau}_{t}-z^{\tau}_{u}$. In \eqref{eq:convol-intro}, note that the integration is done with respect to the upper parameter in $\tilde{z}$ (corresponding to a regularity coming from the kernel $\tilde{k}$) and the lower variable in $z$ (representing the regularity coming from the driving noise $x$).  
This operation will be extended to any two Volterra type objects in the $n-$tuple $\bz$, and leads naturally to the algebraic relation in \eqref{che}. Let us also mention at this stage that the H\"older type norm under consideration in this paper, taking into account both the regularity coming from the kernel $k$ and the noise $x$, will be given in the following way for the  component $\bz^i$ of $\bz$ (below we have $\al,\ga\in(0,1)$), 
\begin{equation}\label{eq:holder-z-i}
	|\bz_{ts}^{i,\tau}|\lesssim 
	|\tau-t|^{-\gamma}|t-s|^{\al},
\end{equation}
where we omit some of the other regularities to be considered for sake of clarity. As mentioned above, expression~\eqref{eq:holder-z-i} is thus separating a singularity of order $\ga$ on the diagonal $t=\tau$ from the $\al$-H\"older regularity in $t-s$. The object $\bz$ satisfying  \eqref{eq:convol-intro} and \eqref{eq:holder-z-i} is called a Volterra rough path.    In order to provide a full picture of the construction of these objects, we include in this article a generalization of the Sewing lemma \cite[Proposition 1]{Gubinelli}, as well as of the rough path extension theorem (see e.g.  \cite[Theorem 3.7]{LyonsLevy}) in the Volterra context.
\smallskip  

Once the construction of a Volterra rough path is secured, our second goal is concerned with the construction of solutions to \eqref{simple Vpeq}. To this end, we will extend the theory of controlled rough paths, as described by Gubinelli in \cite{Gubinelli}, to the Volterra-rough path setting. Observe that this extension also relies upon the convolution product $\ast$ introduced in \eqref{che}. In particular, a Volterra path $(t,\tau)\mapsto y^{\tau}_t$ controlled by the Volterra noise $(t,\tau)\mapsto z_t^\tau$ given as in \eqref{Vp intro} satisfies
\begin{equation}\label{eq:dcp-controlled-process}
y_{ts}^{\tau}=z^\tau_{ts}\ast y_s^{\prime,\tau,\cdot}+R_{ts}^\tau,
\end{equation}
where we recall the notation $y_{ts}^{\tau}=y_{t}^{\tau}-y_{s}^{\tau}$, and where $R_{ts}^\tau$ is a sufficiently regular remainder term. Processes of the form \eqref{eq:dcp-controlled-process} are the ones which can be naturally integrated with respect to $x$ in the rough Volterra sense. Furthermore, once a rough integral is defined for a large enough class of processes and one can prove the stability of the structure \eqref{eq:dcp-controlled-process} under composition with a nonlinear mapping, equations like \eqref{simple Vpeq} are solved thanks to a standard fixed point argument.
\smallskip

Let us now say a few words about the regularity of $x$ and the singularity of $k$ on the diagonal in equation \eqref{simple Vpeq}. 
We believe that, provided they can be pushed to arbitrary orders, expansions like \eqref{eq:dcp-controlled-process} yield a proper notion of Volterra rough type integral as long as $k$ is a singular kernel of order $-\gamma$ and $x$ is a $\alpha$-H\"older continuous noise with $\alpha-\gamma>0$.
 However, for sake of conciseness, this article is restricted to the case $\alpha-\gamma>\frac{1}{3}$. In this situation one only needs to assume the existence of the second step Volterra iterated integral $\bz^2$, and the first order controlled path structure \eqref{eq:dcp-controlled-process} is enough for our purposes. In the forthcoming article \cite{HTW} we investigate this problem in more details, and extend the theory presented in the current article to the case of $\alpha-\gamma>\frac{1}{4}$.  Rougher situations and more singular kernels are deferred to a further publication. For the construction of a Volterra rough path, one should also be aware of the fact that the concept of geometric rough paths is not directly transferable to the Volterra setting. 
 Simply put, if $k$ is a singular kernel one cannot expect to have a satisfying integration by parts formula (at least not in a classical sense) when integrating against $k$. Therefore the Volterra rough path $\bz$ defined by \eqref{Vp intro} and \eqref{def of nstep sign} is in general not a continuous function of the classical rough path above $x$. We thus expect some of the algebraic considerations related to the Volterra case to be different from the classical rough path theory, possibly requiring the  regularity structures techniques of \cite{Hairer}. The construction of Volterra rough paths above standard stochastic processes is deferred to a subsequent publication. 
\smallskip


\smallskip

Below we give a brief outline of the sections in the paper. 
\begin{enumerate}[wide, labelwidth=!, labelindent=0pt, label=\textnormal{(\roman*)}]
\setlength\itemsep{.1in}
\item Section \ref{sec:Basic-concepts-from} provides the elementary tools of rough path theory and fractional calculus needed in order to develop our framework in the sequel. 

\item Section \ref{sec:Volterra-Signatures} gives an introduction to the concept of Volterra iterated integrals and Volterra signatures in the case of smooth driving noise $x$, possibly involving a singular kernel $k$. In this section we will encounter the convolution product $\ast$ for the first time and give a detailed description of this product. We will also provide a working hypothesis on the regularity of the kernel $k$ which will be used throughout the rest of the text.

\item In Section  \ref{sec:Volterra-Rough-Paths} we move to the case when the driving noise $x$ of a Volterra path is only $\alpha$-H\"older continuous with $\alpha\in (0,1)$. We  construct a generalized space of Volterra-H\"older paths, and give a pathwise construction of the Volterra process given by \eqref{Vp intro} sitting in this Volterra-H\"older space. Furthermore, we prove that the convolution product is well defined for any Volterra path. This results in the definition of a  convolutional path (obtained as an extension of Lyon's concept of multiplicative paths) and then the creation of the Volterra signature from such paths. Both algebraic and analytic aspects of these objects are discussed. 

\item Section \ref{non-linear setting} deals with the extension of the rough path theory to the Volterra equations case, through the introduction of the Volterra signature and the convolution product defined in Section \ref{sec:Volterra-Rough-Paths}. To this end we define a class of Volterra controlled paths, and prove that the Volterra integral and the operation of composition with regular functions are continuous operations on this class of functions. This is then used in Section \ref{sec:rough volt eq} to show existence and uniqueness of Volterra integral equations on the form of \eqref{simple Vpeq} with singular kernel $k$ and rough driving noise $x$.

\end{enumerate}

\section{Preliminary notions\label{sec:Basic-concepts-from}}

This section is devoted to some preliminary notations and notions of classical rough paths, which will help to understand our considerations in the Volterra case. We start with some general notation in Section \ref{subsec:Notation}, and recall some notions of rough paths analysis in Section~\ref{subsec:classic-rough-paths}.

\subsection{General notation}\label{subsec:Notation}

We will frequently use Banach spaces $E,V$ and $H$, and write
$|\cdot|=\|\cdot\|_{E}$ as long as this does not leave
any confusion. Throughout we will write $a\lesssim b$ meaning that
there exists a constant $C>0$ such that $a\leq Cb$. We will denote by $\Delta_{n}([a,b])$ the $n$-simplex over $[a,b]$ defined by 
\begin{equation}\label{simplex}
\Delta_{n}([a,b])=\{(x_{1},\ldots   ,x_{n})\in [a,b]^{n} |\,\, a \leq x_{1}<\cdots   <x_{n}\leq b\}, 
\end{equation}
 and when the set $[a,b]$ is clear from context we will just write $\Delta_{n}$. 
 
The kernels involved in equations like \eqref{eq:motiv: Volterra equations} are closely related to fractional integral operators. We will mostly use the  operator $I^{\alpha}:L^{1}\left(\left[0,T\right];E\right)\rightarrow L^{1}\left(\left[0,T\right];E\right)$, which is defined for a given $\al>0$ and $(u,t)\in\Delta_{2}$ as follows:
\[
I_{u+}^{\alpha}\left(f\right)\left(t\right):=
\frac{1}{\Gamma(\al)}
\int_{u}^{t}\left(t-r\right)^{\alpha-1}f\left(r\right)dr,
\]
where $\Gamma$ denotes the Gamma function.
Fractional integrals have been widely studied in the literature, and we refer to \cite{SKM} for a thorough account on the topic. However, we  mention  here a few properties of the operators $I^{\alpha}$  which will be frequently used. Most important is the convolution property;  for $\al,\beta>0$ and $(u,t)\in\Delta_{2}$:
\begin{equation}\label{convolution_fractional_integral}
I_{u+}^{\alpha}\left(I_{u+}^{\beta}\left(f\right)\right)\left(t\right)=I_{u+}^{\alpha+\beta}\left(f\right)\left(t\right).
\end{equation}

We will also use the following action of $I_{u+}^{\al}$  on elementary functions
\begin{equation}\label{RL_on_one}
I_{u+}^{\alpha}(1)(t)=\frac{(t-u)^{\al}}{\Gamma(\al +1)},\,\,\text{ and }\,\,I_{u+}^{\alpha}((\cdot -u)^{\beta})(t)=\frac{(t-u)^{\al+\beta}\Gamma(\beta+1)}{\Gamma(\al+\beta+1)}, 
\end{equation}
where $\alpha,\beta>0$.  
Throughout the article we will rely on partitions of intervals. A partition over an interval $[a,b]$ will be denoted by  $\mathcal{P}[a,b]$. If the interval $[a,b]$ is clear from the context we may write $\mathcal{P}$.  
Throughout this article we will work with increments of functions, which for $(s,t)\in \Delta_2$  will be denoted by 
\begin{equation}\label{eq:notation-increments}
f_{ts}=f_{t}-f_{s}.
\end{equation}
We ask the reader to note that the order of $t$ and $s$ in $f_{ts}$ is changed from the traditional notation used in rough path theory. This is to accommodate the algebraic side of the  Volterra specific setting we will encounter in later sections.  
For $\alpha\in (0,1)$, we will denote by $\mathcal{C}^{\alpha}\left(I;E\right)$ the space of H\"older continuous functions from an interval $I$  to a Banach space $E$. If $I$  is reduced to a singleton $\{t\}$ then 
\begin{equation}\label{singleton space}
\mathcal{C}^{\alpha}\left(\{t\};E\right):
=\left\{f:V_{t} \rightarrow E \, | \,\, \sup_{s\in V_{t}}
 \frac{|f_{ts}|}{|t-s|^{\alpha}}<\infty\right\},
\end{equation} 
where $V_{t}$ stands for a neighbourhood of $\{t\}$. 
Furthermore, we will frequently use an operator $\delta$ well known in the theory  rough paths, given by  
\begin{equation}
\delta_{u}f_{ts}=f_{ts}-f_{tu}-f_{us}.
\end{equation}

\subsection{Short introduction to rough path theory}\label{subsec:classic-rough-paths}

In this section we recall some basic notions about signatures of paths and related geometric structures, which will make the generalization to Volterra type objects more natural.

\subsubsection{Signatures}
One natural way to introduce signatures of paths is to see how they arise from expansions of linear differential equations. Namely assume first  the path $x:\left[0,T\right]\rightarrow E$ is smooth, where $E$ is a given Banach space.
Let $V$ be another Banach space and consider the $V$-valued ODE 
\begin{equation}
\dot{y}_{t}=A\left(\dot{x}_{t}\right)y_{t},\,\,\,\,\,y_{0}=\xi\in V,\label{eq:smooth equation}
\end{equation}
where $A$ is a linear operator, namely $A\in\mathcal{L}\left(E,\mathcal{L}\left(V\right)\right)$. Whenever $x$ is smooth, a Picard type iteration yields the following expansion:
\begin{equation}
y_{t}=\xi\left(1+\sum_{i=1}^{\infty}A^{\circ i}\left(\int_{\Delta_i([0,t])}dx_{r_1}\otimes\cdots   \otimes dx_{r_{i}}\right)\right),\label{eq: picard integral smooth}
\end{equation}
where $A^{\circ i}$is the $i$-th composition of the linear operator $A$  which is given as a linear operator on $E^{\otimes i}$ defined
from the action 
\[
A^{\circ i}\left(x_{1}\otimes\cdots   \otimes x_{i}\right):=A\left(x_{1}\right)\circ\cdots   \circ A\left(x_{i}\right).
\]
The expansion \eqref{eq: picard integral smooth} reveals that $y$ can be seen as a continuous function of the collection $\{\int_{\Delta_i([0,t])}dx_{r_{1}}\otimes\cdots   \otimes dx_{r_{i}}; \, i\ge 1\}$, which is called the signature of $x$.

In order to describe the algebraic structures behind the expansion \eqref{eq: picard integral smooth}, let us first give some definitions.
\begin{defn}\label{def:truncated-algebra}
Let $E$ be a real Banach space. For $l\in\mathbb{N}$, the truncated algebra $T^{(l)}$ is defined by $T^{(l)}=\bigoplus_{n=0}^{l}E^{\otimes n}$,
with the convention $E^{\otimes 0}=\mathbb{R}$. The set $T^{(l)}$ is equipped with a straightforward
vector space structure, plus an operation $\otimes$ defined by
\beq\label{eq:def-product-in-T-l}
\lc g\otimes h\rc^{n}=\sum_{k=0}^{l}g^{n-k}\otimes h^{k},\qquad g,h\in
T^{(l)},
\eeq
where $g^{n}$ designates the projection on the $n$-th tensor level for $n\le l$. 
\end{defn}

\noindent
Notice that $T^{(l)}$ should be denoted $T^{(l)}(E)$. We have dropped the dependence on $E$ for notational sake. Also observe that with Definition \ref{def:truncated-algebra} in hand,
$(T^{(l)},+,\otimes)$ is an associative algebra with unit element $\mathbf{1} \in E^{\otimes 0}$. The polynomial terms in the expansions which will be considered later on are contained in a subspace of $T^{(l)}$ that we proceed to define now.

\begin{defn}\label{def:free-algebra}
The \textit{free nilpotent Lie algebra}
 $\mathfrak{g}^{(l)}$ of order $l$ is defined to be the graded sum 
 \begin{equation*}
\mathfrak{g}^{(l)} \overset{\Delta}{=} \bigoplus_{k=1}^{l}{\cl}_{k}\subseteq T^{(l)} .
 \end{equation*}
 Here $\mathcal{L}_{k}$ is the space of homogeneous Lie polynomials of degree $k$ given inductively by $\mathcal{L}_{1} \overset{\Delta}{=} E $ and $\mathcal{L}_k \overset{\Delta}{=} [ E ,\mathcal{L}_{k-1}]$, where the Lie bracket is defined to be the commutator of the tensor product. \end{defn}

We now define some groups related to the algebras given in Definitions \ref{def:truncated-algebra} and \ref{def:free-algebra}. To this aim, introduce the subspace $T_0^{(l)}\subseteq T^{(l)}$ of tensors whose scalar component is zero and recall that $\mathbf{1}\overset{\Delta}{=}(1,0,\cdots   ,0)$. For $u\in T_0^{(l)}$, one can define the inverse $(1+u)^{-1}$, the exponential $\exp(u)$ and the logarithm $\log(1+u)$ in $T^{(l)}$ by using the standard Taylor expansion formula with respect to the tensor product. For instance, 
\beq\label{eq:def-exp-on-T-l}
\exp({u})\overset{\Delta}{=}\sum_{k=0}^{\infty}\frac{1}{k!} \, {u}^{\otimes k}\in T^{(l)},
\eeq
where the sum is indeed locally finite and hence well-defined. We can now introduce the following group.
\begin{defn}\label{def:free-group}
The \textit{free nilpotent Lie group} $G^{(l)}$ of order $l$ is defined by 
$$G^{(l)}\overset{\Delta}{=}\exp(\mathfrak{g}^{(l)})\subseteq T^{(l)}.
$$ 
The exponential function is a diffeomorphism under which $\mathfrak{g}^{(l)}$ in Definition \ref{def:free-algebra} is the Lie algebra of $G^{(l)}$.
\end{defn}

As mentioned above, the link between free groups and differential equations like \eqref{eq:smooth equation} is made through the notion of signature. Namely a continuous map
$\mathbf{x}:\Delta_{2}\rightarrow T^{(l)}$ is called a
multiplicative functional if for $s<u<t$ one has $\mathbf{x}_{ts}
=\mathbf{x}_{tu}\otimes\mathbf{x}_{us}$, where $\otimes$ is the operation introduced in Definition \ref{def:truncated-algebra}. A particular occurrence of this kind of map is given when one considers a smooth path $w$ and sets for $(s,t)\in\Delta_{2}$,
\begin{equation}\label{eq:def-iterated-intg}
\mathbf{w}_{ts}^{n}=
 \int_{t>r_{n}>\cdots >r_{1}>s}dw_{r_{n}}\otimes\cdots   \otimes dw_{r_{1}} .
\end{equation}
Then the so-called \textit{signature} of $w$ is the following object:
\begin{equation}\label{eq:signature-smooth-x}
S_{l}(w): \Delta_2([0,1]) \rightarrow T^{(l)} ,
\qquad
(s,t)\mapsto
S_{l}(w)_{ts}:=1+\sum_{n=1}^{l}\mathbf{w}_{ts}^{n}.
\end{equation}
It is worth mentioning that $S_{l}(w)$ will be our typical example of multiplicative functional. In addition, signatures of paths belong to the group $G^{(l)}$ introduced in Definition~\ref{def:free-group} and in fact any element in $G^{(l)}$ can be written as the signature of a smooth path.

Another important property in the theory of signatures, originally
proved by Chen~\cite{Chen}, relates the multiplicative property
to the signature of the concatenation of two paths. That is, if $x:\left[0,s\right]\rightarrow E$
and $y:\left[s,t\right]\rightarrow E$ we can define their concatenation
$x\star y:\left[0,t\right]\rightarrow E$ by the mapping 
\begin{equation}\label{eq:def-concatenation}
[x\star y]_{r}
=
\begin{cases}
\begin{array}{cc}
x_{r} & r\in\left[0,s\right]\\
x_{s}+y_{rs} & r\in\left[s,t\right]
\end{array} & .\end{cases}
\end{equation}
Then if $S_{l}$ is the truncated signature of a path as described in \eqref{eq:signature-smooth-x}, we get the following relation, whose proof can be found e.g. in \cite[Theorem 2.9]{LyonsLevy}:
\begin{equation}\label{eq:concatenation-and-signature}
S_{l}\left(x\star y\right)=S_{l}\left(y\right)\otimes S_{l}\left(x\right).
\end{equation}

One can now go back to the the expansion \eqref{eq: picard integral smooth}, and realize that it can be expressed in terms of the signature of the path $x$. Whenever $x$ is smooth, the terms $\bx^{n}$ exhibit a factorial decay, which kill the possibly exponential growth from $A^{\otimes n}$. This fact is not obvious anymore in case of an irregular path $x$, which motivates the notion of rough path introduced below.

\subsubsection{Rough path lift of a H\"older path}

Let us now assume that the path $x$ driving \eqref{eq:smooth equation} is only $\al$-H\"older continuous with $\al\in(0,1)$. Then the iterated integrals appearing in the expansion in equation
\eqref{eq: picard integral smooth} are possibly not well defined. In
particular when the continuity of the driving signal is of order $\alpha\leq\frac{1}{2}$,
there is no canonical way of constructing such integrals. The seminal idea put forward by T. Lyons is that one can construct those iterated integrals by means of probabilistic tools, and then build a differential calculus with respect to $x$ starting from the iterated integrals. Those considerations motivate the introduction of H\"older continuous multiplicative functionals.

\begin{defn}
\label{def:classic rough path}Consider $\alpha\in\left(0,1\right)$
and let $n=\left\lfloor \frac{1}{\alpha}\right\rfloor $. Let $x\in\mathcal{C}^{\alpha}\left(\left[0,T\right];E\right)$
be a Hölder path and assume there exists an object $\mathbf{x}:\Delta_{2}\rightarrow G^{\left(n\right)}\left(E\right)$
defined through the mapping  
\[
\left(s,t\right)\mapsto\mathbf{x}_{ts}:=\left(1,\bx_{ts}^{1},\bx_{ts}^{2},\ldots,\bx_{ts}^{n}\right),
\]
where $\bx_{ts}^{1}:=x_{t}-x_{s}$ and where we recall that $G^{(n)}$ is introduced in Definition \ref{def:free-group}. In addition, we suppose that $\bx$ enjoys the following two properties:
\begin{equation}
\label{Multiplicative property}
\mathbf{x}_{tu}\otimes\mathbf{x}_{us}=\mathbf{x}_{ts}\,\,\,\,\,\text{(Multiplicative\,\,\,property)}
\end{equation}
 and 
\[
|\bx_{ts}^{i}|\leq\| x^{1}\|_{\alpha}^{i}\frac{|t-s|^{i\alpha}}{\Gamma(i\al+1)}
\,\,\,\,\,\text{for all } i\in\left\{ 1,\cdots   ,n\right\} \,\,\,\,\,\text{ (Analytic\,\,\,property)} .
\]
Here $\Gamma$ is the Gamma function. Then we call $\mathbf{x}$ a rough path above $x$ and we denote the space of
all $\al$-H\"older rough paths by $\mathscr{C}^{\alpha}\left(\left[0,T\right];E\right)$. Whenever the underlying domain and range of this space is clear from context, we simply write $\mathscr{C}^\alpha$. 
\end{defn}

Note that $\mathscr{C}^{\alpha}$ is not a vector space. Indeed, $\mathscr{C}^{\alpha}$ is not
a linear space due to the fact that $G^{(n)}$ is not a linear space. However, we can equip $\mathscr{C}^{\alpha}$ with the following metric:
\begin{equation}\label{eq:def-metric-rough}
d_{\alpha}\left(\mathbf{x},\mathbf{y}\right):=\sum_{i=1}^{n}\| \bx^{i}-\by^{i}\|_{i\alpha}.
\end{equation}
 One can also consider a subspace of this space called the space of
geometric rough paths and denoted by $\mathscr{C}_{g}^{\alpha}$, which
is defined as the closure of all smooth rough paths with respect to the metric $d_{\alpha}$ given by \eqref{eq:def-metric-rough}. Otherwise stated, 
$\mathbf{x}\in\mathscr{C}^{\alpha}$ is a geometric rough path if there exists a sequence of smooth paths $\left\{ \mathbf{x}^{n}\right\} :\Delta_{2}\rightarrow G^{\left(n\right)}\left(E\right)$
such that $d_{\al}(\mathbf{x}^{n},\mathbf{x})$ converges to 0.

The next theorem will give us a canonical extension of the rough path
from the truncated space $T^{\left(n\right)}\left(E\right)$ to all
the space $T\left(E\right).$ This extension is crucial in order to
ensure the existence and uniqueness of linear differential equations
controlled by irregular noise. The theorem and its proof can be found
in \cite[Theorem 3.7]{LyonsLevy}. 
\begin{thm}\label{thm:extension-rough-path}
Let $\mathbf{x}\in\mathscr{C}^{\alpha}$ be a rough path of order
$\alpha\in\left(0,1\right)$ and let $n=\left\lfloor \frac{1}{\alpha}\right\rfloor $.
Then there exists a unique extension of $\mathbf{x}$ to the space
$T\left(E\right)$ which satisfies the multiplicative and analytic
property. That is, for all $m\geq n+1$ there exists an object $\bx^{m}:\Delta_{2}\rightarrow E^{\otimes m}$
such that 
\[
\bx_{ts}^{m}=\sum_{i=0}^{m}\bx_{tu}^{m-i}\otimes \bx_{us}^{i},
\]
and for all $\left(s,t\right)\in\Delta_{2}$ we have
\[
|\bx_{ts}^{i}|\leq\| \mathbf{x}^{1}\|_{\alpha}^{i}  \frac{|t-s|^{i\alpha}}{\Gamma(i\al+1)}\,\,\,\,\forall i\geq1.
\]  
\end{thm}
Notice that Theorem \ref{thm:extension-rough-path} tells us that in order to construct the solution to
a rough differential equation in terms of its signature, we just need
to give a probabilistic construction of the first $n=\left\lfloor \frac{1}{\alpha}\right\rfloor $
iterated integrals. Then we know that all higher order
iterated integrals have a canonical (and deterministic) construction
only depending on the lower order integrals. We will try to reproduce this mechanism in the Volterra context.

\section{\label{sec:Volterra-Signatures}Volterra Signatures }

\subsection{\label{Def of Volterra signature}Definition and first properties}
In this section we will define precisely what we mean by a Volterra
signature over a smooth path. In this way the Volterra type integrals will be trivially defined and we can focus on their algebraic and analytic properties. This gives some insight on what can be expected in more irregular cases.
First we need to present an elementary inequality
we will use later (see e.g. \cite[Lemma 4.4]{DeyaTIndel1} for more details). 
\begin{lem}
\label{lem:Nice inequlaity }Let $\beta\in\left[0,1\right]$, $\gamma>0$,
and $0\leq r\leq q\leq\tau\leq T$. Then the following inequality
holds 
\[
|\left(\tau-r\right)^{-\gamma}-\left(q-r\right)^{-\gamma}|\leq\left(\tau-q\right)^{\beta}\left(q-r\right)^{-\gamma-\beta}.
\]

\end{lem}

Our constructions will rely on specific assumptions about the power type singularity of the kernel $k$ appearing in \eqref{eq:motiv: Volterra equations}. Inspired by Lemma \ref{lem:Nice inequlaity },  the main hypothesis we shall use can be summarized as follows.
\begin{description}[wide, labelwidth=!, labelindent=0pt]
\item [H] 
Let $k$ be a kernel $k:\Delta_{2}\rightarrow\mathbb{R}$.
We assume that there exists $\gamma\in (0,1)$ such that for all $\left(s,r,q,\tau\right)\in\Delta_{4}\left(\left[0,T\right]\right)$ and $\eta,\beta\in [0,1]$
 we have
\begin{align}
|k\left(\tau,r\right)| & \lesssim|\tau-r|^{-\gamma}\\\label{2nd}
|k\left(\tau,r\right)-k\left(q,r\right)| & \lesssim |q-r|^{-\gamma-\eta}|\tau-q|^{\eta}
\\\label{3rd}
|k\left(\tau,r\right)-k\left(\tau,s\right)| & \lesssim |\tau -r|^{-\gamma-\eta}|r-s|^{\eta} \\\label{4th}
|k\left(\tau,r\right)-k\left(q,r\right)-k\left(\tau,s\right)+k\left(q,s\right)|&\lesssim |q-r|^{-\gamma-\beta}|r-s|^{\beta} 
\\\label{5th}
|k\left(\tau,r\right)-k\left(q,r\right)-k\left(\tau,s\right)+k\left(q,s\right)| &\lesssim 
|q-r|^{-\gamma-\eta}|\tau-q|^{\eta}.
\end{align} 
Here all the inequalities $\lesssim$ are independent of the parameters $\gamma,\beta$ and $\eta$.
In the sequel a kernel fulfilling condition \textbf{(H)} will be called Volterra kernel of order $-\gamma$. 
\end{description}

\begin{rem}\label{rem7}
If a kernel $k$ satisfies $(\bf{H})$ then by the interpolation inequality $a\wedge b \leq a^{\theta}b^{1-\theta}$ for any $\theta\in [0,1]$ applied to the minimum of (\ref{4th}) and (\ref{5th}) it follows that for any $\beta,\eta\in\left[0,1\right]$ we have 
\begin{equation}\label{bound H}
|k\left(\tau,r\right)-k\left(q,r\right)-k\left(\tau,s\right)+k\left(q,s\right)| \lesssim|\tau-q|^{\eta}|q-r|^{-\beta-\gamma-\eta}|r-s|^{\beta}.
\end{equation} 
\end{rem}

With those assumptions in hand we can now introduce the notion of iterated Volterra integral and Volterra signature, which parallel \eqref{eq:def-iterated-intg} and \eqref{eq:signature-smooth-x}.

\begin{defn}
Let us consider a path $x\in C^{1}\left(\left[0,T\right];E\right)$ and a Volterra kernel $k:\Delta_{2}\rightarrow\mathbb{R}$  satisfying $\mathbf{\left(H\right)}$. The \emph{iterated Volterra integral} of order $n$ is a mapping $\bz^{n}:\Delta_{3}\rightarrow E^{\otimes n}$
given by 
\begin{equation}
\left(s,t,\tau\right)\mapsto \bz_{ts}^{n,\tau}=\int_{t>r_{n}>\cdots   >r_{1}>s}k(\tau,r_{n})\bigotimes_{j=1}^{n-1} k\left(r_{j+1},r_{j}\right)dx_{r_j}.\label{eq:Volterra iterated integrals}
\end{equation}
We also consider the collection
of iterated Volterra integrals as an element of the free algebra. Specifically, we define the element $\bz_{ts}^{\tau}\in T^{(\infty)}(E)$ as follows:
\[
\bz_{ts}^{\tau}=\left(1,\bz_{ts}^{1,\tau},\ldots,\bz_{ts}^{n,\tau},\ldots\right) ,
\]
where we recall that the spaces $T^{(\infty)}(E)$ are introduced in Definition \ref{def:free-group}.
\end{defn}

\begin{rem}
As already highlighted in the introduction, notice that in the definition~\eqref{eq:Volterra iterated integrals} the variable $\tau$ is considered as an additional parameter indexing $z$. While we might be mostly interested in the case $\tau=t$, this extra freedom will play an essential role in our considerations.
\end{rem}
\begin{rem}
Observe that  the Volterra integrals are denoted by $(s,t,\tau)\mapsto\bz_{ts}^{n,\tau}$ as opposed to $(s,t)\mapsto\bz_{st}^{n}$   in the regular rough path setting. This small modification will ease our notation when one has to deal with integrals of the form
$\int_{0}^{t}k(t,r)f_{r}dx_r$.
\end{rem}
\begin{rem}
A particularly important note is that the collection of Volterra iterated integrals $\bz=(1,\bz^1,\ldots)$ is not contained in the free nilpotent Lie group of  $G$ given in Definition~\ref{def:free-group}. We expect that one needs a different algebraic approach to these integrals due to the kernels $k$ involved in the integrals. Especially in the singular case it is quite intuitive that Volterra iterated integrals do not lie in the free nilpotent lie group, as there is no concept of integration by parts. That is, let $x^i$ and $x^j$ be two real valued smooth paths, and consider the second level $\bz^2$. Then observe that a simple integration by parts would yield 
\begin{multline}
\int_{t>r>u>s}k(t,r)k(r,u)dx^i_udx^j_r = 
\int_{s}^{t} k(t,r) dx_{r}^{i} \, \int_{s}^{t} k(t,r) dx_{r}^{j}
\\ - \int_{t>r>u>s}k(t,r)k(r,r)dx^j_udx^i_r-\int_{t>r>u>s}k(t,r)\frac{d}{dr}k(r,u)dx^j_udx^i_r.
\end{multline}
However, since $k$ is singular we have $k(r,r)=\infty$, and the derivative $\frac{d}{dr}k(r,u)$ would no longer be integrable. This additional singularity prevents us to exhibit a bracket  defined as the commutator of the tensor product in Definition~\ref{def:free-algebra} (here considered for the second level term).  Therefore a deeper investigation into the algebraic properties of the Volterra iterated integrals given in \eqref{eq:Volterra iterated integrals} would be highly interesting, and we hope to tell more on this aspect in the future. 
\end{rem}

When $x$ is a smooth  function, iterated Volterra integrals enjoy a regularity property which is similar to the analytic property in Definition \ref{def:classic rough path}. This is labelled in the following proposition.

\begin{prop}\label{prop:gamma bound smooth functionals}
Let $k:\Delta_{2}(\left[0,T\right])\rightarrow\mathbb{R}$ be a Volterra
kernel which satisfies $\mathbf{\left(H\right)}$ with $\gamma<1$,
and assume $x$ is a continuously differentiable function. For $n\geq1$, consider the path $\bz^{n,\tau}$ defined by \eqref{eq:Volterra iterated integrals}.
Then for $(s,t)\in\Delta_{2}(\left[0,T\right])$ we have that
\[
|\bz_{ts}^{n,\tau}|\leq
\frac{\big( \| x\|_{\cac^{1}}\Gamma(1-\gamma) \big)^{n}}{\Gamma\left(n\left(1-\gamma\right)\right)}
\left(\tau-s\right)^{-\gamma}\left(t-s\right)^{\left(n-1\right)\left(1-\gamma\right)+1}
,
\]
where the $\cac^{1}$ norm of $x$ is defined by $\| x\|_{\cac^{1}}:=\sup_{t\in\left[0,T\right]}\left(|x_{t}|+|\dot{x}_{t}|\right)$. 
\end{prop}

\begin{proof}
Starting from Definition (\ref{eq:Volterra iterated integrals}) and invoking the fact that $x$ is a $\cac^{1}$ function, we directly get
\begin{align*}
|\bz_{ts}^{n,\tau}|&=\left| \int_{t>r_{n}>\cdots   >r_{1}>s}k(\tau,r_{n})\bigotimes_{j=1}^{n-1} k\left(r_{j+1},r_{j}\right)dx_{r_j}\right|
\\
&\leq\int_{t>r_{n}>\cdots   >r_{1}>s}|k\left(\tau,r_{n}\right)|\prod_{i=1}^{n-1}|k\left(r_{i+1},r_{i}\right)||\dot{x}_{r_{1}}|\otimes\cdots   \otimes|\dot{x}_{r_{n}}|dr_{1}\cdots   dr_{n}.
\end{align*}
Therefore hypothesis $(\mathbf{H})$ on the kernel $k$ entails
\begin{align*}
|\bz_{ts}^{n,\tau}|&
\leq\| x\|_{C^{1}}^{n}\int_{t>r_{n}>\cdots   >r_{1}>s}\left(\tau-r_{n}\right)^{-\gamma}\prod_{i=1}^{n-1}\left(r_{i+1}-r_{i}\right)^{-\gamma}dr_{1}\cdots   dr_{n}\\
&=\| x\|_{C^{1}}^{n}\Gamma\left(1-\gamma\right)^{n-1}\int_{s}^{t}\left(\tau-r\right)^{-\gamma}I_{s+}^{\left(n-1\right)\left(1-\gamma\right)}\left(1\right)\left(r\right)dr,
\end{align*}
where we have used the convolution property (\ref{convolution_fractional_integral}) of the Riemann-Liouville
integral operator $I^{\alpha}$ described in Section \ref{subsec:Notation}. Furthermore, it follows from the identities in Equation~\eqref{RL_on_one} that
\begin{align}\label{a2}
&\int_{s}^{t}\left(\tau-r\right)^{-\gamma}I_{s+}^{\left(n-1\right)\left(1-\gamma\right)}\left(1\right)\left(r\right)dr=
c_{n,\ga}\int_{s}^{t}\left(\tau-r\right)^{-\gamma}\left(r-s\right)^{\left(n-1\right)\left(1-\gamma\right)}dr
\notag \\
&
=c_{n,\ga}\left(t-s\right)^{\left(n-1\right)\left(1-\gamma\right)+1}\left(\tau-s\right)^{-\gamma}\int_{0}^{1}\left(1-\theta\frac{t-s}{\tau-s}\right)^{-\gamma}\theta^{\left(n-1\right)\left(1-\gamma\right)-1}d\theta, 
\end{align}
where we have used the notation $c_{n,\ga}=[\Gamma((n-1)(1-\gamma)+1)]^{-1}$ and the substitution $r=s+\theta\left(t-s\right)$.
In addition, since $\tau\geq t,$ it is clear that 
\begin{equation}\label{a3}
\int_{0}^{1}\left(1-\theta\frac{t-s}{\tau-s}\right)^{-\gamma}\theta^{\left(n-1\right)\left(1-\gamma\right)-1}d\theta\leq B\left(1-\gamma,\left(n-1\right)\left(1-\gamma\right)\right),
\end{equation}
where $B$ is the Beta function. Observe that classical identities for Gamma and
Beta functions, yields that 
\begin{equation}\label{a4}
\frac{B\left(1-\gamma,\left(n-1\right)\left(1-\gamma\right)\right)}{\Gamma\left(\left(n-1\right)\left(1-\gamma\right)\right)}=\frac{\Gamma(1-\gamma)}{\Gamma((n-1)(1-\gamma)+1)}
\end{equation}
plugging relation (\ref{a4}) into (\ref{a3}) and then (\ref{a2}) we have then obtained 
\begin{equation}\nonumber
\int_{s}^{t}\left(\tau-r\right)^{-\gamma}I_{s+}^{\left(n-1\right)\left(1-\gamma\right)}\left(1\right)\left(r\right)dr
\leq \frac{\Gamma(1-\gamma)}{\Gamma\left((n-1)(1-\gamma)+1\right)}(\tau-s)^{-\gamma}(t-s)^{(n-1)(1-\gamma)+1},
\end{equation}
which is our claim.
\end{proof}

\subsection{\label{subsec:convolution product}Convolution product}
We will now try to get an equivalent to the multiplicative property of the signature (\ref{Multiplicative property}) in a Volterra context. Unfortunately this property does not hold directly for a Volterra rough path, due to the interaction between variables in the kernel $k$. However, we will show that if
we modify the tensor product to be a type of convolution product,
then we still get a concatenation type property under this product. 
\begin{prop}
\label{prop:convolutional Chens relation smooth} Let $(s,u,t)\in \Delta_{3}$. Consider two $C^{1}$ functions  $x:\left[s,u\right]\rightarrow E$
and $y:\left[u,t\right]\rightarrow E$, and denote by $q=x\star y$
their concatenation. Let $\bz^{n}$ be the $n$-th Volterra integral of $q$  on $(s,t)$ as defined in (\ref{eq:Volterra iterated integrals}), namely for all $(s,t,\tau)\in \Delta_{3}$ set
\[
\bz_{ts}^{n,\tau}:=\int_{t>r_{n}>\cdots   >r_{1}>s}\bigotimes_{j=n}^{1} k\left(r_{j+1},r_{j}\right)dq_{r_j},
\]
with the convention that $r_{n+1}=\tau$. Then for $(s,u,t,\tau)\in\Delta_{4}$ we have 
\begin{equation}\label{con con}
\bz_{ts}^{n,\tau}=\sum_{i=0}^{n}\bz_{tu}^{n-i,\tau}\ast \bz_{us}^{i,\cdot},
\end{equation}
where the convolution product $\ast$ is defined as follows for all $0\leq i\leq n$
 \begin{align}\label{con prod}
&\bz_{tu}^{n-i,\tau}\ast \bz_{us}^{i,\cdot}\\
&:=\int_{t>r_{n}>\cdots   >r_{i+1}>u}\bigotimes_{j=n}^{i+1} k\left(r_{j+1},r_{j}\right)dy_{r_j}
\otimes \int_{u>r_{i}>\cdots >r_{1}>s}k\left(r_{i+1},r_{i}\right)\bigotimes_{j=i-1}^{1} k\left(r_{j+1},r_{j}\right)dx_{r_j}. \notag
\end{align}
Here we have used the convention $\bz^0\equiv 1$ and $\bz^n\ast1=1\ast\bz^n=\bz^n$.
\end{prop}

\begin{proof}
This proof is left to the patient reader. The result is easily checked by splitting the domain $$\Delta_{n}([s,t]) =\{(r_{1},\cdots   ,r_{n})\in[s,t] \,| \ t>r_{n}>\cdots   >r_{1}>s\}$$ 
into sub-domains 
$$\Delta_{n,j}=\{(r_{1},\cdots   ,r_{n})\in[s,t]\, | \ t>\cdots   >r_{j+1}>u>r_{j}>\cdots   >s\}.$$
\end{proof}
\begin{rem}\label{rem14}
In order to make formula (\ref{con prod}) more concrete, let us explicitly compute the integrals we obtain for $n=2$. In this case relation (\ref{con con}) reads 
\begin{equation}\label{relation 320}
\bz_{ts}^{2,\tau}=\bz_{tu}^{2,\tau}+\bz_{tu}^{1,\tau}\ast\bz_{us}^{1,\cdot}+\bz_{us}^{2,\tau}, 
\end{equation}
and we observe that 
\begin{equation}\label{eq:sim}
\bz_{tu}^{1,\tau}\ast\bz_{us}^{1,\cdot}=\int_{t>r_{2}>u}k(\tau,r_{2})dx_{r_{2}}\otimes\int_{u>r_{1}>s}k(r_{2},r_{1})dx_{r_{1}},
\end{equation}
where we note the common integration variable $r_{2}$ in  the above product.  In relation~(\ref{con prod}), we also notice  that since the kernel $k$  is smooth  except on the diagonal, the function 
\begin{equation}\nonumber
l_{r_{2}}=\int_{u>r_{1}>s}k(r_{2},r_{1})dx_{r_{1}} 
\end{equation}
inherits the smoothness of $k$. Therefore the integral
\begin{equation}\nonumber
\int_{t>r_{2}>u}k(\tau,r_{2})dx_{r_{2}}\otimes l_{r_{2}},
\end{equation}
which features in (\ref{eq:sim}), can be interpreted as a Riemann-Stieltjes integral. One of our main task will then be to control possible singularities arising from $k$ when $x$  is no longer assumed to be smooth, but rather a Hölder  path. We refer to Section \ref{sec:Volterra-Rough-Paths}  for a further analysis of this point. 
\end{rem}

Next we will present a technical lemma which will become useful in later analysis of the Volterra signature. 
It states that the convolution product $\ast$ behaves similarly to the tensor product $\otimes$ on small scales. 
\begin{lem}
Let $\mathcal{P}$ be a partition of $\left[s,t\right]$ such that
$|\mathcal{P}|\rightarrow0$, and consider $\bz^{j}$ for $j=1,\ldots,p$
as constructed in Equation (\ref{eq:Volterra iterated integrals})
with a continuously differentiable driving noise and a kernel $k$
satisfying $\left(\mathbf{H}\right)$ with singularity of order $\gamma<\frac{1}{2}$.
Then for $n,p\geq1$ with $p-n\geq1,$ we have 
\begin{equation}\label{b1}
\lim_{|\mathcal{P}|\rightarrow0} \left| \sum_{\left[u,v\right]\in\mathcal{P}}
\bz_{vu}^{p-n,\tau}\ast \bz_{us}^{n,\cdot}
-\bz_{vu}^{p-n,\tau}\otimes \bz_{us}^{n,u}\right| =0,
\end{equation}
\end{lem}

\begin{proof}
In order to study the left hand side of (\ref{b1}), let us set for $(u,v)\in \Delta_{2}$  
\begin{equation}\label{c2}
D\left(u,v\right)=\bz_{vu}^{p-n,\tau}\ast \bz_{us}^{n,\cdot}-\bz_{vu}^{p-n,\tau}\otimes \bz_{us}^{n,u}
\end{equation}
Then according to Definition (\ref{con prod}) it is readily checked that 
\begin{multline}\nonumber
D\left(u,v\right)=\int_{v>r_{p}>\cdots   >r_{n+1}>u}\,\bigotimes_{i=p}^{n+1} k\left(r_{i+1},r_{i}\right)dx_{r_i} 
\\
 \otimes \int_{u>r_{n}>\cdots>r_{1}>s} \Big[ k\left(r_{n+1},r_{n}\right)-k\left(u,r_{n}\right)\Big] dx_{r_{n}}\bigotimes_{i=n-1}^{1}k\left(r_{i+1},r_{i}\right)dx_{r_{i}},
\end{multline}
where we have written $r_{n+1}=\tau$ for the sake of readability.
We now proceed along the same lines as for Proposition \ref{prop:gamma bound smooth functionals}. Namely if we assume that  $\|\dot{x}\|_{\infty}\leq M$, we get
\begin{multline*}
\left| D\left(u,v\right)\right| \leq M^{p}\int_{v>r_{p}>\cdots >r_{n+1}>u}\prod_{i=n+1}^{p}|k\left(r_{i+1},r_{i}\right)|
\\
\times\left(\int_{u>r_{n}>\cdots >r_{1}>s}|k\left(r_{n+1},r_{n}\right)-k\left(u,r_{n}\right)|\prod_{i=1}^{n-1}|k\left(r_{i+1},r_{i}\right)|dr_{n}\dots dr_{1}\right)dr_{p}\ldots dr_{n+1}.
\end{multline*}
Furthermore, from $\left(\mathbf{H}\right)$ we have $|k\left(r_{i+1},r_{i}\right)|\lesssim|r_{i+1}-r_{i}|^{-\gamma},$
and any $\beta\in [0,1]$ we have
\[
|k\left(r_{n+1},r_{n}\right)-k\left(u,r_{n}\right)|\lesssim |r_{n+1}-u|^{\beta}|u-r_{n}|^{-\gamma-\beta}.
\]
 Thus restricting  $\beta\in\left(0,1-\gamma\right),$ we
get 
\begin{multline}\nonumber
|D\left(u,v\right)|\leq M^{p}\int_{v>r_{p}>\cdots >r_{n+1}>u}\left(\tau-r_{n}\right)^{-\gamma}\prod_{i=n+1}^{p-1}|r_{i+1}-r_{i}|^{-\gamma}|r_{n+1}-u|^{\beta}dr_{n+1}\cdots dr_{p}\\
\times \int_{u>r_{n}>\cdots >r_{1}>s}|u-r_{n}|^{-\gamma-\beta}\prod_{i=1}^{n-1}|r_{i+1}-r_{i}|^{-\gamma}dr_{1}\cdots dr_{n}.
\end{multline}
Hence, integrating the outside integral over the simplex $u>r_{n}>\cdots>r_{1}>s$, we end up with
\begin{align}
&|D\left(u,v\right)| \notag\\
&\leq C_{\beta,\gamma,p,n}\int_{v>r_{p}>\cdots   >r_{n+1}>u}\prod_{i=n+1}^{p}|r_{i+1}-r_{i}|^{-\gamma}\left(r_{n+1}-u\right)^{\beta}dr_{n+1}\ldots dr_{p}\times (u-s)^{n(1-\gamma)-\beta}\notag\\
&\leq C_{\beta,\gamma,p,n}\int_{u}^{v}\left(\tau-r\right)^{-\gamma}(r-u)^{\left(p-n-2\right)\left(1-\gamma\right)+\left(\beta+1\right)}dr\times (u-s)^{n(1-\gamma)-\beta},\label{a5}
\end{align}
where $C_{\beta,\gamma,p,n}:=\frac{\Gamma(1-\gamma-\beta)\Gamma(1-\gamma)^{n-1}M^{p}}{\Gamma(n(1-\gamma)-\beta+1)}$ and where we have used the convolution property~(\ref{convolution_fractional_integral}) of the Riemann-Liouville fractional integral. Now we can do a change of variables $r=u+\theta\left(v-u\right)$ and
find
\[
\int_{u}^{v}\left(\tau-r\right)^{-\gamma}\left(r-u\right)^{\left(p-n-2\right)\left(1-\gamma\right)+\beta+1}dr
=\left(\tau-u\right)^{-\gamma}\left(v-u\right)^{\left(p-n-2\right)\left(1-\gamma\right)+\beta+2}c_{\gamma,\tau,u,v},
\]
where $c_{\gamma,\tau,u,v}$ is a function bounded by the Beta function, i.e.
\begin{eqnarray*}
c_{\gamma,\tau,u,v}&=&
\int_{0}^{1}\left(1+\theta\frac{v-u}{\tau-u}\right)^{-\gamma}\theta^{\left(p-n-2\right)\left(1-\gamma\right)+\beta+1}d\theta
\\
&\leq& 
B(1-\gamma,(p-n-2)(1-\gamma)+\beta+2)<\infty.
\end{eqnarray*}
Plugging this identity into (\ref{a5}) and writing $C=C_{\beta,\gamma,p,n}$ for constants which may change from line to line, we get
\[
|D\left(u,v\right)|\leq C\left(\tau-u\right)^{-\gamma}\left(v-u\right)^{\left(p-n-2\right)\left(1-\gamma\right)+\beta+2}(u-s)^{n\left(1-\gamma\right)-\beta}
\]
Therefore it is readily checked that 
\[\sum_{[u,v]\in \mathcal{P}} |D(u,v)|\leq C |\mathcal{P}|^{\left(p-n-2\right)\left(1-\gamma\right)+\beta+1}\times\int_{s}^{t}\left(\tau-u\right)^{-\gamma}(u-s)^{n\left(1-\gamma\right)-\beta}du,
\]
where $|\mathcal{P}|$ denotes the size of the mesh of $\mathcal{P}$. Taking into account the definition (\ref{c2}) of $D(u,v)$, this finishes the proof.
\end{proof}

\section{\label{sec:Volterra-Rough-Paths}Volterra Rough Paths}

To begin the study of Volterra rough paths, we need to understand the
structure and regularity which may be extracted from a Volterra path.
As we have already seen, a Volterra path is really a two parameters
function on a simplex $\Delta_{2}$ taking values
in some space $E$. A simple example of a function of this form could
be the singular kernel
\begin{equation}\label{f1}
f_{t}^{\tau}:=\left(\tau-t\right)^{-\ga}, 
\end{equation}
defined for $t\leq \tau$ and $\ga \in (0,1)$. Note that for a function $f$ given as in (\ref{f1}) and $(s,t,\tau)\in\Delta_{3}$  it follows from Lemma \ref{lem:Nice inequlaity } that 
\[
|f_{ts}^{\tau}|\leq\left(\tau-t\right)^{-(\ga+1)}\left(t-s\right).
\]
This tells us that as long as $t<\tau$ then we have a Lipschitz
bound on $f^{\tau}$, i.e. for any $\epsilon>0$ we have $f^{\tau}\in C_{\rm{Lip}}\left(\left[0,\tau-\epsilon\right]\right).$
Similarly one can consider the function
\[
g_{t}^{\tau}=\left(\tau-t\right)^{\alpha},
\]
for some $\alpha\in\left(0,1\right)$ and 
$t\leq \tau$. Then it is easy to see
that globally, $g$ is $\alpha$-H\"older continuous in both variables.
However, for any small $\epsilon>0$ we have that $t\mapsto g_{t}^{\tau}$
is $C^{\infty}\left(\left[0,\tau-\epsilon\right]\right).$ Along the same lines,
one can see that $\tau\mapsto g_{t}^{\tau}\in C^{\infty}\left(\left[t+\epsilon,T\right]\right)$.
In the sequel we will generalize the above considerations  to processes of the form 
\begin{equation}\label{f2}
z_{ts}=\int_{s}^{t} k(t,r)dx_{r},
\end{equation}
where $x$ is an $\alpha$-H\"older path and $k$ a possibly singular kernel of order $-\ga$. 
This section is devoted to a definition and analysis of generic Volterra type rough paths like in (\ref{f2}).  
\subsection{Definition and sewing lemma}

	Let us go back for a moment to the increment defined in (\ref{f2}). One way to define the term $\int_{s}^{t}k(\tau,r)dx_{r}$ is to split the integral in the right hand side of (\ref{f2}) along a partition $\mathcal{P}$ of $[s,t]$, 
\[
\int_{s}^{t}k\left(t,r\right)dx_{r}=\sum_{\left[u,v\right]\in\mathcal{P}\left[s,t\right]}\int_{u}^{v}k\left(t,r\right)dx_{r},
\]
Then for each $[u,v]\in \mathcal{P}$ we have some regularity of $\int_{u}^{v}k\left(t,r\right)dx_{r}$
coming from the difference $v-u$ which is contributed by the driving
noise, and some (possibly singular) regularity coming from the difference
$t-v$. Much of the difficulty in the analysis of Volterra rough paths
will be due to such considerations. 
In order to capture the different regularities discussed above, we will make use of three different quantities, which will later be used in the definition of various classes of Volterra H\"older functions. For two parameters $(\alpha,\gamma)\in (0,1)^2$, we set $\rho=\alpha-\gamma$, and we will consider 
the semi-norms defined by
\begin{align}\label{norm1}
\| z\|_{\left(\alpha,\gamma\right),1}:= & \sup_{\left(s,t,\tau\right)\in\Delta_{3}}\frac{|z_{ts}^{\tau}|}{|\tau-t|^{-\gamma}|t-s|^{\alpha}\wedge |\tau-s|^{\rho}}
\\
\| z\|_{\left(\alpha,\gamma\right),1,2}:= & \sup_{\substack{\left(s,t,\tau^{\prime},\tau\right)\in\Delta_{4}\\\eta\in [0,1],\zeta \in [0,\rho)}}\frac{|z_{ts}^{\tau\tau^{\prime}}|}{|\tau-\tau^{\prime}|^{\eta}|\tau^{\prime}-t|^{-\eta+\zeta}\left(|\tau^{\prime}-t|^{-\gamma-\zeta}|t-s|^{\alpha} \wedge |\tau^{\prime}-s|^{\rho-\zeta}\right)} ,
\label{eq: two variable  function norm change}
\end{align}
with the convention $z_{ts}^{\tau}=z_{t}^{\tau}-z_{s}^{\tau}$
and $z_{s}^{\tau\tau^{\prime}}=z_{s}^{\tau}-z_{s}^{\tau^{\prime}}$
With these quantities at hand, let us define a space of functions  which we will call Volterra paths.
 
\begin{defn}\label{holder norms}
Let $\left(\alpha,\gamma\right)\in\left(0,1\right)^{2}$ and consider a function  $z:\Delta_{2}\rightarrow E,$ such that $(t,\tau)\mapsto z_{t}^{\tau}$. We assume that for all $\tau\in\left[0,T\right]$ we have 
\[
t\mapsto z_{t}^{\tau}\in\mathcal{C}^{\alpha-\gamma}\left(\left\{ \tau\right\} \right)\cap\mathcal{C}^{\alpha}\left(\left[0,\tau\right)\right),
\]
where we recall that the notation $\mathcal{C}^{\alpha-\gamma}\left(\left\{ \tau \right\} \right)$ has been introduced in (\ref{singleton space}). We also assume that for all $t\in\left[0,T\right]$ the following holds: 
\[
\tau\mapsto z_{t}^{\tau}\in\mathcal{C}^{\alpha-\gamma}\left(\left\{ t\right\} \right)\cap\mathcal{C}^{1}\left(\left(t,T\right]\right).
\]
Then for such a function $z$, define 
\begin{equation}\label{norm2}
\| z\|_{\left(\alpha,\gamma\right)}:=\| z\|_{\left(\alpha,\gamma\right),1}+\| z\|_{\left(\alpha,\gamma\right),1,2},
\end{equation}
where the norms are given as in \eqref{norm1} and \eqref{eq: two variable  function norm change}.  We define the space
of Volterra paths $z:\Delta_{2}\rightarrow E$ as
all paths such that $z_{0}^{\tau}=z_{0}\in E$ for all $\tau\in(0,T]$,
and 
\[
\| z\|_{\left(\alpha,\gamma\right)}<\infty.
\]
We denote this space by $\mathcal{V}^{\left(\alpha,\gamma\right)}\left(\Delta_{2};E\right)$.
In addition, under the mapping 
\[
z\mapsto|z_{0}|+\| z\|_{\left(\alpha,\gamma\right)},
\]
 the space $\mathcal{V}^{\left(\alpha,\gamma\right)}$ is a Banach
space. 
\end{defn}
\begin{rem}
Conventionally, we will use the notation $y_{ts}^{\tau}$ to signify
both functions with three arguments, and the increment of functions
with two arguments, i.e. $y_{ts}^{\tau}=y^{\tau}(s,t)$
and $y_{ts}^{\tau}=y_{t}^{\tau}-y_{s}^{\tau}.$ We hope the specific meaning will always be clear from the context.
Moreover, we will use the same norms as those defined in (\ref{norm2}) for three variable functions $y:\Delta_{3}\rightarrow E$ given by  $(s,t,\tau)\mapsto y_{ts}^{\tau}$. 
\end{rem}

\begin{rem}
The space $\mathcal{V}^{\left(\alpha,\gamma\right)}$ really captures
three different regularities in different areas of $\Delta_{2}\left(\left[0,T\right]\right)$. On
the diagonal line, $\left(t,t\right)$ we clearly have that $z\in \mathcal{V}^{(\alpha,\gamma)}$ is of
$\rho$-Hölder regularity in both variables, where $\rho=\ga-\al$. However, at any point
off the diagonal we have $\alpha$-regularity in the lower variable
and $1$-regularity in the upper variable. The space could have therefore
be defined more generally to capture three different regularities.
However, for our purposes, under the assumption $\left(\mathbf{H}\right)$
and the fact that a Volterra path is of the form $z_{t}^{\tau}=\int_{0}^{t}k\left(\tau,r\right)dx_{r}$,
we easily get the $1$-regularity in the upper argument. This will
play a central role throughout the analysis of such paths. 
\end{rem}

\begin{rem}
The reader might wonder about the introduction of an extra parameter $\zeta$ in the definition  \eqref{eq: two variable  function norm change} of  $\|z\|_{(\alpha,\gamma),1,2}$. In order to justify this new parameter, consider $z\in \mathcal{V}^{(\alpha,\gamma)}$ such that $z_0^\tau=0$ for all $\tau$. We wish to bound the diagonal difference $z_t^t-z_s^s$. To this aim, we decompose the difference as 
\begin{equation}\label{eq:z decomp}
z_t^t-z_s^s=z_{ts}^t+z^{ts}_s=z_{ts}^t+z_{s0}^{ts}. 
\end{equation}
Then the term $z_{ts}^t$ in the right hand side of~\eqref{eq:z decomp} is easily bounded by $|t-s|^\rho$ according to~\eqref{norm1}. In order to bound the term $z_{s0}^{ts}$, we resort to \eqref{eq: two variable  function norm change} and write 
\begin{equation}\label{eq:double sing}
|z_{s0}^{ts}|\lesssim |t-s|^{\eta}|s-s|^{-\eta+\zeta}|s-0|^{\rho-\zeta}. 
\end{equation}
We now tune the parameter $\zeta$ in order to avoid the singularity in $|\cdot|^{-\eta}$ in \eqref{eq:double sing}. Namely, pick $\eta\leq \rho$, and $\zeta=\eta$. This yields 
\begin{equation*}
|z^{ts}_{so}|\lesssim |t-s|^\eta.
\end{equation*}
Plugging this information into 	\eqref{eq:z decomp} we end up with $|z_t^t-z_s^s|\lesssim |t-s|^\eta$ for any $\eta\leq \rho$, which is the desired regularity on the diagonal. 
\end{rem}

\begin{rem}\label{rem:a}
The norms and spaces in Definition \ref{holder norms} can be easily generalized to increments of two variables, which yields the definition of a space $\mathcal{V}_2^{(\alpha,\gamma)}(\Delta_3,E)$. The norm on $\mathcal{V}_2^{(\alpha,\gamma)}(\Delta_3,E)$  is given by 
\begin{equation}\label{eq:three variable gen norm}
\|z\|_{(\alpha,\gamma)}=\|z\|_{(\alpha,\gamma),1}+\|z\|_{(\alpha,\gamma),1,2}.
\end{equation} 
Those spaces will be used for the definition of convolutional controlled paths in Section~\ref{non-linear setting}.
\end{rem}

Our construction of solutions to rough Volterra equations like (\ref{eq:motiv: Volterra equations}) will hinge heavily on a Volterra version of the Sewing Lemma. We start by defining the class $\mathscr{V}^{(\alpha,\gamma)}$ of paths to which this Sewing Lemma will apply. 

\begin{defn}\label{abstract integrnds}
Let $\alpha\in\left(0,1\right)$, $\gamma\in (0,1)$ with $\alpha-\gamma>0$, $\kappa \in (0,\infty)$ and $\beta\in\left(1,\infty\right)$.
Denote by $\mathscr{V}^{(\alpha,\gamma)(\beta,\kappa)}\left(\Delta_{3}\left[0;T\right];E\right)$,
the space of all functions $\varXi:\Delta_{3}\left(\left[0,T\right]\right)\rightarrow E$
such that 
\begin{equation}\label{abstract integrand space}
\|\varXi\|_{\mathscr{V}^{(\alpha,\gamma)(\beta,\kappa)}}=\|\varXi\|_{\left(\alpha,\gamma\right)}+\|\delta\varXi\|_{\left(\beta,\kappa\right)}<\infty,
\end{equation}
where $\delta$  is the operator defined for any $s<u<t$  and a two variables function $g$ by 
\begin{equation}\label{delta}
\delta_{u}g_{ts}=g_{ts}-g_{tu}-g_{us}.
\end{equation}
In (\ref{abstract integrand space}), we also use the following convention: the norm $\|\varXi \|_{(\alpha,\gamma)}$  is given by  \eqref{eq:three variable gen norm}, while we have 
\begin{equation*}
\|\delta \varXi\|_{(\alpha,\gamma)}=\|\delta \varXi\|_{(\alpha,\gamma),1}+\|\delta \varXi\|_{(\alpha,\gamma),1,2},
\end{equation*}
where
the quantities $\|\delta \varXi \|_{(\beta,\gamma),1}$ and $\|\delta \varXi \|_{(\beta,\gamma),1,2}$ are slight modifications of (\ref{norm1}) respectively defined by
\begin{align}\label{dd3}
\|\delta\varXi\|_{\left(\beta,\kappa\right),1}&:=\sup_{\left(s,m,t,\tau\right)\in\Delta_{4}}\frac{|\delta_{m}\varXi_{ts}^{\tau}|}{|\tau-t|^{-\kappa}|t-s|^{\beta}\wedge |\tau-s|^{\beta -\kappa}}
\\\label{dd4}
\|\delta\varXi\|_{\left(\beta,\kappa\right),1,2}&:=\sup_{\substack{\left(s,m,t,\tau^\prime,\tau\right)\in\Delta_{5} \\ \eta\in[0,1],\,\zeta\in [0,\beta-\kappa)}}\frac{|\delta_{m}\varXi_{ts}^{\tau,\tau^\prime}|}{|\tau-\tau^{\prime}|^{\eta}|\tau^{\prime}-t|^{-\eta+\zeta}\left(|\tau^{\prime}-t|^{-\kappa-\zeta}|t-s|^{\beta} \wedge |\tau^{\prime}-s|^{\beta-\kappa-\zeta}\right)} .
\end{align}
In the sequel the space $\mathscr{V}^{(\alpha,\gamma)(\beta,\kappa)}$ will be 
our space of abstract Volterra integrands. 
\end{defn}
We are now ready to state our Sewing Lemma adapted to Volterra integrands.

\begin{lem}\label{lem:(Volterra-sewing-lemma)}
\emph{(Volterra sewing lemma)}
Consider four exponents $\beta\in (1,\infty)$, $\kappa\in (0,1)$, $\alpha\in\left(0,1\right)$ and $\gamma\in(0,1)$
such that $\beta-\kappa\geq \alpha-\gamma >0$. Let $\mathscr{V}^{(\alpha,\gamma)(\beta,\kappa)}$  and $\mathcal{V}^{\left(\al,\gamma\right)}$ be the spaces defined in Definition \ref{abstract integrnds} and  \ref{holder norms} respectively. Then there exists a linear continuous map
$\mathcal{I}:\mathscr{V}^{(\alpha,\gamma)(\beta,\kappa)}\left(\Delta_{3};E\right)\rightarrow\mathcal{V}^{\left(\alpha,\gamma\right)}\left(\Delta_{3};E\right)$
such that the following holds true:

\noindent
$(i)$ The quantity $\mathcal{I}(\varXi^{\tau})_{ts}:=\lim_{|\mathcal{P}|\rightarrow 0} \sum_{[u,v]\in\mathcal{P}} \varXi_{vu}^{\tau}$ exists for all $(s,t,\tau)\in \Delta_{3}$, where $\mathcal{P}$  is a generic partition of $[s,t]$  and $|\mathcal{P}|$  denotes the mesh size of the partition. 

\noindent
$(ii)$ For all $(s,t,\tau)\in \Delta_{3}$ we have that 
\begin{align}\label{sy lemma bound}
|\mathcal{I}\left(\varXi^{\tau}\right)_{ts}-\varXi_{ts}^{\tau}|\lesssim & \|\delta\varXi\|_{\left(\beta,\kappa\right),1}\left(|\tau-t|^{-\kappa}|t-s|^{\beta}\wedge|\tau-s|^{\beta-\kappa}\right),
\end{align} 
while for $(s,t,\tau^{\prime},\tau)\in \Delta_{4}$ we get for any $\eta\in[0,1]$ and $\zeta\in [0,\beta-\kappa)$
\begin{multline}\label{sy lemma upper arg bound}
|\mathcal{I}(\varXi^{\tau\tau^\prime})_{ts}-\varXi_{ts}^{\tau\tau^\prime}| \\
\lesssim  \|\delta\varXi\|_{\left(\beta,\kappa\right),1,2}\left[|\tau-\tau^{\prime}|^{\eta}|\tau^{\prime}-t|^{-\eta+\zeta}\left(|\tau^{\prime}-t|^{-\kappa-\zeta}|t-s|^{\beta} \wedge |\tau^{\prime}-s|^{\beta-\kappa-\zeta}\right)\right].
\end{multline}

\end{lem}

\begin{proof}
This is an elaboration of \cite[Lemma~4.2]{FriHai} and we give some details here for the sake of completeness.
 Specifically, we will focus on the convergence of Riemann type sums
 $\sum_{[u,v]\in \mathcal{P}}\varXi_{vu}^{\tau}$ along dyadic partitions. Referring to \cite[Lemma~4.2]{FriHai}, we leave to the patient reader the task of checking the convergence of  $\sum_{[u,v]\in \mathcal{P}}\varXi_{vu}^{\tau}$ along a general partition whose mesh converges to $0$, as well as the relation $\delta \mathcal{I}\left(\varXi\right)=0$.

With those preliminaries in mind, let us consider the $n$-th order  dyadic partition $\mathcal{P}^{n}$ of $\left[s,t\right]$
where each set $\left[u,v\right]\subset\mathcal{P}^{n}$ is of length
$2^{-n}|t-s|$. We define the $n$-th order Riemann sum of $\varXi^{\tau}$, denoted $\mathcal{I}^{n}\left(\varXi\right)_{ts}$, as follows
\[
\mathcal{I}^{n}\left(\varXi^{\tau}\right)_{ts}=\sum_{\left[u,v\right]\in\mathcal{P}^{n}}\varXi_{vu}^{\tau}.
\]
Our aim is to show that the sequence $\left\{\mathcal{I}^{n}\left(\varXi^{\tau}\right);n\geq 1\right\}$ converges to an element $\mathcal{I}(\varXi)$ which fulfills relation (\ref{sy lemma bound}). To this aim we will analyse differences
 $\mathcal{I}^{n+1}(\varXi^{\tau})-\mathcal{I}^{n}(\varXi^{\tau})$ 
and prove the following bound

\begin{equation}\label{ccc}
|\mathcal{I}^{n+1}(\varXi^{\tau})-\mathcal{I}^{n}(\varXi^{\tau})| 
\lesssim 
\frac{\|\delta\varXi\|_{(\beta,\kappa),1}}{2^{n(\beta-1)}}\left(|\tau-t|^{-\kappa}|t-s|^{\beta}\wedge|\tau-s|^{\beta-\kappa}\right).  
\end{equation}
In order to prove (\ref{ccc}), observe that   
\begin{equation}\label{n+1 diff}
\mathcal{I}^{n+1}\left(\varXi^{\tau}\right)_{ts}-\mathcal{I}^{n}\left(\varXi^{\tau}\right)_{ts}=\sum_{\left[u,v\right]\in\mathcal{P}^{n}}\delta_{m}\varXi_{vu}^{\tau},
\end{equation}
where we recall that $\delta$ is given by relation (\ref{delta}) and where we have set $m=\frac{u+v}{2}$. Plugging relation (\ref{dd3}) into (\ref{n+1 diff}), it is thus readily checked that 
\begin{equation}\label{a9}
|\sum_{\left[u,v\right]\in\mathcal{P}^{n}}\delta_{m}\varXi_{vu}^{\tau}|\lesssim\|\delta\varXi\|_{\left(\beta,\kappa\right)}\sum_{\left[u,v\right]\in\mathcal{P}^{n}}|\tau-v|^{-\kappa}|v-u|^{\beta}.
\end{equation}
We will now upper bound the right hand side above. Invoking the fact that $\beta>1$  and $|v-u|=2^{-n}|t-s|$ for $u,v\in \mathcal{P}^{n}$  we write 
\begin{equation}\label{bound3}
\sum_{\left[u,v\right]\in\mathcal{P}^{n}}|\tau-v|^{-\kappa}|v-u|^{\beta}
\leq
2^{-n\left(\beta-1\right)}|t-s|^{\left(\beta-1\right)}\sum_{\left[u,v\right]\in\mathcal{P}^{n}}|\tau-v|^{-\kappa}|v-u| .
\end{equation}
Hence, some elementary considerations on the Riemann sums corresponding to the integral $\int_{s}^{t}|\tau-r|^{-\kappa}dr$ for a $t<\tau$ and parameter $\kappa\in (0,1)$  yield 
\begin{equation}\label{bound2}
\sum_{\left[u,v\right]\in\mathcal{P}^{n}}|\tau-v|^{-\kappa}|v-u|^{\beta} \lesssim2^{-n\left(\beta-1\right)}|t-s|^{\left(\beta-1\right)}\int_{s}^{t}|\tau-r|^{-\kappa}dr.
\end{equation}
In addition, some elementary calculations similar to those in Remark \ref{rem7}  show that for $\kappa \in (0,1) $ we have 
\begin{equation}\nonumber
\int_{s}^{t}|\tau-r|^{-\kappa}dr\lesssim (\tau-t)^{-\kappa}(t-s)\wedge (\tau-s)^{1-\kappa}, 
\end{equation}
where we have used the fact that the integral $\int_{s}^{t}|\tau-r|^{-\kappa}dr$ is converging for $\kappa< 1$. Putting this inequality into (\ref{bound2}) we get 
\begin{equation}\label{bound riemann sum volterra}
\sum_{\left[u,v\right]\in\mathcal{P}^{n}}|\tau-v|^{-\kappa}|v-u|^{\beta}\lesssim2^{-n\left(\beta-1\right)}\left((\tau-t)^{-\kappa}(t-s)^{\beta}\wedge (\tau-s)^{\beta-\kappa}\right).
\end{equation}
Inserting (\ref{bound riemann sum volterra}) into (\ref{bound3}) and then into (\ref{a9}), our claim (\ref{ccc}) is thus easily obtained. With relation (\ref{ccc}) in hand, one immediately gets that the sequence 
 $\left\{\mathcal{I}^{n} \left( \varXi^\tau
  \right)_{ts}\right\}_{n\geq 0}$ 
  is a Cauchy sequence. It thus converges to a quantity $\mathcal{I}\left(\varXi^\tau\right)_{ts}$ which satisfies  (\ref{sy lemma bound}). As mentioned above, the remainder of the proof goes along the same lines as \cite[Lemma~4.2]{FriHai}. We leave it to the patient reader for the sake of conciseness. This proves that the element $\mathcal{I}\left(\varXi^{\tau}\right)$ has finite $\|\cdot\|_{(\beta,\kappa),1}$ norm and that \eqref{sy lemma bound} holds. The next step will be to show that also the integral $\mathcal{I}(\varXi^{\tau\tau^\prime})$ of the increment in the upper variable $\varXi_{ts}^{\tau\tau^\prime}$ is finite in the $\|\cdot\|_{(\beta,\kappa),1,2}$ norm. Following the lines for the proof above, we can just change the integrand $\varXi_{ts}^\tau$ with $\varXi^{\tau\tau^\prime}_{ts}$ and the norms accordingly. Thus, using exactly the same arguments as before,  inequality \eqref{sy lemma upper arg bound} holds as well. 
 This concludes the proof.  
\end{proof}
In order to test the compatibility of our first definitions with the Sewing lemma, we will show that one can construct a Volterra path of the form
$z_{ts}^{\tau}=\int_{s}^{t}k\left(\tau,r\right)dx_{r}$
 in terms of Lemma \ref{lem:(Volterra-sewing-lemma)}.
\begin{thm}
\label{thm:Regularity of Volterra path} Let $x\in\mathcal{C}^{\alpha}$
and $k$ be a Volterra kernel of order $-\gamma$ satisfying $\left(\mathbf{H}\right)$,
such that $\rho=\alpha-\gamma>0$. 
We define an element $\varXi_{ts}^{\tau}=k(\tau,s)x_{ts}$. Then the following holds true:

\begin{enumerate}[wide, labelwidth=!, labelindent=0pt, label=\emph{(\roman*)}]
\setlength\itemsep{.1in}

\item
There exists a $\beta > 1$ and $\kappa >0$ with $\beta-\kappa=\alpha-\gamma$ such that
$\varXi \in \mathscr{V}^{(\alpha,\gamma)(\beta,\kappa)}$, where $\mathscr{V}^{(\alpha,\gamma)(\beta,\kappa)}$  is given in Definition \ref{abstract integrnds}. Therefore the element $\mathcal{I}\left(\varXi^{\tau}\right)$  obtained by applying Lemma \ref{lem:(Volterra-sewing-lemma)} is well defined as an element of $\mathcal{V}^{(\alpha,\gamma)}$ and we set $z_{ts}^{\tau}\equiv\mathcal{I}\left(\varXi^{\tau}\right)_{ts}=\int_{s}^{t}k(\tau,r)dx_{r}$. 

\item
There exists a strictly positive $c$   such that for $(s,t,\tau)\in \Delta_{3}$  we have 
\begin{equation}\label{regularity of z}
|z_{ts}^{\tau}-k(\tau,s)x_{ts}|\leq c\left[(\tau-t)^{-\gamma}(t-s)^{\alpha}\wedge (\tau-s)^\rho\right],  
\end{equation}
and in particular $z$ verifies $\|z\|_{(\al,\ga),1}<\infty$.

\item
For any $\eta \in [0,1]$ and $\zeta\in [0,\rho)$  there exists a strictly positive constant $c$  such that for any $(s,t,q,p)\in \Delta_{4}$  we have 
\begin{equation}\label{eq: joint reg in z}
|z_{ts}^{pq}|\leq c|p-q|^{\eta}|q-t|^{-\eta+\zeta}\left[|q-t|^{-\gamma-\zeta}|t-s|^{\alpha}\wedge |q-s|^{\rho-\zeta}\right],
\end{equation}
where $z_{ts}^{pq}=z^{p}_{t}-z^{q}_{t}-z_{s}^{p}+z_{s}^{q}$. 

\end{enumerate}
\end{thm}

\begin{rem}
According to the standard rules of algebraic integration we would be naturally prone to set $\varXi_{ts}^{\tau}=k(\tau,t)x_{ts}$. Here we have chosen to take $\varXi_{ts}^{\tau}=k(\tau,s)x_{ts}$, which  will ease the treatment of the singularity of $k$ on the diagonal. This small twist on the usual theory does not affect the fact that we are generalizing Volterra equations from the smooth to the rough case.
\end{rem}
\begin{proof}
Recall that we have set $\varXi_{ts}^{\tau}=k(\tau,s)x_{ts}$. We will show that Lemma \ref{lem:(Volterra-sewing-lemma)} may be applied to $\varXi$, which amounts to check that $\varXi\in \mathscr{V}^{(\alpha,\gamma)(\beta,\kappa)}_{2}$ with some parameters $\beta>1$  and $\kappa>0$ to be chosen later on.
 Furthermore, in order to show that $\|\varXi\|_{\mathscr{V}^{(\alpha,\gamma)(\beta,\kappa)}_{2}}<\infty$  we will focus on the norms $\|\delta\varXi\|_{(\beta,\kappa),1}$ and $\|\delta\varXi\|_{(\beta,\kappa),1,2}$ defined by \eqref{dd3} and \eqref{dd4}, and we leave the proof of $\|\varXi\|_{(\alpha,\gamma)}<\infty$ to the reader for the sake of conciseness.
 	
 	In order to check that $\|\delta \varXi\|_{(\beta,\kappa),1}<\infty $, we start by noting that the increment $\delta_{m}\varXi_{ts}^{\tau}$ can be written as $\delta_{m}\varXi_{ts}^{\tau}=\left[k(\tau,s)-k(\tau,m)\right]x_{tm}$, which stems from elementary algebraic manipulations. Therefore, according to \eqref{3rd} in Hypothesis $(\bf{H})$ we have for an additional parameter $\nu\in [0,1]$  
\begin{equation}\label{estimate}
|\delta_{m}\varXi_{ts}^{\tau}|\lesssim \|x\|_{\alpha}(\tau-m)^{-\gamma-\nu}(t-m)^{\alpha}(m-s)^{\nu}.
\end{equation}
Next we pick our parameter $\nu$  such that the condition  
\begin{align}\label{cond on nu}
\beta \equiv \nu+\alpha>1
\end{align}
is satisfied. 
As far as the singularity at $\tau$ is concerned, relation (\ref{eq: joint reg in z}) asserts that in order to apply Lemma \ref{lem:(Volterra-sewing-lemma)} item $(ii)$ we get the restriction 
\begin{equation}\label{kapa cond}
\kappa\equiv \gamma+\nu <1.
\end{equation}
Note that if we put conditions (\ref{cond on nu}) and (\ref{kapa cond}) together, we get $1-\alpha<\nu<1-\gamma$ which can be fulfilled as long as $\alpha>\gamma$. Furthermore, it is immediate that $\beta-\kappa=\alpha-\gamma$.   Then putting together (\ref{estimate}) with (\ref{cond on nu}) we get that $\|\delta \varXi\|_{(\beta,\kappa),1}<\infty $. 
Next we need to show that $\|\delta \varXi\|_{(\beta,\kappa),1,2}<\infty $. To this aim, define $g_{p}\left(q,s\right)=k\left(p,s\right)-k\left(q,s\right)$.
Then combining~\eqref{4th} and~\eqref{5th} in assumption $\left(\mathbf{H}\right)$ there exist two parameters $\eta,\theta\in [0,1]$ such that  for $p>q>t>m>s$ we have 
\begin{equation}\label{c6}
|g_{p}\left(q,m\right)-g_{p}\left(q,s\right)|\lesssim\left(p-q\right)^{\eta}
\left(q-m\right)^{-(\ga+\theta+\eta)}\left(m-s\right)^{\theta}.
\end{equation}
 With this estimate in mind, let us now define a new abstract Volterra integrand $\varXi_{ts}^{pq}=g_{p}(q,s)x_{ts}$. Repeating the computations of step $(i)$ with $(s,m,t,q,p)\in\Delta_{5}$, and applying~(\ref{bound H}) on $g$ we end up with 
 \begin{equation}\label{429}
 |\delta_{m}\varXi_{ts}^{pq}|\lesssim (p-q)^{\eta}
 (q-t)^{-\eta+\zeta}(q-m)^{-(\ga+\theta+\zeta)}(m-s)^{\theta}(t-m)^{\alpha}, 
 \end{equation}
where $\eta,\theta \in [0,1]$.  Observe that 
$(m-s)^{\theta} (t-m)^{\al} \leq (t-s)^{-\theta+\al}$. 
Thus for any $\zeta\in [0,\beta-\kappa)$  set $\kappa=\gamma+\theta+\zeta<1$ and $\beta=\theta+\alpha>1$ in the same way as in the previous step. Note that this is always possible due to the fact that $\beta-\kappa>0$.  It follows that 
\begin{equation}\nonumber
\|\delta_{m}\varXi\|_{(\beta,\kappa),1,2}<\infty.
\end{equation}
It is therefore clear that $\varXi\in \mathscr{V}^{(\alpha,\gamma),(\beta,\kappa)}$. An application of Lemma \ref{lem:(Volterra-sewing-lemma)} now yields that $\mathcal{I}(\varXi)\in \mathcal{V}^{(\alpha,\gamma)}$ and that the inequalities in (ii)-(iii) holds. 
\end{proof}

\begin{rem}
\label{rem:generality of volterra paths}Owing to Theorem \ref{thm:Regularity of Volterra path}, we now know that a typical example of a Volterra path in $\mathcal{V}^{(\alpha,\gamma)}$ is given by processes of the form  $\int_{s}^{t}k\left(\tau,r\right)dx_{r}$. Having this large class of objects in hand, we will mostly focus on computations for
general elements in $\mathcal{V}^{(\alpha,\gamma)}$ whenever it is not needed to explicitly
state the kernel $k$ or the driving noise $x$. 
\end{rem}
\subsection{Convolution product in the rough case}
\label{sec:conv-prod-rough}
As we have seen in Section \ref{subsec:convolution product}, the equivalent of Chen's relation in our Volterra context involves convolution type integrals.
 In order to clarify this point, let us go back to Remark \ref{rem14} concerning second order iterated integrals. One way to rephrase relation (\ref{relation 320})  with the operator $\delta$ introduced in (\ref{delta}) is the following 		
\begin{equation}
\delta_{s}\bz_{t0}^{\tau,2}=\int_{t>r_{2}>s}k\left(\tau,r_{2}\right)dx_{r_{2}}\otimes  \int_{s>r_{1}>0}k\left(r_{2},r_{1}\right)dx_{r_{1}},\label{eq: second order chens relation}
\end{equation}
In the right hand side of (\ref{eq: second order chens relation})
we point out that the limits of the integration with respect to $x_{r_{1}}$
are fixed; the only thing that is connecting the two integrals is the
dependence on $r_{2}$ through the kernels. Thus the integral $\int_{0}^{s}k\left(r_{2},r_{1}\right)dx_{r_{1}}$
can really be thought of as a re-scaling of the path $x$ as $ r_{2}$
moves from $s$ to $t$. Our next step is to show that this operation
is indeed valid for two generic Volterra paths $y,z$.

\begin{thm}
\label{thm: Integral product well defined}
We consider two Volterra paths
 $z\in\mathcal{V}^{\left(\alpha,\gamma\right)}$ and $y\in\mathcal{V}^{(\alpha^\prime,\gamma^\prime)}$ as given in Definition~\ref{holder norms}, where we recall that $\al,\ga,\al^\prime,\ga^\prime\in(0,1)$, and define $\rho\equiv \al-\ga>0$ and $\rho^\prime\equiv\al^\prime-\ga^\prime>0$. Then the convolution product is a bilinear operation on $\mathcal{V}^{\left(\alpha,\gamma\right)}$ given by 
\begin{equation}\label{convolution in 1d}
\text{\ensuremath{z_{tu}^{\tau}\ast y_{us}^{\cdot}}}
=
\int_{t>r>u}dz_{r}^{\tau}\otimes y_{us}^{r} :=\lim_{|\mathcal{P}|\rightarrow0}\sum_{\left[u^{\prime},v^{\prime}\right]\in\mathcal{P}} z_{v^{\prime}u^{\prime}}^{\tau}\otimes y_{us}^{u^{\prime}}.
\end{equation}
 The integral is understood as a Volterra-Young integral for all $ (s,u,t,\tau)\in \Delta_{4}$. Moreover, the 
 following inequality holds true,
\begin{equation}\label{reg of conv 25}
|z_{tu}^{\tau}\ast y_{us}^{\cdot}|\lesssim\|z\|_{(\alpha,\gamma),1}\|y\|_{(\alpha^\prime,\gamma^\prime),1,2}\left[\left(\tau-t\right)^{-\gamma}\left(t-s\right)^{\rho+\rho^\prime+\gamma}\wedge(\tau-s)^{\rho+\rho^\prime}\right].
\end{equation}
\end{thm}

\begin{proof}
Define $\varXi_{r^{\prime}r}^{\tau}:=z_{r^{\prime}r}^{\tau}\otimes y_{us}^{r},$ for $0\leq s<u\leq r\le m \leq r^{\prime}\leq t$. In spirit of Lemma~\ref{lem:(Volterra-sewing-lemma)}, we will show that 
\begin{equation}\nonumber
|\mathcal{I}\left(\varXi^{\tau}\right)_{tu}-\varXi^{\tau}_{tu}|\lesssim \|z\|_{(\alpha,\gamma),1}\|y\|_{(\alpha^\prime,\gamma^\prime),1,2}\lc (\tau-t)^{-\gamma}(t-s)^{\rho+\rho^\prime+\gamma}\wedge(\tau-s)^{\rho+\rho^\prime} \rc.
\end{equation}
Following the strategy outlined in the proof of Lemma \ref{lem:(Volterra-sewing-lemma)}, we know from (\ref{a9}) that we must show that the sum $\sum_{\left[r,r^\prime\right]\in\mathcal{P}^{n}[u,t]}|\delta_{m}\varXi_{r^\prime r}^{\tau}|$ is converging (here $\mathcal{P}^n$  is the dyadic partition used in the proof of Lemma \ref{lem:(Volterra-sewing-lemma)}). Let us therefore consider the action of $\delta$  on $\varXi$. By simple algebraic manipulations we see that 
 \begin{equation}\label{delta int prod}
 \delta_{m}\varXi_{r^{\prime}r}^{\tau} = -
  z_{r^{\prime}m}^{\tau}\otimes y_{us}^{mr}.
 \end{equation}
Let us now analyse the right hand side of (\ref{delta int prod}). The term $z_{r^{\prime}m}^{\tau}$ can be bounded thanks to assumption (\ref{norm1}). We get
\begin{equation}\label{bound in prod 1}
|z_{r^{\prime}m}^{\tau}|\leq \|z\|_{(\alpha,\gamma),1} |\tau-r^{\prime}|^{-\gamma}|r^{\prime}-m|^{\alpha}.
\end{equation}
As for the term $y_{us}^{mr}$ we can use assumption (\ref{eq: two variable  function norm change}) to write 
\begin{equation}\label{bound in prod 2}
|y_{us}^{mr}|\leq \|y\|_{(\alpha^\prime,\gamma^\prime),1,2}|m-r|^{\eta} |r-u|^{-\eta} |r-s|^{\rho^\prime},
\end{equation}
for an arbitrary $\eta\in[0,1]$. Hence gathering (\ref{bound in prod 1})  and (\ref{bound in prod 2})   we bound (\ref{delta int prod}) by 
\begin{equation}\label{reg in two sing}
 | z_{r^{\prime}m}^{\tau}\otimes y_{us}^{mr}|\lesssim  \|y\|_{(\alpha^\prime,\gamma^\prime),1,2}\|z\|_{(\alpha,\gamma),1} 
 (r-u)^{-\eta}(\tau-r^{\prime})^{-\gamma}(r^{\prime}-r)^{\alpha+\eta}\,(r-s)^{\rho^\prime},
\end{equation}
where we have used the fact that $|r^{\prime}-m|\lesssim |r^\prime -r|$ and $|m-r|\lesssim |r^\prime -r|$.
 
 Combining (\ref{reg in two sing}) with (\ref{delta int prod})  and summing over the points of the dyadic partition $\mathcal{P}^n$, we end up with 
 \begin{multline}\label{double sing riemann}
 \sum_{\left[r,r^\prime\right]\in\mathcal{P}^{n}[u,t]}|\delta_{m}\varXi_{r^\prime r}^{\tau}|
 \\
 \lesssim 
 \|y\|_{(\alpha^\prime,\gamma^\prime),1,2}\|z\|_{(\alpha,\gamma),1}  \sum_{\left[r,r^\prime\right]\in\mathcal{P}^{n}[u,t]} 
 (r-u)^{-\eta}(\tau-r^{\prime})^{-\gamma}(r^{\prime}-r)^{\alpha+\eta}\,(r-s)^{\rho^\prime} .
 \end{multline}
 Note that we have two separate possible singular points above, both when $r\rightarrow u$  and $r^\prime \rightarrow \tau$.   However, taking limits in the Riemann sums on the right hand side of~(\ref{double sing riemann}), we know that we obtain a converging integral  as long as $\eta+\alpha>1$  and $\eta<1$. Indeed, the right hand side of (\ref{double sing riemann}) is bounded (up to a multiplicative constant) by the integral $|\mathcal{P}^{n}|^{\alpha+\eta-1}\int_{u}^{t}(\tau-a)^{-\gamma}(a-u)^{-\eta}(a-s)^{\rho^\prime}da$, and by doing a change of variables $a=u+\theta(t-u)$ as well as applying the inequality 
 \begin{equation*}
 \sup_{\theta\in[0,1]}(u-s+\theta(t-u))^{\rho^\prime}\leq (t-s)^{\rho^\prime},
\end{equation*} 
we find that 
 \begin{align} \label{int term1}
 \int_{u}^{t}(\tau-a)^{-\gamma}(a-u)^{-\eta}(a-s)^{\rho^\prime}da 
 \leq 
 c_{\eta,\ga} \, (\tau-u)^{-\gamma}(t-u)^{1-\eta}(t-s)^{\rho^\prime} 
 \end{align}
 where $c_{\eta,\ga}=B(1-\gamma,1-\eta)$ and we recall that $B$ stands for the Beta function as in the proof of Proposition~\ref{prop:gamma bound smooth functionals}. 
It follows that 
\begin{equation}\label{choose eta}
\sum_{\left[r,r^\prime\right]\in\mathcal{P}^{n}[u,t]}|\delta_{m}\varXi_{r^\prime r}^{\tau}|\lesssim |\mathcal{P}^{n}|^{\alpha+\eta-1} \|y\|_{(\alpha^\prime,\gamma^\prime),1,2}\|z\|_{(\alpha,\gamma),1}(\tau-u)^{-\gamma}(t-u)^{1-\eta}(t-s)^{\rho^\prime}, 
\end{equation}
 Since we must choose $\eta>1-\al$, let us choose $\eta=1-\al+\epsilon$  for some small $\epsilon>0$ satisfying $\rho-\epsilon>0$. Then inequality (\ref{choose eta}) reads
\begin{equation}\label{2choose eta}
\sum_{\left[r,r^\prime\right]\in\mathcal{P}^{n}[u,t]}|\delta_{m}\varXi_{r^\prime r}^{\tau}|\lesssim |\mathcal{P}^{n}|^{\epsilon} \|y\|_{(\alpha^\prime,\gamma^\prime),1,2}\|z\|_{(\alpha,\gamma),1}(\tau-u)^{-\gamma}(t-u)^{\alpha-\epsilon}(t-s)^{\rho^\prime}. 
\end{equation}
Note that for the dyadic partition $\mathcal{P}^n$ we have $|\mathcal{P}^{n}|^{\epsilon} =2^{-n\epsilon}(t-u)^{\epsilon}$, and observe that~(\ref{2choose eta}) is the equivalent of  (\ref{bound riemann sum volterra}) in our current setting. Therefore, one one can follow the same steps as in Lemma \ref{lem:(Volterra-sewing-lemma)} in order to get the following relation, which is the analog of \eqref{ccc}:
\begin{equation}\label{eqq1}
|\mathcal{I}^{n+1}(\varXi^{\tau})-\mathcal{I}^{n}(\varXi^{\tau})| 
\lesssim 
\frac{ \|y\|_{(\alpha^\prime,\gamma^\prime),1,2}\|z\|_{(\alpha,\gamma),1}}{2^{n\epsilon}}(\tau-u)^{-\gamma}(t-u)^{\alpha}(t-s)^{\rho^\prime}, 	
\end{equation}
where we recall that $2\alpha-\gamma=2\rho+\gamma$. We also let the patient reader check from (\ref{eqq1}) that 
\begin{equation}\label{eqq2}
|\mathcal{I}^{n+1}(\varXi^{\tau})-\mathcal{I}^{n}(\varXi^{\tau})| 
\lesssim 
\frac{ \|y\|_{(\alpha^\prime,\gamma^\prime),1,2}\|z\|_{(\alpha,\gamma),1}}{2^{n\epsilon}}\left[(\tau-t)^{-\gamma}(t-s)^{\rho+\rho^\prime+\gamma}\wedge (\tau-s)^{\rho+\rho^\prime}\right], 	
\end{equation}
where we have used that for $s\leq u\leq t\leq \tau$ the following two inequalities holds: 
\begin{equation*}
(\tau-u)^{-\gamma}(t-u)^\alpha(t-s)^{\rho^\prime}\leq (\tau-s)^{\rho+\rho^\prime} \quad \mathrm{and} \quad  (\tau-u)^{-\gamma}(t-u)^\alpha(t-s)^{\rho^\prime} \leq (\tau-t)^{-\gamma}(t-s)^{\rho+\rho^\prime+\gamma}.
\end{equation*}

 Putting together (\ref{eqq1})  and (\ref{eqq2})  and reasoning exactly as in Lemma \ref{lem:(Volterra-sewing-lemma)} after (\ref{bound riemann sum volterra}), we obtain that $\mathcal{I}^{n}(\varXi^{\tau})$ converges to an element $ \mathcal{I}(\varXi^{\tau})$
verifying 
\begin{equation}\label{443}
|\mathcal{I}(\varXi^{\tau})_{tu}-\varXi_{tu}^{\tau}| \lesssim 
\|y\|_{(\alpha^\prime,\gamma^\prime),1,2}\|z\|_{(\alpha,\gamma),1}
\lc (\tau-t)^{-\gamma}(t-s)^{\rho+\rho^\prime +\gamma}\wedge(\tau-s)^{\rho+\rho^\prime} \rc.
\end{equation}
We therefore define
$
z_{tu}^{\tau}\ast y_{us}^{\cdot}:=\mathcal{I}\left(\varXi^{\tau}\right)_{tu},$
and one can directly see from (\ref{443}) that $ z_{tu}^{\tau}\ast y_{us}^{\cdot}$ satisfies the relation
\begin{equation}\nonumber
|z_{tu}^{\tau}\ast y_{us}^{\cdot}|
\lesssim 
\|y\|_{(\alpha^\prime,\gamma^\prime),1,2}\|z\|_{(\alpha,\gamma),1}
\lc (\tau-t)^{-\gamma}(t-s)^{\rho+\rho^\prime+\gamma}\wedge(\tau-s)^{\rho+\rho^\prime} \rc. 
\end{equation}
This completes the proof. 
\end{proof}

Our next step is to mimick Proposition \ref{prop:convolutional Chens relation smooth} in a rough Volterra context. Specifically we would like to extend Theorem \ref{thm: Integral product well defined} in order to get a proper definition of the $n$-th order convolution products for Volterra rough paths (where we recall that Volterra rough paths are introduced in Definition \ref{holder norms}). For those $n$-th order convolution rough paths, we also wish to get a multiplicative property similar to Proposition \ref{prop:convolutional Chens relation smooth}. 

Observe that in order to properly define the aforementioned  $n$-th order convolution product, we will need to extend the domain of the definition of our convolution product  $\ast$. Namely, we would like to define products of the form $z_{ts}^{2,\tau}\ast f^{\cdot,\cdot}_{s}$ for a generic function $(s,\tau_1,\tau_2)\mapsto f_{s}^{\tau_2,\tau_1}$.
To motivate this construction, suppose $x:[0,T]\rightarrow \RR$ is a smooth path, and $f:\Delta_2\rightarrow \RR$ is a smooth function. Furthermore, assume that there exists $f':\Delta_3\rightarrow \RR$ and $R:\Delta_3\rightarrow \RR$ such that $f_{ts}^\tau=z_{ts}^\tau\ast f_s^{',\tau,\cdot}+R_{ts}^\tau$ for some smooth $z:\Delta_2\rightarrow \RR$. Consider the integral $\int_s^t k(\tau,r)f_rdx_r$. Inserting the relation on $f$ inside the integral, we see that 
\begin{equation*}
\int_s^t k(\tau,r)f_rdx_r=\int_s^t k(\tau,r)f_s^rdx_r+ \int_s^tk(\tau,r) z_{rs}^r\ast f_s^{',r,\cdot}dx_r+\int_s^tk(\tau,r)R_{rs}^\tau dx_r
\end{equation*}
If we assume that $z_t^\tau=\int_0^tk(\tau,r)dx_r$, we recognize that $\int_s^t k(\tau,r)f_s^rdx_r=z_{ts}^\tau\ast f_s^\cdot$, which is the first order convolution product. However, observe that for the second term we have (since all functions considered are smooth) 
\begin{equation}\label{eq:prod442}
\int_s^tk(\tau,r) z_{rs}^r\ast f_s^{',r,\cdot}dx_r=\int_s^t \int_s^r k(\tau,r)k(r,u)f_s^{',r,u}dx_udx_r.
\end{equation}
Now observe  that $\int_s^t \int_s^r k(\tau,r)k(r,u)dx_udx_r=\bz^{2,\tau}_{ts}$, and so we are tempted to define $\int_s^t \int_s^r k(\tau,r)k(r,u)f_s^{',r,u}dx_udx_r=\bz^{2,\tau}_{ts}\ast f_s^{',\cdot_1,\cdot_2}$. 
However, in the current situation the convolution product is performed over two upper variables. Hence we need to extend the construction from Theorem \ref{thm: Integral product well defined} to this context. In subsequent sections we will give a proper definition of controlled rough Volterra paths, and will then see that this is exactly the type of relations that is needed in order to define rough integrals.

 Let us first explain how a product like \eqref{eq:prod442} behaves in case  of a smooth path $x$ with a Volterra kernel $k$. Namely in this situation, consider a smooth three variable function $f:\Delta_{3}\rightarrow \mathcal{L}(E,\mathcal{L}(E))$. Then a natural way to define $z_{ts}^{2,\tau}\ast f^{\cdot,\cdot}_{s}$ is the following (the reason we assume $f$  has two upper arguments will be discussed in detail in Section~\ref{non-linear setting}).

\begin{defn}\label{def two step conv}
Let $x$  be a continuously differentiable function and consider a Volterra kernel $k$ which fulfills $\left(\bf{H}\right)$ with $\gamma<1$. Let also $f:\Delta_3\rightarrow \mathcal{L}\left(E,\mathcal{L}\left(E\right)\right)$ be a smooth function. Then for $\tau\ge t>s\ge v$ the convolution $\bz^{2,\tau}_{ts}\ast f_{v}^{\cdot_{1},\cdot_{2}}$  is defined by 
\begin{equation}\label{conv 2 step}
\bz^{2,\tau}_{ts}\ast f_{v}^{\cdot_{1},\cdot_{2}}=\int_{t>r>s}k(\tau,r)dx_{r}\otimes \int_{r>l>s}k(r,l)f_{v}^{r,l} dx_{l},
\end{equation}
where the notation $f_{v}^{\cdot_{1},\cdot_{2}}$  is introduced to prevent ambiguities about the order of integration.
\end{defn}

We now state an algebraic type lemma which will be useful in order to extend Definition~\ref{def two step conv} to rougher contexts. 

\begin{lem}
Under the same conditions as in Definition \ref{def two step conv}, let $\bz^{2,\tau}_{ts}\ast f_{s}^{\cdot_{1},\cdot_{2}}$  be the increment given by (\ref{conv 2 step}). Consider $(s,t)\in \Delta_{2}$  and a generic partition $\mathcal{P}$  of $[s,t]$. Then we have  
\begin{equation}\label{division of z2 conv}
\bz^{2,\tau}_{ts}\ast f_{s}^{\cdot_{1},\cdot_{2}}
=
\lim_{|\mathcal{P}|\to 0}\sum_{[u,v]\in \mathcal{P}} \bz^{2,\tau}_{vu}\ast f_{s}^{\cdot_{1},\cdot_{2}} 
+\left(\delta_{u}\bz^{2,\tau}_{vs}\right)\ast f_{s}^{\cdot_{1},\cdot_{2}}. 
\end{equation} 
\end{lem}
\begin{proof}
Starting from \eqref{conv 2 step}, we first write 
\begin{equation*}
\bz^{2,\tau}_{ts}\ast f_{s}^{\cdot_{1},\cdot_{2}}
=
\sum_{[u,v]\in \mathcal{P}}
\int_{v>r>u}k(\tau,r)dx_{r}\otimes \int_{r>l>s}k(r,l)f_{s}^{r,l} dx_{l}.
\end{equation*}
Then for each $[u,v]\in \mathcal{P}$, divide the region $\{v>r>u \} \cap \{r>l>s \}$
into 
\begin{equation*}
\{v>r>l>u\}\cup\{v>r>u>l>s\}. 
\end{equation*}
This yields a decomposition of 
$\bz^{2,\tau}_{ts}\ast f_{s}^{\cdot_{1},\cdot_{2}}$  of the form 
\begin{equation}\label{AB eq}
\bz^{2,\tau}_{ts}\ast f_{s}^{\cdot_{1},\cdot_{2}}=\sum_{[u,v]\in \mathcal{P}} A^{\tau}_{vu}+B^{\tau}_{vu},
\end{equation}
where $A$  and $B$  are respectively given by  
\begin{align*}
A_{vu}^{\tau}=\int_{v>r>u}k(\tau,r)dx_{r}\otimes \int_{r>l>u}k(r,l)f_{s}^{r,l}dx_{l} 
\\
B_{vu}^{\tau}=\int_{v>r>u}k(\tau,r)dx_{r}\otimes \int_{u>l>s}k(r,l)f_{s}^{r,l}dx_{l}. 
\end{align*}
Now we immediately recognize the term $A_{vu}^\tau$  as the expression $\bz^{2,\tau}_{vu}\ast f_{s}^{\cdot_{1},\cdot_{2}}$ given by (\ref{conv 2 step}). Moreover, it is also readily checked that $B_{vu}^{\tau}=z^{1,\tau}_{vu}\ast \left(z^{1,\cdot_1}_{us}\ast f_{s}^{\cdot_1,\cdot_2 }\right)$. Hence  thanks to relation~(\ref{con con}) for smooth paths we can also write 
\begin{equation}\nonumber
B_{vu}^{\tau}=\left(\delta_{u}\bz^{2,\tau}_{vs}\right)\ast f_{s}^{\cdot_1,\cdot_2}.
\end{equation}
Plugging this relation into (\ref{AB eq}) and gathering the information we have on the term $A^\tau$, our proof is complete. 
\end{proof}

\begin{rem}
The identity (\ref{division of z2 conv}) makes sense as long as one can define $\bz^{2,\tau}\ast f^{\cdot_1,\cdot_2}$  and if $\bz^{2}$ verifies~(\ref{con con}). This opens the way to a generalization to rougher situations, having Theorem~\ref{thm: Integral product well defined} in  mind for the equivalent of (\ref{con con}). These considerations motivate the definition in Theorem~\ref{thm:genreal volterra convolution }. 
\end{rem}

We now take another step towards a proper definition of general convolution products. To this aim, we will assume for a moment that our generic Volterra path $z^{\tau}$  gives raise to a stack $\{\bz^{j,\tau};j\leq n\}$  of iterated integrals. Specifically our standing assumption is the following: 
\begin{description}[wide, labelwidth=!, labelindent=0pt]
\item[H2] Let $z\in\mathcal{V}^{(\alpha,\gamma)}$ be a Volterra path, as introduced in Definition \ref{holder norms}. For $n$ such that $(n+1)\rho+\ga>1$, we assume that there exists a family $\{\bz^{j,\tau};j\leq n\}$ with $\bz^1=z$ satisfying
\begin{equation}\label{eq: hyp delta rel}
\delta_{u}\bz^{j,\tau}_{ts}=\sum_{i=1}^{j-1}\bz^{j-i,\tau}_{tu}\ast\bz^{i,\cdot}_{us},
\end{equation}
where the convolution product is defined by the right hand side of (\ref{convolution in 1d}). In addition, we suppose that for $j=1,\ldots,n$ we have $\bz^{j}\in \mathcal{V}^{(j\rho+\gamma,\gamma)}(\Delta_3,E)$.
\end{description}

Let us also specify the kind of norm we shall consider for processes with 2 upper variables of the form $y^{\cdot_1,\cdot_2}$.

\begin{defn}\label{def: upper vb space}
Let $y$ be a function from $\Delta_3$ to $V$ such that for any $(\tau_1,\tau_2)\in \Delta_2$ we have $y_0^{\tau_1,\tau_2}=y_0\in V$, and such that 
\begin{equation}\label{upper vb norm}
\|y^{\cdot_1,\cdot_2}\|_{(\alpha,\gamma),1,2}:=\|y^{\cdot_1,\cdot_2}\|_{(\alpha,\gamma),1,2,>} +\|y^{\cdot_1,\cdot_2}\|_{(\alpha,\gamma),1,2,<} <\infty
\end{equation}
where the two norms $\|\cdot\|_{(\alpha,\gamma),2,>}$ and $\|\cdot\|_{(\alpha,\gamma),2,<}$ are small variations of \eqref{eq: two variable  function norm change}, respectively defined by
\begin{align}\label{>}
\|y^{\cdot_1,\cdot_2}\|_{(\alpha,\gamma),1,2,>} =\sup_{\substack{(s,t,r^\prime,r_1,r_2)\in \Delta_5\\ \eta\in [0,1],\, \zeta\in [0,\alpha-\gamma)}} \frac{|y_{ts}^{r^\prime,r_2}-y_{ts}^{r^\prime,r_1}|}{|r_2-r_1|^{\eta}|r_1-t|^{-\eta+\zeta}\left[|r_1-t|^{-\gamma-\zeta}|t-s|^\alpha\wedge |r_1 -s|^{\alpha-\gamma-\zeta}\right]}
\\\label{<}
\|y^{\cdot_1,\cdot_2}\|_{(\alpha,\gamma),1,2,<} =\sup_{\substack{(s,t,r_1,r_2,r^\prime)\in \Delta_5\\ \eta\in [0,1],\,\zeta\in [0,\alpha-\gamma)}} \frac{|y_{ts}^{r_2,r^\prime}-y_{ts}^{r_1,r^\prime}|}{|r_2-r_1|^{\eta}|r_1-t|^{-\eta+\zeta}\left[|r_1-t|^{-\gamma-\zeta}|t-s|^\alpha\wedge |r_1 -s|^{\alpha-\gamma-\zeta}\right]},
\end{align}
We denote the space of functions such that \eqref{upper vb norm} is fulfilled by $\mathcal{V}^{\cdot_1,\cdot_2}_{(\alpha,\gamma)}$. 
\end{defn}
\begin{rem}\label{norm_rem}
In the sequel we will need to estimate differences of functions $y^{\cdot_1,\cdot_2}:\Delta_3 \rightarrow V$ the form $|y_s^{\tau,v}-y_s^{\tau,u}|\lesssim |v-u|^\eta|u-s|^{-\eta}$ uniformly over $\tau$ and $s$.
 However, if $y\in \mathcal{V}^{\cdot_1,\cdot_2}_{(\alpha,\gamma)}$, it is readily checked that 
\begin{align*}
|y_s^{\tau,v}-y_s^{\tau,u}|&\leq |y_0^{\tau,v}-y_0^{\tau,u}|+|y_{s0}^{\tau,v}-y_{s0}^{\tau,u}|
\\
&\leq \|y^{\cdot_1,\cdot_2}\|_{(\alpha,\gamma),1,2,>}|v-u|^\eta|u-s|^{-\eta}|s-0|^{\alpha-\gamma}
\end{align*}
where we have used that fact that since $y\in\mathcal{V}^{\cdot_1,\cdot_2}_{(\alpha,\gamma)}$ the difference $|y_0^{\tau,v}-y_0^{\tau,u}|=0$. Thus, we can use the norm in \eqref{>} to control the increments $y_s^{\tau,v}-y_s^{\tau,u}$. The same can of course be done for increments in the first variable, using the norm in \eqref{<}. 
\end{rem}

Assuming Hypothesis $(\bf{H2})$, and having Definition \ref{def: upper vb space} in mind,  we now state a general convolution result for functions defined on $\Delta_3$.  

\begin{thm}\label{thm:genreal volterra convolution }
Let $z\in\mathcal{V}^{(\alpha,\gamma)}$ with $\alpha,\gamma\in(0,1)$  satisfying $\rho=\alpha-\gamma>0$, as given in Definition~\ref{holder norms}. We assume that $z$  fulfills hypothesis $(\bf{H2})$.   Consider a function  $y:\Delta_3\rightarrow \mathcal{L}(E,V)$ such that $y$ is in the space $\mathcal{V}^{\cdot_1,\cdot_2}_{(\alpha,\gamma)}$ given in Definition \ref{def: upper vb space}.
Then we have for all fixed $(s,t,\tau)\in \Delta_{3}$ that 
\begin{equation}
\label{def of conv}
\bz_{ts}^{2,\tau}\ast y^{\cdot_{1},\cdot_{2}}_{s}:=\lim_{|\mathcal{P}|\rightarrow0}\sum_{\left[u,v\right]\in\mathcal{P}} \bz_{vu}^{2,\tau}\otimes y^{u,u}_{s} +(\delta_{u}\bz_{vs}^{2,\tau}) \ast y^{\cdot_{1},\cdot_{2}}_{s}
\end{equation}
 is a well defined Volterra-Young integral. It follows that $\ast$  is a  well defined bilinear operation between the three parameters Volterra function $\bz^{2}$  and a $3$-parameter path $y$.   Moreover, we have that 
\begin{multline}\label{eq:reg of conv product}
|\bz_{ts}^{2,\tau}\ast y^{\cdot_{1},\cdot_{2}}_{s}- \bz_{ts}^{2,\tau }\otimes y^{s,s}_{s}|
\lesssim 
\|y^{\cdot_{1},\cdot_{2}} \|_{(\alpha,\gamma),1,2}  
\left(\|\bz^{2}\|_{(2\rho+\gamma,\gamma),1}+\|\bz^{1}\|_{(\alpha,\gamma),1,2}\|\bz^{1}\|_{(\alpha,\gamma),1}\right)
\\
\times\lp |\tau-s|^{-\gamma}|t-s|^{2\rho+\gamma}\wedge|\tau-s|^{2\rho} \rp.
\end{multline}
\end{thm}

\begin{rem}
Our definition (\ref{def of conv}) for $\bz_{ts}^{2,\tau}\ast y^{\cdot_{1},\cdot_{2}}_{s}$ is obviously motivated by  (\ref{division of z2 conv}), which had been obtained for smooth Volterra paths. We are now extending this identity to a generic path in $\mathcal{V}^{(\alpha,\gamma)}$. 
\end{rem}
\begin{rem}
The term $(\delta_{u}\bz_{vs}^{2,\tau}) \ast y^{\cdot_{1},\cdot_{2}}_{s}$ in (\ref{conv 2 step}) is defined in the following way: observe that according to (\ref{eq: hyp delta rel}) we have 
\begin{equation}\label{e1}
\delta_{u}\bz_{vs}^{2,\tau}=\bz^{1,\tau}_{vu}\ast\bz^{1,\cdot}_{us}.
\end{equation}
Therefore we get 
\begin{equation}\nonumber
(\delta_{u}\bz_{vs}^{2,\tau})\ast y_{s}^{\cdot_1,\cdot_2} =\bz^{1,\tau}_{vu}\ast\bz^{1,\cdot}_{us} \ast y_{s}^{\cdot_1,\cdot_2},
\end{equation}
which is well defined from a successive application of Theorem \ref{thm: Integral product well defined}. Indeed, the convolution $\bz^{1,p}_{ts}\ast y^{r,\cdot}_{s}$ for $p\geq r$ can be constructed in the exact same way as we constructed $\bz^{1,\tau}_{vu}\ast\bz^{1,\cdot}_{us}$. Namely, $y^{\cdot_1,\cdot_2}_s$ has to be considered as a constant in the lower variable.  However, in light of Remark \ref{norm_rem}, the $\|y\|_{(\alpha,\gamma),1,2}$ norm invoked in \eqref{reg of conv 25} will be changed to the regularity required in \eqref{upper vb norm}.
\end{rem}

\begin{proof}[Proof of Theorem \ref{thm:genreal volterra convolution }]
Let us denote by $\mathcal{I}_{\mathcal{P}}$ the approximation of the right hand side of~\eqref{def of conv}, that is 
\begin{equation} \label{Riemann sums2}
\mathcal{I}_{\mathcal{P}}:=\sum_{\left[u,v\right]\in\mathcal{P}} \varXi_{vu}^{\tau}:=\sum_{\left[u,v\right]\in\mathcal{P}}\bz_{vu}^{2,\tau}\otimes y^{u,u}_{s}+(\delta_{u}\bz_{vs}^{2,\tau}) \ast y_{s}^{\cdot_{2},\cdot_{1}}.
\end{equation}
 Our goal is to apply Lemma \ref{lem:(Volterra-sewing-lemma)}  to the increment $\varXi$, and we  must therefore check the regularity of the integrand under the action of $\delta$. To this aim, two simple computations using that $\delta_{r} \bz_{vu}^{2,\tau}=\bz_{vr}^{1,\tau}\ast \bz_{ru}^{1,\cdot} $ reveal 
\begin{align}\label{delta rule}
\delta_{r} (\bz_{vu}^{2,\tau}\otimes y^{u,u}_{s}) &=-\bz_{vr}^{2,\tau}\otimes (y^{r,r}_{s}-y^{u,u}_{s})+ \bz_{vr}^{1,\tau}\ast \bz_{ru}^{1,\cdot}\otimes y^{u,u}_{s} ,
\\\label{delta two rule}
\delta_{r}((\delta_{u}\bz_{vs}^{2,\tau}) \ast y^{\cdot_{1},\cdot_{2}}_{s}) 
&=-\bz_{vr}^{1,\tau}\ast \bz_{ru}^{1,\cdot} \ast y^{\cdot_{1},\cdot_{2}}_{s}, 
\end{align}
where we notice that (since we are computing $\delta_r\varXi_{vu}^\tau$) we have 
\begin{equation*}
\delta_r\left(\delta_u \bz^{2,\tau}_{vs}\right)=\delta_u \bz^{2,\tau}_{vs}-\delta_r \bz^{2,\tau}_{vs}-\delta_u \bz^{2,\tau}_{rs}
=-\bz_{vr}^{1,\tau}\ast \bz_{ru}^{1,\cdot},
\end{equation*}
where we invoked \eqref{e1} for the last identity. 
 Let us now analyse the regularities of the terms in \eqref{delta rule}-\eqref{delta two rule}, starting with the right hand side of \eqref{delta rule}. 
 Namely we recall that we assume in hypothesis $(\bf{H2})$  that $\bz^2\in\mathcal{V}^{(2\rho+\gamma,\gamma)}$, and we also have $\|y^{\cdot_1,\cdot_2}\|_{(\alpha,\gamma),1,2}<\infty$ according to~\eqref{upper vb norm}. Therefore recalling~\eqref{>},~\eqref{<} and Remark \ref{norm_rem},   and also recalling that $u\leq r\leq v$   we  have  for all $\eta\in[0,1]$
\begin{equation}\label{reg z2 con y}
| \bz_{vu}^{2,\tau}\otimes (y^{r,r}_{s}-y^{u,u}_{s}) |\lesssim \|y^{\cdot_{1},\cdot_{2}} \|_{(\alpha,\gamma),1,2} \|\bz^{2}\|_{(2\rho+\gamma,\gamma),1}|u-s|^{-\eta}|\tau-v|^{-\gamma}|v-u|^{2\rho+\gamma+\eta}, 
\end{equation} We then choose $\eta$  such that  $2\rho+\gamma+\eta>1$, at the same time as $\eta<1$, which is always possible, since $\rho>0$.

In order to treat the remaining terms in (\ref{delta rule}) and (\ref{delta two rule}), observe that formula (\ref{convolution in 1d}) trivially yields (recall again that $y_s^{u,u}$ has to be considered as a constant  in the lower variable)
\begin{equation*}
\bz^{1,\tau}_{ts}\ast y_{s}^{u,u}=\bz^{1,\tau}_{ts}\otimes y_{s}^{u,u}.
\end{equation*}
 Therefore we can gather our two remaining terms into 
\begin{equation}\label{four sixty two}
 \bz_{vr}^{1,\tau}\ast \bz_{ru}^{1,\cdot}\otimes y^{u,u}_{s}-\bz_{vr}^{1,\tau}\ast \bz_{ru}^{1,\cdot} \ast y^{\cdot_{1},\cdot_{2}}_{s}=-\bz_{vr}^{1,\tau}\ast \bz_{ru}^{1,\cdot} \ast (y^{\cdot_{1},\cdot_{2}}_{s}-y^{u,u}_{s}).
\end{equation}
Now in the spirit of Theorem \ref{thm: Integral product well defined}, Inequality \eqref{reg of conv 25}  and using condition~\eqref{upper vb norm} as well as relation \eqref{reg z2 con y}, we  have  
\begin{multline}\label{reg con con y}
|\bz_{vr}^{1,\tau}\ast \bz_{ru}^{1,\cdot} \ast (y^{\cdot_{1},\cdot_{2}}_{s}-y^{u,u}_{s})|
\\
\lesssim 
\|y^{\cdot_{1},\cdot_{2}} \|_{(\alpha,\gamma),1,2}\|\bz^{1}\|_{(\alpha,\gamma),1} 
\|\bz^{1}\|_{(\alpha,\gamma),1,2}|\tau-v|^{-\gamma}|v-u|^{2\rho+\gamma+\eta}|u-s|^{-\eta}.
\end{multline} 

Notice that the regularity obtained in (\ref{reg con con y}) is the same as for (\ref{reg z2 con y}). Hence repeating the same arguments as after (\ref{reg z2 con y}) and recalling (\ref{delta rule}) and  (\ref{delta two rule}), we have obtained that 
\begin{equation}\nonumber
|\delta_{r}\varXi_{vs}^\tau|\lesssim c_{y,\bz}|\tau-v|^{-\gamma}|u-s|^{-\eta}|v-u|^\mu,  
\end{equation}
where $\eta<1$  and $\mu=2\rho+\gamma+\eta>1$, and where the constant $c_{y,\bz}$  is the same as in the right hand side of (\ref{reg con con y}). 

We are now in a situation which is similar to the one we had encountered in the proof of Theorem \ref{thm: Integral product well defined} (see inequality~\eqref{double sing riemann} in particular). Thus along the same lines as Theorem~\ref{thm: Integral product well defined}, resorting to a slight modification of the Sewing lemma \ref{lem:(Volterra-sewing-lemma)} involving two possible singularities, we get that the Riemann sums defined by \eqref{Riemann sums2} converge as $|\mathcal{P}|\rightarrow 0$, and we define 
\begin{equation}\label{four sixty six}
\bz_{ts}^{2,\tau}\ast y^{\cdot_{1},\cdot_{2}}:=\lim _{|\mathcal{P}|\rightarrow 0} \mathcal{I}_{\mathcal{P}}.
\end{equation}

In order to check \eqref{eq:reg of conv product}, let us apply inequality \eqref{sy lemma bound} to the increment $\varXi^\tau$ defined in Equation \eqref{Riemann sums2}. To this aim, observe that taking $v=t$  and $u=s$  in the definition of $\varXi^\tau$ we get $\delta_s\bz^{2,\tau}_{ts}=0$, and thus $\varXi_{ts}^\tau=\bz^{2,\tau}_{ts}\otimes y_{s}^{s,s} $. In addition, we have just seen in~(\ref{four sixty six}) that $\mathcal{I}\left(\varXi^\tau\right)_{ts}=\bz_{ts}^{2,\tau}\ast y^{\cdot_{1},\cdot_{2}}$, and thus 
\begin{equation}\nonumber
\mathcal{I}\left(\varXi^\tau\right)_{ts}-\varXi_{ts}^\tau=\bz_{ts}^{2,\tau}\ast y^{\cdot_{1},\cdot_{2}}-\bz^{2,\tau}_{ts}\otimes y_{s}^{s,s}.
\end{equation}
Our claim (\ref{eq:reg of conv product}) is then a direct application of Lemma \ref{lem:(Volterra-sewing-lemma)}, together with the inequality estimates (\ref{reg z2 con y}) and (\ref{reg con con y}). 
\end{proof}

\begin{rem}\label{zaa}
The general convolution $\bz_{ts}^{2,\tau}\ast y^{\cdot_{1},\cdot_{2}}_{s}$ given in (\ref{def of conv}), for a path $y$ defined on  $\Delta_{3}$, will be invoked for our rough path constructions in the remainder of the article. If one wishes to consider the convolution restricted to a path $y^{\cdot}_{u}$ defined on $\Delta_{2}$, a natural way to proceed is to define 
\begin{equation*}
\bz^{2,\tau}_{ts}\ast y^{\cdot}_s:= \bz^{2,\tau}_{ts}\ast \hat{y}^{\cdot_1,\cdot_2}_s,
\quad\text{with}\quad
\hat{y}_{s}^{r_1,r_2}=y_{s}^{r_2}.
\end{equation*}
  Namely the path $\hat{y}$ has no dependence in $r_1$. We let the patient reader check the norm identity $\|\hat{y}^{\cdot_1,\cdot_2}\|_{(\alpha,\gamma),1,2}\simeq \|y\|_{(\alpha,\gamma),1,2}$, where $\|\hat{y}^{\cdot_1,\cdot_2}\|_{(\alpha,\gamma),1,2}$ is given as in \eqref{upper vb norm}  and  $\|y\|_{(\alpha,\gamma),1,2}$ is introduced in \eqref{eq: two variable  function norm change}.
\end{rem}

\begin{rem}
As a special case of Remark \ref{zaa}, we can define  the convolution $\bz^{2,\tau}_{tu}\ast \bz^{1,\cdot}_{us}$ by setting $y^{r}_{u}=\bz^{1,r}_{us}$. Then $y$ trivially satisfies $\|y\|_{(\alpha,\gamma),1,2}<\infty$ if $\bz^{1}\in\mathcal{V}^{(\alpha,\gamma)}$, which ensures a proper definition of $\bz^{2,\tau}_{tu}\ast \bz^{1,\cdot}_{us}$. Moreover, a direct application of  Theorem \ref{thm:genreal volterra convolution } yields  
\begin{equation}\nonumber
|\bz^{2,\tau}_{tu}\ast\bz^{1,\cdot}_{us}|\lesssim |\tau-t|^{-\gamma}|t-s|^{3\rho+\gamma}\wedge |\tau-s|^{3\rho}.
\end{equation}
\end{rem}

\begin{rem}\label{rem:extn of conv}
In our applications to rough Volterra equations we will consider the case $\rho=\alpha-\gamma\in(1/3,1/2]$, and therefore it is sufficient to show that the convolution product $\ast$ can be performed on the first and second level of a Volterra rough path. Indeed, whenever $\rho>1/3$, the convolution product for third or higher order terms in the Volterra rough path are of regularity $3\rho$ which is greater than $1$. Therefore the higher order convolutions $\bz^{n,\tau}$ introduced in (\textbf{H2}) may be constructed as a classical Riemann integral. 
For a general $\rho\in(0,1)$, it is easily conceived that one could extend the construction
 of the convolution product given in Theorem \ref{thm:genreal volterra convolution }
  to any order Volterra rough path $\bz^{n}$ satisfying (\ref{eq: hyp delta rel}).
   This can be done by induction on $n$, and one first need to give a proper definition of the convolution product up to order $k=[1/\rho]$. The convolution product between elements $\bz^{K}$ of order $K\geq k+1$ is then constructed canonically through Riemann integration, together with (\ref{eq: hyp delta rel}). We defer this extension to a the forthcoming paper \cite{HTW}.
\end{rem}

\subsection{Volterra convolutional functionals}

With the preliminary notions of Section~\ref{sec:conv-prod-rough} in hand, we are now ready to generalize the notion of multiplicative functional (as introduced by Lyons et. al.  in \cite{LyonsLevy}) to a Volterra context. The basic definition of Volterra convolutional functional is the following.

\begin{defn}
\label{def:Volterra convolutional functional}
Let $n\geq1$, and recall that $T^{(n)}=T^{(n)}(E)$ has been introduced in Definition~ \ref{def:truncated-algebra}. We consider a continuous map
\begin{equation*}
\bz:\Delta_{3}\rightarrow T^{\left(n\right)},
\qquad
\left(s,t,\tau\right)\mapsto{\bf z}_{ts}^{\tau}=\left(1,\bz_{ts}^{1,\tau},\ldots ,\bz_{ts}^{n,\tau}\right).
\end{equation*}
We call this mapping a Volterra convolutional functional if for all $\left(s,u,t,\tau\right)\in\Delta_{3}$ it satisfies
\[
{\bf z}_{ts}^{\tau}={\bf z}_{tu}^{\tau}\ast{\bf z}_{us}^{\cdot},
\]
  where for all $1\leq p\leq n$  the convolution product $\left({\bf \bz}_{tu}^{\tau}\ast{\bf z}_{us}^{\cdot}\right)^{p}$ is defined by 
\begin{equation}\label{p-step conv}
\left({\bf \bz}_{tu}^{\tau}\ast{\bf z}_{us}^{\cdot}\right)^{p}=\sum_{i=0}^{p}\bz_{tu}^{p-i,\tau}\ast \bz_{us}^{i,\cdot},
\end{equation}
 and where the convolution in the right hand side of (\ref{p-step conv}) is understood as in \eqref{convolution in 1d} or~\eqref{def of conv}. 
\end{defn}
\begin{rem}
In order to define \eqref{p-step conv}, we need in fact an extension \eqref{def of conv} to higher order integrals of the form $\bz^{j,\tau}$.  As mentioned in Remark \ref{rem:extn of conv}, we defer this extension to a future article. Notice however that, thanks to our restriction to $\rho>\frac{1}{3}$, we mostly need $p-i$ and $i$ $\leq 2$ in \eqref{p-step conv}. This case is covered by \eqref{convolution in 1d} or \eqref{def of conv}. 
\end{rem}

Proceeding as in \cite{LyonsLevy}, we will now define some H\"older type norms adapted to our Volterra multiplicative functionals. 

\begin{defn}
\label{def:Volterra Rough path} For $\alpha,\gamma\in\left(0,1\right)$
with $\rho:=\alpha-\gamma>0$, consider a Volterra convolutional functional
$\bz$ of degree $n=\lfloor \rho^{-1}\rfloor $
 as given in Definition \ref{def:Volterra convolutional functional}.
Let us assume that for $1\leq j\leq n$ the component $\bz^{j}$ of $\bz$ satisfies $\bz^{j}\in \mathcal{V}^{(j\rho+\gamma,\gamma)}_{2}$ where the space $\mathcal{V}^{(\alpha,\gamma)}$ has been introduced in Definition \ref{holder norms}. In addition we suppose that 
\begin{equation}\label{c1c}
\|\bz^{j}\|_{(j\rho+\gamma,\gamma),1} \lesssim\frac{M^j}{\Gamma\left(j\rho+1\right)} \,\,\,\, \mathrm{and} \,\,\,\, \|\bz^{j}\|_{(j\rho+\gamma,\gamma),1,2} \lesssim\frac{M^j}{\Gamma\left(j\rho+1\right)},
\end{equation}
for all $1\leq j\leq n$, 
where $M$ is a constant such that  $\|z^{1}\|_{(\alpha,\gamma)} \leq M$. 
 Then we say that $\bz$ is a Volterra rough path, and we denote
the space of Volterra rough paths of regularity $\left(\alpha,\gamma\right)$
by $\mathcal{\mathscr{V}}^{\left(\alpha,\gamma\right)}\left(\Delta_{2}\left(\left[0,T\right]\right);E\right).$ 
\end{defn}

\begin{rem}
All rough paths (in the classical framework recalled in Section \ref{subsec:classic-rough-paths}) are also Volterra rough
paths with Volterra kernel $k=1$,  i.e. $x_{t}^{1}:=\int_{0}^{t}1dx_{r}$.
Thus the definition of Volterra rough paths is truly extending the
definition of a rough path, and the convolutional product $\ast$
is extending the usual truncated tensor product by coupling the product
through the integration of kernels. 
\end{rem}

By definition we can see that a Volterra rough path is a continuous
mapping from $\Delta_{3}\left(\left[0,T\right]\right)$
to $T^{(\lfloor \rho^{-1}\rfloor )}\left(E\right)$.
We will also find it useful to equip the space with a metric generalizing (\ref{eq:def-metric-rough}). Let
us therefore define a metric for two Volterra rough paths $\bz$ and
$\mathbf{y}$ in $\mathscr{V}^{\left(\alpha,\gamma\right)}$ where $\rho=\alpha-\gamma$
by \begin{equation}
d_{\left(\alpha,\gamma\right)}\left(\bz,\mathbf{y}\right)=|z_{0}-y_{0}|+\sum_{m=1}^{\lfloor \rho^{-1}\rfloor }\| z^{m}-y^{m}\|_{\left(m\rho+\gamma,\gamma\right)}.\label{eq:rough path metric}
\end{equation}

\begin{defn}
We define the space of geometric Volterra paths as the closure of
smooth Volterra paths (i.e. paths in $\mathcal{V}^{\left(1,\gamma\right)}$)
in the rough path metric from equation \eqref{eq:rough path metric}.
The space of all geometric Volterra rough paths is denoted by $\mathcal{\mathscr{GV}}^{\left(\alpha,\gamma\right)}$. 
\end{defn}

\begin{rem}
Note that the geometric Volterra paths are not contained in a free-nilpotent Lie group, as is the case for regular rough paths. Indeed, there exists no concept of integration by parts in general for Volterra paths due to the possible singularities, and thus the notion of geometric Volterra paths can not be seen as an object in the space $G^{(l)}$ given in Definition~\ref{def:free-group}. 
\end{rem}
The following is an equivalent of the extension theorem for multiplicative functionals to a Volterra context. It can also be seen as an extension of Proposition \ref{prop:gamma bound smooth functionals}  and Proposition~\ref{prop:convolutional Chens relation smooth} to a rough context. 

\begin{thm}\label{extension}
Let $n=\lfloor \rho^{-1}\rfloor $ for $\rho=\alpha-\gamma>0$
and assume that $\mathbf{z}\in\mathscr{V}^{\left(\alpha,\gamma\right)}$
is an $n$-th order Volterra rough path with values in $T^{\left(n\right)}\left(E\right)$ according to Definition \ref{def:Volterra Rough path}.
Then there exists a unique extension of $\mathbf{z}$ to $T\left(E\right)$.
In particular, for all $m\geq n+1$ there exists a unique element $\mathbf{z}^{m}\in E^{\otimes m}$
such that for any $u\in[s,t]$ the following algebraic property is satisfied
\begin{equation}
\mathbf{z}_{ts}^{m,\tau} =\sum_{i=0}^{m}\mathbf{z}_{tu}^{m-i,\tau}\ast\mathbf{z}_{us}^{i,\cdot},\label{eq:abc}\\
\end{equation}
where we have used the convention $\bz^0 \equiv 1$ and  $\bz^j\ast 1=1\ast \bz^j =\bz^j$. 
In addition the bound~(\ref{c1c}) can be extended to $\bz$. Namely for $m\geq n+1$ we have for a constant $M>0$  such that  $\|\bz^{1}\|_{(\alpha,\gamma)} \leq M$ the following properties
\begin{equation}\label{cde}
\|\bz^{m}\|_{(m\rho+\gamma,\gamma),1} \lesssim\frac{M^m}{\Gamma\left(m\rho+1\right)}, \,\,\,\, \mathrm{and} \,\,\,\, \|\bz^{m}\|_{(m\rho+\gamma,\gamma),1,2} \lesssim\frac{M^m}{\Gamma\left(m\rho+1\right)},
\end{equation}
for any $\beta\in[0,1]$. 
It follows that there exists a unique Volterra signature with respect to the $n$-th order Volterra rough path. 
\end{thm}

\begin{proof}

We will divide the proof into several steps.

\noindent
\textit{Step 1: Uniqueness.}
The uniqueness problem will be addressed by induction. Indeed, for $m=n+1$  relation (\ref{eq:abc}) reads 
\beq\label{equiv delta abc}
\delta_{u}\bz^{\tau}_{ts}=\sum_{i=1}^{n}\mathbf{z}_{tu}^{n+1-i,\tau}\ast\mathbf{z}_{us}^{i,\cdot}.
\eeq
The right hand side of (\ref{equiv delta abc}) only depends on the stack $\{\bz^{j}|\,1\leq j\leq n\}$, and is therefore uniquely defined thanks to our assumptions. 
Now consider  $\tilde{\mathbf{z}}^{m}$ and
$\bar{\mathbf{z}}^{m}$ two candidates for  $\mathbf{z}^{m}$ with $m=n+1$, and 
define $\psi_{ts}^{\tau}=\tilde{\mathbf{z}}_{ts}^{m,\tau}-\bar{\mathbf{z}}_{ts}^{m,\tau}$. Then according to (\ref{equiv delta abc}) and relation~\eqref{cde} we have 
\beq\label{prop psi}
\delta \psi^{\tau}=0 , \,\,\, \quad\mathrm{and}\quad\,\,\,|\psi_{ts}^{\tau}|\lesssim|\tau-s|^{-\gamma}|t-s|^{(n+1)\rho+\gamma}.
\eeq
In particular $\psi$ is an additive functional with regularity greater than $1$. It is thus readily seen that $\psi=0$, which proves the uniqueness for $m=n+1$.  Once the uniqueness is  shown for the levels  $k=n+1,\ldots,m$, an induction procedure similar to what lead to~\eqref{prop psi} 	also shows uniqueness for $k=m+1$.

\noindent
\textit{Step 2: Existence.}
The existence will be proved again based on induction. We will first show that an $(m=n+1)$-th order Volterra rough path  can be constructed purely based on the information of $\bz^1,\ldots,\bz^n$. To this aim note that if there exists
a lift $\mathbf{z}^{m}$, then it must satisfy for any partition $\text{\ensuremath{\mathcal{P}} of \ensuremath{\left[s,t\right]} }$
\begin{equation}\label{g1}
\mathbf{z}_{ts}^{m,\tau}
=\sum_{\left[u,v\right]\in\mathcal{P}} 
\lp \mathbf{z}_{vu}^{m,\tau}+\delta_{u}\mathbf{z}_{vs}^{m,\tau}\rp.
\end{equation}
We will now take limits in \eqref{g1} as $|\mathcal{P}|\rightarrow0$.

To this aim, notice that according to (\ref{cde})  we have 
 $$|\mathbf{z}_{vu}^{m,\tau}|\lesssim|\tau-u|^{-\gamma}|v-u|^{m\rho+\gamma},$$
Hence, since  $m\rho>1$ we easily check that 
$
\lim_{|\mathcal{P}|\rightarrow0}\sum_{\left[u,v\right]\in\mathcal{P}}\mathbf{z}_{vu}^{m,\tau}=0.
$
In particular we obtain 
\begin{equation}
L_{ts}^{m,\tau}\equiv
 \lim_{|\mathcal{P}|\rightarrow0}\sum_{\left[u,v\right]\in\mathcal{P}} 
\lp \mathbf{z}_{vu}^{m,\tau}+\delta_{u}\mathbf{z}_{vs}^{m,\tau}\rp 
=
\ensuremath{\lim_{|\mathcal{P}|\rightarrow0}\sum_{\left[u,v\right]\in\mathcal{P}}\delta_{u}\mathbf{z}_{vs}^{m,\tau}}.
\label{eq:m order}
\end{equation}
In addition $\bz^{m,\tau}$ is required to satisfy (\ref{eq:abc}). Thus recalling that $m=n+1$ we have  
\begin{equation}\label{varxi def 478}
L_{ts}^\tau=\lim_{|\mathcal{P}|\rightarrow 0}\sum_{[u,v]\in\mathcal{P}}\varXi_{vu}^{\tau}
\,\,\,\quad\mathrm{where}\,\,\,\quad
\varXi_{vu}^\tau=\sum_{i=1}^{n}\mathbf{z}_{vu}^{m-i,\tau}\ast\mathbf{z}_{us}^{i,\cdot}. 
\end{equation}
Our strategy is now to prove that $L_{ts}^{\tau}$ exists by applying the Sewing Lemma \ref{lem:(Volterra-sewing-lemma)} to the increment $\varXi$. The main assumption to check in order to apply Lemma \ref{lem:(Volterra-sewing-lemma)} concerns $\delta\varXi$, and thus we obtain
\begin{equation}\label{472}
|\delta_{r}\varXi_{vu}^\tau|
=
\left|\sum_{i=1}^{n}\mathbf{z}_{vr}^{m-i,\tau}\ast\mathbf{z}_{ru}^{i,\cdot}\right|
\lesssim M^m\sum_{i=1}^{n}\frac{|\tau-r|^{-\gamma}|v-r|^{\left(m-i\right)\rho+\gamma}|r-u|^{i\rho}}{\Gamma\left(\left(m-i\right)\rho+1\right)\Gamma\left(i\rho+1\right)},
\end{equation}
where the first identity is obtained thanks to an elementary computation of $\delta_{r}(\mathbf{z}_{vu}^{m-i,\tau}\ast\mathbf{z}_{us}^{i,\cdot})$. Also note that the second inequality in (\ref{472}) directly stems from the assumption~(\ref{c1c}), which stipulates that 
\begin{equation}
\|\bz^j\|_{(j\rho+\gamma,\gamma),1}\leq M^j\Gamma(j\rho+1)^{-1}.
\end{equation} 

One can improve \eqref{472} in the following way: applying the neo-classical inequality from \cite[Lemma 3.8]{LyonsLevy}, we
know that there exists a $C>0$ such that 
\[
\sum_{i=1}^{m-1}\frac{|v-r|^{\left(m-i\right)\rho+\gamma}|r-u|^{i\rho}}{\Gamma\left(\left(m-i\right)\rho+1\right)\Gamma\left(i\rho+1\right)}\leq C\frac{|v-u|^{m\rho+\gamma}}{\Gamma\left(m\rho+1\right)}.
\]
Plugging this information into \eqref{472}, we conclude that $\delta\varXi$ satisfies 
\begin{equation}\label{neo classic delta varxi}
|\delta_{r}\varXi_{vu}^\tau|\lesssim M^m \frac{|\tau-v|^{-\gamma}|v-u|^{m\rho+\gamma}}{\Gamma\left(m\rho+1\right)}.
\end{equation}
With (\ref{neo classic delta varxi}) in hand, we can apply Lemma \ref{lem:(Volterra-sewing-lemma)} to the increment $\varXi $. We get that the limit $L_{ts}^{\tau}$  defined by (\ref{eq:m order}) exists, and we set $\mathcal{I}\left(\varXi^\tau\right)_{ts}=L_{ts}^{\tau}=\bz^{m,\tau}_{ts}$ for $m=n+1$. Moreover, a direct application of (\ref{sy lemma bound}) together with the fact that $\varXi_{ts}^\tau=0$ yield 
\begin{equation}\label{bbbbb}
|\mathbf{z}_{ts}^{m,\tau}|\lesssim M^m \frac{\left(|\tau-t|^{-\gamma}|t-s|^{m\rho+\gamma}\right)\wedge |\tau-s|^{m\rho}}{\Gamma\left(m\rho+1\right)}.
\end{equation}
It now follows that \beq\label{b1b1b}
\|\bz^m\|_{(m\rho+\gamma,\gamma),1}\lesssim M^m \Gamma(m\rho+1)^{-1}.
\eeq 
We also let the patient reader check that a simple induction procedure allows to generalize  all our considerations until (\ref{b1b1b}) for a generic $m\geq n+1$.

We will now prove that 
\beq\label{relation d}
\|\bz^m\|_{(m\rho+\gamma,\gamma),1,2}\lesssim M^m \Gamma(m\rho+1)^{-1}.
\eeq
To this aim, we need to repeat the procedure of Steps 1-2 for $\bz^{m,\tau\tau^\prime}_{ts}=\bz^{m,\tau}_{ts}-\bz^{m,\tau^\prime}_{ts}$. In particular, the equivalent  of the incremental $\varXi^\tau$ defined in (\ref{varxi def 478}) will be 
\beq\nonumber
\varXi^{\tau\tau^\prime}_{ts}=\sum_{i=1}^{m-1}\mathbf{z}_{vu}^{m-i,\tau\tau^\prime}\ast\mathbf{z}_{us}^{i,\cdot}.
\eeq
With this increment in hand, relation \eqref{relation d} is proved along the same lines as \eqref{b1b1b}. Details are omitted for the sake of conciseness. The norm $\|\bz^m\|_{(m\rho+\gamma,\gamma),2}$ can also be estimated with the same kind of argument. Hence gathering (\ref{b1b1b}) and (\ref{relation d}), we have obtained  that $\bz^m\in \mathcal{V}^{(m\rho+\gamma,\gamma)}$ where $\mathcal{V}^{(\alpha,\gamma)}$ is given in Definition  \ref{holder norms}. 

\noindent
\textit{Step 3: Convolutional property.}
It remains to be proven that $\mathbf{z}^{m}$ is a convolutional
functional in terms of Definition \ref{def:Volterra convolutional functional}, i.e. that for $m\geq n+1$  and $(s,r,t,\tau)\in \Delta_4$ it satisfies 
\begin{equation}\label{prove conv relation}
\mathbf{z}_{ts}^{m,\tau}=\sum_{i=0}^{m}\mathbf{z}_{tr}^{m-i,\tau}\ast\mathbf{z}_{rs}^{i,\cdot}.
\end{equation}
In order to prove identity (\ref{prove conv relation}), recall that (\ref{varxi def 478}) can be read as 
\beq\nonumber
\mathbf{z}_{ts}^{m,\tau}  =\lim_{|\mathcal{P}|\rightarrow0}\sum_{\left[u,v\right]\in\mathcal{P}}\ensuremath{\sum_{i=1}^{m-1}\mathbf{z}_{vu}^{m-i,\tau}\ast\mathbf{z}_{us}^{i,\cdot}}.
\eeq
Let us now divide a typical partition  $\mathcal{P}$  into $\mathcal{P}\cap [s,r]$ and $\mathcal{P}\cap [r,t]$. This yields  
\begin{align}\nonumber
\mathbf{z}_{ts}^{m,\tau} & =\lim_{|\mathcal{P}|\rightarrow0}\sum_{\left[u,v\right]\in\mathcal{P}\cap\left[s,r\right]}\ensuremath{\sum_{i=1}^{m-1}\mathbf{z}_{vu}^{m-i,\tau}\ast\mathbf{z}_{us}^{i,\cdot}}+\lim_{|\mathcal{P}|\rightarrow0}\sum_{\left[u,v\right]\in\mathcal{P}\cap\left[r,t\right]}\ensuremath{\sum_{i=1}^{m-1}\mathbf{z}_{vu}^{m-i,\tau}\ast\mathbf{z}_{us}^{i,\cdot}}
\\\label{eq:last in algn thm 39}
 &=\bz^{m,\tau}_{rs}+\hat{L}_{trs}^{\tau},
\end{align}
where we have invoked (\ref{varxi def 478}) again for the second identity and where we have set 
\beq\nonumber
\hat{L}_{trs}^{\tau}=\lim_{|\mathcal{P}|\rightarrow0}\sum_{\left[u,v\right]\in\mathcal{P}\cap\left[r,t\right]}\ensuremath{\sum_{i=1}^{m-1}\mathbf{z}_{vu}^{m-i,\tau}\ast\mathbf{z}_{us}^{i,\cdot}}.
\eeq
As in the previous steps we now proceed by induction. Namely assume that \eqref{eq:abc} holds for $k=1,\ldots,m-1$, and let us propagate the relation until $k=m$. Then applying the identity $\mathbf{z}_{ts}^{i,\tau} =\sum_{j=0}^{i}\mathbf{z}_{tu}^{i-j,\tau}\ast\mathbf{z}_{us}^{j,\cdot}$, which is valid for all $l<m$,  we get 
\beq\label{hatL division}
\hat{L}^{\tau}_{trs}=\hat{L}^{1,\tau}_{trs}+\hat{L}^{2,\tau}_{trs},
\eeq
where we define 
\begin{align}\nonumber
\hat{L}_{trs}^{1,\tau}&=\lim_{|\mathcal{P}|\rightarrow0}\sum_{\left[u,v\right]\in\mathcal{P}\cap\left[r,t\right]}\ensuremath{\sum_{i=1}^{m-1}\mathbf{z}_{vu}^{m-i,\tau}\ast\mathbf{z}_{ur}^{i,\cdot}}
\\\nonumber
\hat{L}_{trs}^{2,\tau}&=\lim_{|\mathcal{P}|\rightarrow0}\sum_{\left[u,v\right]\in\mathcal{P}\cap\left[r,t\right]}\ensuremath{\sum_{i=1}^{m-1}\mathbf{z}_{vu}^{m-i,\tau}\ast\left[\sum_{j=1}^{i}\mathbf{z}_{ur}^{i-j,\cdot}\ast\mathbf{z}_{rs}^{j,\cdot}\right]}.
\end{align}
Next, another application of (\ref{varxi def 478}) enables us to obtain directly 
\beq\label{hatL1}
\hat{L}_{trs}^{1,\tau}=\bz^{m,\tau}_{tr}.
\eeq 
In order to handle the term $\hat{L}^{2,\tau}_{trs}$, let us change the order of the sums with respect to $i,j$ and invoke the associativity of the convolution product $*$. We get
\begin{align}
\hat{L}_{trs}^{2,\tau}=&\lim_{|\mathcal{P}|\rightarrow0}\sum_{\left[u,v\right]\in\mathcal{P}\cap\left[r,t\right]}
\sum_{j=1}^{m-1}\sum_{i=j}^{m-1}\mathbf{z}_{vu}^{m-i,\tau}\ast\mathbf{z}_{ur}^{i-j,\cdot}\ast\mathbf{z}_{rs}^{j,\cdot} \notag
\\
\label{foruninty six}
=&\sum_{j=1}^{m-1} \left[\lim_{|\mathcal{P}|\rightarrow0}\sum_{\left[u,v\right]\in\mathcal{P}\cap\left[r,t\right]}
\sum_{i=j}^{m-1}\mathbf{z}_{vu}^{m-i,\tau}\ast\mathbf{z}_{ur}^{i-j,\cdot}\right]\ast\mathbf{z}_{rs}^{j,\cdot}.
\end{align}
Now an elementary change of variable and \eqref{eq:abc} yield
\begin{equation*}
\sum_{i=j}^{m-1}\mathbf{z}_{vu}^{m-i,\tau}\ast\mathbf{z}_{ur}^{i-j,\cdot}
=
\sum_{k=0}^{m-1-j}\mathbf{z}_{vu}^{m-j-k,\tau}\ast\mathbf{z}_{ur}^{k,\cdot}
=
\mathbf{z}_{vr}^{m-j,\tau}-\mathbf{z}_{ur}^{m-j,\tau}
=
\mathbf{z}_{vu}^{m-j,\tau} +\delta_{u}\mathbf{z}_{vr}^{m-j,\tau}.
\end{equation*}
Plugging this information into \eqref{foruninty six} and invoking \eqref{g1}, we end up with
\begin{equation}\label{dddd}
\hat{L}_{trs}^{2,\tau}
=
\sum_{j=1}^{m-1} 
\left[\lim_{|\mathcal{P}|\rightarrow0}\sum_{\left[u,v\right]\in\mathcal{P}\cap\left[r,t\right]}
\mathbf{z}_{vu}^{m-j,\tau} +\delta_{u}\mathbf{z}_{vr}^{m-j,\tau} 
\right]\ast\mathbf{z}_{rs}^{j,\cdot}
=
\sum_{j=1}^{m-1} \bz^{m-j,\tau}_{tr} \ast\mathbf{z}_{rs}^{j,\cdot}.
\end{equation}

Let us summarize our considerations so far: gathering (\ref{dddd}) and (\ref{hatL1}) into (\ref{hatL division}), and then inserting (\ref{hatL division}) into (\ref{eq:last in algn thm 39}) we have obtained that 
\beq\nonumber
 \bz^{m,\tau}_{ts}
 =\bz^{m,\tau}_{tr}+\bz^{m,\tau}_{rs}+\sum_{j=1}^{m-1} \mathbf{z}_{tr}^{m-i,\tau}\ast\mathbf{z}_{rs}^{i,\cdot}
 =
 \sum_{j=0}^{m} \mathbf{z}_{tr}^{m-i,\tau}\ast\mathbf{z}_{rs}^{i,\cdot}.
\eeq
This concludes our induction procedure, and thus (\ref{eq:abc}) holds for all $m\geq 1$.
\end{proof}

\begin{rem} 
Theorem \ref{extension} tells us that the Volterra signature associated to a Volterra path is uniquely determined from the Volterra rough path introduced in Definition~\ref{def:Volterra Rough path}. That is, once we have constructed a truncated Volterra rough path (remember that this object is by no means unique) then there exists a unique extension with respect to the full Volterra rough path.  
\end{rem}

\section{Non-linear Volterra integral equations driven by rough noise}
In this section we will see how we can substitute the conventional tensor product from rough path theory with the convolution product defined in Section \ref{sec:Volterra-Rough-Paths} in order to show existence and uniqueness of Volterra equations with singular kernels. Similarly to the theory of controlled rough path introduced by Gubinelli in \cite{Gubinelli}, we define a class Volterra controlled paths. The composition of the Volterra controlled paths with the Volterra rough path from  Definition~\ref{def:Volterra Rough path} gives an abstract Riemann integrand such that we may construct a Volterra integral by application of the Volterra Sewing Lemma \ref{lem:(Volterra-sewing-lemma)}. This abstract integration step is then the key in order to define and solve Volterra type equations.

\subsection{Volterra controlled processes and rough Volterra integration}\label{non-linear setting}
  
As many of the results here are extensions of classical texts on rough path such as \cite{FriHai}  or \cite{Gubinelli}, we will try to keep the proofs as concise as possible. The reader is sent to the aforementioned  references for further information on the results and properties of controlled rough paths and solutions to non-linear differential equations driven by rough paths. 
We will first give a definition of another modification of the Volterra-H\"older spaces given in Definition \ref{holder norms} in order to give a precise analysis of Volterra-controlled paths. 

\begin{defn}\label{def:mod Volterra holder}
Let  $\mathcal{W}_{2}^{\left(\alpha,\gamma\right)}$  denote the space of functions $u:\Delta_3\rightarrow V  $ such that $(p,q,s)\mapsto u^{p,q}_s \in V$ and 
\begin{equation}\label{lllll}
\|u^{\cdot_1,\cdot_2}\|_{(\alpha,\gamma)}:=\|u^{\cdot_1,\cdot_2}\|_{(\alpha,\gamma),1}+\|u^{\cdot_1,\cdot_2}\|_{(\alpha,\gamma),1,2} <\infty 
\end{equation}
 where we define the norm (recall the convention $\rho=\alpha-\gamma$ below)
\begin{align}\label{eq:derivative norm1}
\|u^{\cdot_1,\cdot_2}\|_{(\alpha,\gamma),1}
:=
\sup_{(s,t,\tau)\in\Delta_3}\frac{\left|u_{ts}^{\tau,\tau}\right|}{|\tau-t|^{-\gamma}|t-s|^{\alpha}\wedge|\tau-s|^{\rho}},
\end{align}
and the norm $\|u^{\cdot_1,\cdot_2}\|_{(\alpha,\gamma),1,2}$ is given as in Definition \ref{def: upper vb space}. 
\end{defn}
\begin{rem}
Note in particular that the definition of the  space  $\mathcal{W}_{2}^{\left(\alpha,\gamma\right)}$ does not involve a norm similar to \eqref{eq: two variable  function norm change}.  Although the definition of $\|u^{\cdot_1,\cdot_2}\|_{(\alpha,\gamma)}$ is a slight abuse of notation, we believe that it will be clear from the superscripts of $u$ what norm we apply.   
\end{rem}
We now turn to the definition of controlled Volterra paths, which is crucial for a proper definition  of rough Volterra equations.

\begin{defn}\label{Volterra controlled path}
Let $z\in\mathcal{V}^{\left(\alpha,\gamma\right)}\left(E\right)$
for some $\rho=\alpha-\gamma>0$. We assume that there exists two functions
$y:\Delta_{2}\rightarrow V$ and $y^{\prime}:\Delta_{3}\rightarrow\mathcal{L}\left(E,V\right),$ such that $y^\tau_0=y_0\in  E$  for  any $\tau\in [0,T]$ and $y^{\prime,p,q}_0=y^\prime_0\in E$ for any $(q,p)\in\Delta_2$, and satisfying the relation 
\begin{equation}\label{controlled rel}
y_{ts}^{\tau}=z_{ts}^{\tau} \ast y_{s}^{\prime,\tau,\cdot}+R_{ts}^{\tau},
\end{equation}
where $R\in\mathcal{V}_{2}^{\left(2\alpha,2\gamma\right)}\left(V\right)$ and  $y^\prime \in \mathcal{W}_{2}^{\left(\alpha,\gamma\right)}$.  (Recall that the spaces $\mathcal{V}_{2}^{\left(2\alpha,2\gamma\right)}$ and $\mathcal{W}_{2}^{\left(\alpha,\gamma\right)}$ are respectively introduced in Remark \ref{rem:a} and Definition \ref{def:mod Volterra holder}). 
Whenever $(y,y^\prime)$ satisfies relation~\eqref{controlled rel} we say that $\left(y,y^{\prime}\right)$ is a Volterra path controlled
by $z$ (or controlled Volterra path in general) and we write $\left(y,y^{\prime}\right)\in\mathscr{D}_{z}^{(\alpha,\gamma)}\left(\Delta_{2};V\right)$. We equip this space with a semi-norm $\|\cdot \|_{z,(\alpha,\gamma)}$  given by 
\begin{equation}\label{eq:controlled norm def}
\|y,y^{\prime}\|_{z,(\alpha,\gamma)}=\|y^{\prime,\cdot_1,\cdot_2}\|_{(\alpha,\gamma)} +\|R\|_{(2\alpha,2\gamma)}. 
\end{equation}
Under the mapping $(y,y^{\prime})\mapsto |y_{0}|+|y^{\prime}_{0}|+\|y,y^{\prime}\|_{z,(\alpha,\gamma)} $ the space $\mathscr{D}_{z}^{\left(\alpha,\gamma\right)}\left(\Delta_{2};V\right)$ is a Banach space. The remainder term $R$ in (\ref{controlled rel}) with respect to a Volterra path $(y,y^\prime)\in \mathscr{D}_{z}^{\left(\alpha,\gamma\right)}$ will typically be denoted by $R^y $. 
\end{defn}

\begin{rem}
We call the function $y^{\prime}$ the Volterra-Gubinelli derivative, and emphasize that this function is evaluated on $\Delta_{3}$, where it has \emph{two} upper arguments. This is denoted by $\Delta_{3}\ni (s,p,q)\mapsto y_{s}^{\prime,q,p}$  as opposed to the increment of a path $y$ in the upper variable denoted by $\Delta_{3}\ni (s,p,q)\mapsto y_{s}^{qp}$ .
\end{rem}

\begin{rem}\label{rem: inherited reg}
For a controlled Volterra path the regularity of $y$ in the upper
argument is inherited from the regularity of the upper argument of
the driving noise $z$, Gubinelli derivative and remainder term $R^y$. That is, it is implied from relation (\ref{controlled rel}) that for $\left(y,y^{\prime}\right)\in\mathscr{D}_{z}^{\left(\alpha,\gamma\right)}\left(\Delta_{2};V\right)$
we have 
\begin{equation}\label{inherit reg of controled path}
y_{ts}^{qp}=z^{qp}_{ts}\ast y_s^{\prime,p,\cdot_2}+z^{q}_{ts} \ast y_s^{\prime,qp,\cdot} +R_{ts}^{qp}.
\end{equation}

\end{rem}
Our next step is to show that we may construct the Volterra rough integral in a very similar way to the classical rough path integral, but changing $\otimes$ for $\ast$ as well as applying the Volterra Sewing lemma \ref{lem:(Volterra-sewing-lemma)}. It follows that the Volterra integral of a controlled path with respect to a driving Hölder noise $x\in \mathcal{C}^{\alpha}$  is again a controlled Volterra path.   
\begin{thm}
\label{thm:Volterra integral of controlled path is controlled path}
Let $x\in\mathcal{C}^{\alpha}$ and $k$ be a Volterra kernel satisfying
$\left(\mathbf{H}\right)$ with a parameter $\gamma$ such that $\rho=\alpha-\gamma>\frac{1}{3}.$
Thanks to Theorem \ref{thm:Regularity of Volterra path}, define $z_{t}^{\tau}=\int_{0}^{t}k\left(\tau,r\right)dx_{r}$  and
assume there exists a second order Volterra rough path $\mathbf{z}\in\mathscr{V}^{(\alpha,\gamma)}\left(\Delta_{2};E\right)$
built from $z$ according to Definition \ref{def:Volterra Rough path}. Additionally, suppose  both components of $\bz $  are uniformly bounded. Namely, we assume there exists an $M>0$ such that 
\begin{equation}\label{eq:bz notation}
\|\bz\|_{(\alpha,\gamma)} := \|\bz^{1}\|_{(\alpha,\gamma)}+\|\bz^{2}\|_{(2\rho+\gamma,\gamma)}\leq M,
\end{equation}
where the two norm quantities corresponds to the norms given in Definition \ref{holder norms} and Remark~\ref{rem:a}.  
We now consider a controlled Volterra path $(y,y^\prime)\in \mathscr{D}_{\bz^{1}}^{(\alpha,\gamma)}(\Delta_2;\mathcal{L}(E,V))$. Then the following holds true: 

\begin{enumerate}[wide, labelwidth=!, labelindent=0pt, label=\emph{(\roman*)}]
\setlength\itemsep{.1in}

\item 
The following limit exists for all $(s,t,\tau)\in \Delta_3$, 
\begin{equation}\label{eq: def of V integral}
w_{ts}^{\tau}=\int_{s}^{t}k(\tau,r)y_{r}^{r}dx_{r}:=\lim_{|\mathcal{P}|\rightarrow 0} \sum_{[u,v]\in \mathcal{P}} \bz_{vu}^{1,\tau}\ast y_{u}^{\cdot}+\bz_{vu}^{2,\tau}\ast y_{u}^{\prime,\cdot_1,\cdot_2}. 
\end{equation}

\item
Let $w$ be defined by \eqref{eq: def of V integral}. There exists a constant $C=C_{M,\alpha,\gamma}$ such that for all $(s,t)\in \Delta_2$  we have 
\begin{multline}
\left|w_{ts}^\tau-\bz_{ts}^{1,\tau}\ast y_{s}^{\cdot}-\mathbf{z}_{ts}^{2,\tau}\ast y_{s}^{\prime,\cdot_1,\cdot_2}\right| \\
\leq C\|y,y^\prime\|_{z,(\alpha,\gamma)}\|\bz\|_{(\alpha,\gamma)} \left[|\tau-t|^{-\gamma}|t-s|^{3\rho+\gamma}\wedge |\tau-s|^{3\rho}\right].\label{eq:bound volterra ointegration}
\end{multline}

\item
For all $(s,t,p,q)\in \Delta_4$ and $\eta\in[0,1]$ and $\zeta\in [0,\rho)$  we have 
\begin{multline}\label{eq:upper arg reg}
\left|w_{ts}^{qp}-\bz_{ts}^{1,qp}\ast y_{s}^{\cdot}-\mathbf{z}_{ts}^{2,qp}\ast y_{s}^{\prime,\cdot_1,\cdot_2}\right|
\\
\leq C\|y,y^\prime\|_{z,(\alpha,\gamma)}\|\bz\|_{(\alpha,\gamma)}|p-q|^{\eta}|q-t|^{-\eta+\zeta}\left[|q-t|^{-\gamma-\zeta}|t-s|^{3\rho+\gamma}\wedge |q-s|^{3\rho-\zeta}\right].
\end{multline}

\item 
The couple $(w,w^\prime)$ is a controlled Volterra path in $\mathscr{D}_{\bz^{1}}^{(\alpha,\gamma)}(\Delta_2,V)$, where we recall that $w$ is defined by (\ref{eq: def of V integral}) and $w_t^{\prime,\tau,p}=y_{t}^{p}.$
\end{enumerate} 
\end{thm}

\begin{rem}
According to our computations (see in particular \eqref{a ineq} below) we believe that Theorem \ref{thm:Volterra integral of controlled path is controlled path} should hold true under the condition $3\rho+\gamma>1$ (vs. $3\rho>1$). We have sticked to the more restrictive assumption $3\rho>1$ in order to be compatible with Definition \ref{def:Volterra Rough path} for $n=2$. 
\end{rem}

\begin{proof}[Proof of Theorem \ref{thm:Volterra integral of controlled path is controlled path}]
We define $\varXi_{vu}^{\tau}=\bz_{vu}^{1,\tau}\ast y_{u}^{\cdot}+\mathbf{z}_{vu}^{2,\tau}\ast y_{u}^{\prime,\cdot_1,\cdot_2},$
where $\mathbf{z}_{vu}^{2,\tau}\ast y_{u}^{\prime,\cdot_1,\cdot_2}$ is
understood according to Theorem \ref{thm:genreal volterra convolution }. Namely, it is readily  checked, whenever $(y,y^\prime)\in \mathscr{D}^{(\alpha,\gamma)}_{\bz^1}$ that $y^\prime\in \mathcal{V}^{\cdot_1,\cdot_2}_{(\alpha,\gamma)}$ where $ \mathcal{V}^{\cdot_1,\cdot_2}_{(\alpha,\gamma)}$ is given in Definition \ref{def: upper vb space}. Therefore Theorem \ref{thm:genreal volterra convolution } enables to define 
\begin{equation}\nonumber
\mathbf{z}_{ts}^{2,\tau}\ast y_{s}^{\prime,\cdot_1,\cdot_2} =\lim_{|\mathcal{P}|\rightarrow 0}\sum_{[u,v]\in \mathcal{P}}\mathbf{z}_{vu}^{2,\tau}\otimes y_{s}^{\prime,u,u}+\delta_{u}\mathbf{z}_{vs}^{2,\tau}\ast y_{s}^{\prime,\cdot_1,\cdot_2}.
\end{equation}
Now that $\varXi$ is properly defined, our next step is to invoke lemma \ref{lem:(Volterra-sewing-lemma)} in order to define 
\begin{equation}\nonumber
w_{ts}^\tau=\int_{s}^{t}k\left(\tau,r\right)y_{r}^{r}dx_{r}=\mathcal{I}\left(\varXi\right)_{ts}.
\end{equation}
To this aim, similarly to the proof of Theorem \ref{extension}, we need to check that $\delta \varXi$ is sufficiently regular. This is what we proceed to do below in order to obtain (\ref{eq: def of V integral}).

We first compute $\delta\varXi^\tau$, where we recall that  $\varXi_{vu}^{\tau}=\bz_{vu}^{1,\tau}\ast y_{u}^{\cdot}+\mathbf{z}_{vu}^{2,\tau}\ast y_{u}^{\prime,\cdot_1,\cdot_2}$. That is, combining elementary algebraic properties of the operator $\delta$ and relation (\ref{p-step conv}) read for $p=1,2$ we get the following relation for $(u,m,v,\tau)\in \Delta_4$,
\begin{eqnarray}
\delta_{m}\varXi_{vu}^{\tau}=-\bz_{vm}^{1,\tau}\ast y_{mu}^{\cdot} -\bz_{vm}^{2,\tau}\ast y_{mu}^{\prime,\cdot_1,\cdot_2}+
\bz_{vm}^{1,\tau}\ast \bz_{mu}^{1,\cdot}\ast y^{\prime,\cdot,\cdot} .\label{eq:delta at controlled Xi}
\end{eqnarray}
Now we resort to the fact that $y$ satisfies (\ref{controlled rel}) in order to write
\[
\bz_{vm}^{1,\tau}\ast y_{mu}^{\cdot}=\bz_{vm}^{1,\tau}\ast \left(\bz_{mu}^{1,\cdot}\ast y_{u}^{\prime,\cdot_1,\cdot_2}\right)+\bz_{vm}^{1,\tau}\ast R_{mu}^{\cdot}.
\]
Plugging this into (\ref{eq:delta at controlled Xi}) we obtain
\begin{equation}\label{relation delta controlled}
\delta_{m}\varXi_{vu}^{\tau}
=-\bz_{vm}^{2,\tau}\ast y_{mu}^{\prime,\cdot_1,\cdot_2} -\bz_{vm}^{1,\tau}\ast R_{mu}^{\cdot}.
\end{equation}
Thanks to relation (\ref{relation delta controlled}), we can now analyze the regularity of $\delta\varXi^\tau$. Indeed,  invoking Theorem \ref{thm:genreal volterra convolution }  we get 
\begin{equation} \label{est1}
|\bz_{vm}^{2,\tau}\ast y_{mu}^{\prime,\cdot_1,\cdot_2}|\leq\| y^{\prime,\cdot_1,\cdot_2}\|_{\left(\alpha,\gamma\right),1,2}
 \| \bz^{2}\|_{\left(2\rho+\gamma,\gamma\right)}|u-m|^{\rho}|\tau-m|^{-\gamma}|v-m|^{2\rho+\gamma}
\end{equation}
and similarly 
\begin{equation}\label{est2}
|\bz_{vm}^{1,\tau}\ast R_{mu}^{\cdot}|\leq\| R\|_{\left(2\rho +\gamma,\gamma\right)}\| \bz^{1}\|_{\left(\alpha,\gamma\right)}|\tau-m|^{-\gamma}|v-m|^{\alpha}|u-m|^{2\rho}.
\end{equation}
Gathering  (\ref{est1}) and (\ref{est2}) into (\ref{relation delta controlled}) and recalling that $\tau>v>m>u$, we thus obtain that 
\begin{equation}\label{a ineq}
|\delta_{m}\varXi^\tau_{vu}|\lesssim \| y,y^\prime\|_{\bz^1,\left(\alpha,\gamma\right)}\| \bz\|_{\left(\alpha,\gamma\right)}|\tau-v|^{-\gamma}|v-u|^{3\rho+\gamma}. 
\end{equation}
 Since $3\rho+\gamma>1$, we can apply the Volterra Sewing Lemma \ref{lem:(Volterra-sewing-lemma)} and define $w_{ts}^\tau:=  \mathcal{I}(\varXi^\tau)_{ts}$. This achieves the proof of (\ref{eq: def of V integral}) and  relation (\ref{eq:bound volterra ointegration}).

Next, we shall prove Inequality (\ref{eq:upper arg reg}). Start to
set $\varXi_{ts}^{qp}=\bz_{ts}^{1,qp}\ast y_{s}^{\cdot}+\bz_{ts}^{2,qp}\ast y_{s}^{\prime,\cdot_1,\cdot_2}$, and
observe that by the exact same computations as above (remember that $u\mapsto\delta_{u}$
acts on the lower argument of a function $f_{t}^{\tau}$) we obtain
\[
\delta_{u}\varXi_{ts}^{qp}=-\bz_{vm}^{2,qp}\ast y_{mu}^{\prime,\cdot_1,\cdot_2}- \bz_{vm}^{1,qp}\ast R_{mu}^{\cdot}.
\]
Thus, the regularity $\delta_{u}\varXi_{ts}^{p,q}$ follows from the assumption \eqref{c1c} of regularity on the Volterra rough path $\bz$  and the controlled path $(y,y^\prime)$ together with  equivalent bounds as in (\ref{est1}) and (\ref{est2}), taking into account the increment in the upper parameters.   We  therefore obtain for $(s,u,t,p,q)\in \Delta_{5}$ and $\eta\in [0,1]$ and $\zeta\in [0,\rho)$ 
\begin{multline}
|\delta_{u}\varXi_{ts}^{qp}|\leq
\left(\| y^{\prime,\cdot_1,\cdot_2}\|_{\left(\alpha,\gamma\right),1,2}\|\mathbf{z}^{2}\|_{\left(2\rho+\gamma,\gamma\right)}+\| R\|_{1,\left(\alpha,\gamma\right)}\| \bz^{1}\|_{\left(\alpha,\gamma\right)}\right)
\\
\times|q-p|^{\eta}|p-t|^{-\eta+\zeta}|p-t|^{-\gamma-\zeta}|t-s|^{3\rho+\gamma}.
\end{multline}
Applying again the Volterra Sewing Lemma \ref{lem:(Volterra-sewing-lemma)}, we now easily  conclude that (\ref{eq:upper arg reg}) holds. 
\end{proof}
\begin{rem}
The definition of a controlled Volterra rough path tells
us that $y^{\prime}:\Delta_{3}\rightarrow V$, i.e.
it takes three ordered time variables as input. However,  the computations of Theorem \ref{thm:Volterra integral of controlled path is controlled path} reveal that when $\left(y,y^{\prime}\right)\in\mathscr{D}_{\bz^{1}}^{\left(\alpha,\gamma\right)}\left(\mathcal{L}\left(E,V\right)\right)$, the controlled derivative of $w_{t}^{\tau}=\int_{0}^{t}k\left(\tau,r\right)y_{r}^{r}dx_{r}$
only depends on two variables.  Specifically we have  $w_{t}^{\prime,\tau,q}=w_{t}^{\prime,q}\equiv y_{t}^{q}$
, which is seen from item (\rm{iv}) in Theorem \ref{thm:Volterra integral of controlled path is controlled path}. One can thus refine Theorem~\ref{thm:Volterra integral of controlled path is controlled path} and state that the Volterra rough integration sends $(y,y^\prime)\in \mathscr{D}_{z}^{(\alpha,\gamma)} $ to a controlled process $(w,w^\prime)\in \hat{ \mathscr{D}}_{z}^{(\alpha,\gamma)}$ where the space $\hat{ \mathscr{D}}_{z}^{(\alpha,\gamma)}$ is defined by 
\begin{equation}\label{eq:def of hat D}
\hat{ \mathscr{D}}_{z}^{(\alpha,\gamma)}(\Delta_2;V):= \{(w,w^\prime)\in \mathscr{D}_{z}^{(\alpha,\gamma)}\,|\, w_s^{\prime,\tau,p}= w_s^{\prime,p}\} .
\end{equation}
The space $\hat{ \mathscr{D}}_{z}^{(\alpha,\gamma)}$ will be used in the composition step below.  

\end{rem}

\begin{prop}\label{prop: composition with reg function}
Let $f\in\mathcal{C}_{b}^{3}\left(V\right)$ and assume $\left(y,y^{\prime}\right)\in\hat{\mathscr{D}}_{z}^{\left(\alpha,\gamma\right)}\left(V\right)$. 
Then the composition $(\varphi,\varphi^{\prime}):=\left(f\left(y\right),y^{\prime}f^\prime\left(y\right)\right)$ is a controlled Volterra path in $\mathscr{D}_{z}^{\left(\alpha,\gamma\right)}\left(V\right)$,
where the derivative $\varphi^{\prime}:\Delta_{3}\rightarrow V$
is given by 
\begin{equation}\label{eq:def of comp derivative}
\Delta_{3}\ni\left(t,p,q\right)\mapsto y_{t}^{\prime,p} f^\prime\left(y_{t}^{q}\right).
\end{equation}
Moreover, there exists a constant  $C=C_{M,\alpha,\gamma,\|f\|_{C^{3}_{b}}} >0$  such that 
\begin{equation}\label{bound for function composed with controlled path}
\|\varphi,\varphi^\prime\|_{z;(\alpha,\gamma)}\leq C \left(1+\|z\|_{(\alpha,\gamma)}\right)^{2}  \left[\left( |y^\prime_{0}|+\|y,y^{\prime}\|_{z,(\alpha,\gamma)}\right)\vee \left( |y^\prime_{0}|+\|y,y^{\prime}\|_{z,(\alpha,\gamma)}\right)^{2}\right] 
\end{equation}
 
\end{prop}

\begin{proof}
Let us first prove the algebraic part of the proposition, namely relation \eqref{eq:def of comp derivative}. We start to decompose the increment $f(y^q)_{ts}$ into 
\begin{equation}\nonumber
f(y^q)_{ts}=y_{ts}^\tau f^\prime(y^q_s)+\left[f(y^q)_{ts}-y_{ts}^\tau f^\prime(y^q_s)\right].
\end{equation}
We then resort to relation \eqref{controlled rel} in order to write 
\begin{equation}\nonumber
f(y^q)_{ts}=z_{ts}^q \ast y_s^{\prime,q,\cdot}f^\prime(y_s^q)+R^{f(y),q}_{ts},	
\end{equation}
where we have set $R^y$ to be the remainder of $y$ in \eqref{controlled rel}, and 
\begin{equation}\label{eq:Rf(y) relation}
R^{f(y),q}_{ts}=\left[f(y^q)_{ts}-y_{ts}^\tau f^\prime(y^q_s)\right]+ R^{y,q}_{ts} f^\prime(y^q_s).
\end{equation}
In addition, recalling that $(y,y^\prime)\in \hat{\mathcal{P}}^{(\alpha,\gamma)}$ the path $y^{\prime,q,\cdot}$ does not depend on $q$. Hence we get 
\begin{equation}\label{eq: corect Rf(y)}
f(y^q)_{ts}=z_{ts}^q \ast y_s^{\prime,\cdot}f^\prime(y_s^q)+R^{f(y),q}_{ts}.
\end{equation}
We now set $\varphi^{\prime,q,p}_s = y^{\prime,p}_s f^\prime(y^q_s)$. With relation~\eqref{convolution in 1d} in mind it is readily checked that~\eqref{eq: corect Rf(y)} can be recast as
\begin{equation}\nonumber
f(y^q)_{ts}=z_{ts}^q \ast \varphi^{\prime,q,\cdot}_s +R_{ts}^{f(y),q},
\end{equation}
which corresponds to our claim in \eqref{eq:def of comp derivative}. 

Let us now focus on Inequality \eqref{bound for function composed with controlled path}.  To this end, recall that  the norm $\|\varphi,\varphi^\prime\|_{z;(\alpha,\gamma)}$ is defined by \eqref{eq:controlled norm def}. Thus we have 
\begin{equation}\label{eq:controlled norm relation}
\|\varphi,\varphi^\prime\|_{z;(\alpha,\gamma)}=\|f\left(y\right),y^{\prime}f\left(y\right)\|_{z,\left(\alpha,\gamma\right)}=\|y^{\prime,\cdot_2}f\left(y^{\cdot_1}\right)\|_{\left(\alpha,\gamma\right)}+\|R^{f\left(y\right)}\|_{\left(2\rho+\gamma,\gamma\right)}.
\end{equation}
We shall analyse the two terms in the right hand side of \eqref{eq:controlled norm relation} separately. We start with the derivative $\varphi^\prime$, for which we will bound the two norms given by \eqref{eq:derivative norm1} and \eqref{upper vb norm}. Specifically,  
observe first that the difference in the lower variable for the derivative $\varphi^\prime$ is given by 
\[
\left(y^{\prime}f\left(y\right)\right)_{ts}^{q,p}=y_{t}^{\prime,p}f\left(y_{t}^{q}\right)-y_{s}^{\prime,p}f\left(y_{s}^{q}\right),
\]
and  thus by addition and subtraction of $y_{s}^{\prime,p}f\left(y_{t}^{q}\right)$
it is readily checked that the following bound is satisfied
\begin{equation}\label{control first f}
\| y^{\prime,\cdot_2}f\left(y^{\cdot_1}\right)\|_{\left(\alpha,\gamma\right),1}\lesssim\| f\|_{\mathcal{C}_{b}^{1}}\left(\| y^{\prime,\cdot_2}\|_{\left(\alpha,\gamma\right)}+\| y\|_{\left(\alpha,\gamma\right)}\right).
\end{equation}

Let us now consider the quantity $\| y^{\prime,\cdot_2}f\left(y^{\cdot_1}\right)\|_{\left(\alpha,\gamma\right),1,2}$.  To this end, we will in two stages encounter first order Taylor expansions, and thus we recall that for a differentiable function $f$ on $V$ we have for $a,b\in V$ 
\begin{equation}
g(a)-g(b)=L(a,b),\quad {\rm where} \quad L(a,b)=\int_0 ^1 Df(\theta a+(1-\theta)b)d\theta (a-b).
\end{equation}  
 We now need to control the simultaneous increment in the upper and lower variables according to \eqref{upper vb norm}. Let us first consider $y_{t}^{\prime,p}f\left(y_{t}^{q}\right)$ with fixed $p$ and increments in the variables $t$ and $q$.  By a simple addition and subtraction argument, we obtain the identity
\begin{equation}\label{many vb rel}
\left(y_{\cdot}^{\prime,p}f\left(y_{\cdot}^{\cdot}\right)\right)_{ts}^{qr} = y_t^{\prime,p}L(y_t^q,y_t^r)y_t^{qr}-y_s^{\prime,p}L(y_s^q,y_s^r)y_s^{qr},
\end{equation}
where $L$ is given as above. Observe now that by adding and subtracting the quantity $ y_s^{\prime,p}L(y_t^q,y_t^r)y_t^{qr}$ to the right hand side in of \eqref{many vb rel}, we obtain that 
\begin{equation}\label{up lo inc exp 1}
y_t^{\prime,p}L(y_t^q,y_t^r)y_t^{qr}-y_s^{\prime,p}L(y_s^q,y_s^r)y_s^{qr} =y_{ts}^{\prime,p}F^{qr}_t+y_s^{\prime,p}(F^{qr}_t-F^{qr}_s),
\end{equation}
where $F_t^{qr}:=L(y_t^q,y_t^r)y_{t}^{qr}$. Due to the boundedness assumption on $f$ and its derivatives, it is clear that $|F_t^{qr}|\lesssim \|f\|_{\mathcal{C}^1_b}\|y\|_{(\alpha,\gamma),1,2}|q-r|^\eta|r-t|^{-\eta}$. Furthermore, it is readily seen that 
\begin{equation}\label{bL and L y}
F^{qr}_t-F^{qr}_s={\bf L}(y_t^q,y_t^r,y_s^q,y_s^r)y_t^{qr}+L(y_t^q,y_t^r)y_{ts}^{qr},
\end{equation}
Where ${\bf L}(a,b,a',b')$ is given as the remainder of a first order two-variable Taylor approximation of $L$, in the sense that  $L(a,b)-L(a',b')={\bf L}(a,b,a',b')$,  which is explicitly given by 
\begin{multline*}
{\bf L}(a,b,a',b')\\
=
 \int_0^1 \int_0^1 D^2g(\theta'(\theta a+(1-\theta)b)+(1-\theta')(\theta a'+(1+\theta)b'))d\theta'( \theta (a-a')+(1-\theta)(b-b')) d\theta.
\end{multline*}

 Again, due to the boundedness of $f$ and its derivatives, for any $\eta\in [0,1]$ we obtain that  
\begin{equation}\label{F inc bound}
|F^{qr}_t-F^{qr}_s|\lesssim \|f\|_{\mathcal{C}^2_b}\|y\|_{(\alpha,\gamma)1,2}|q-r|^\eta|r-t|^{-\eta}\left[|r-t|^{-\gamma}|t-s|^\alpha \wedge |r-s|^{\rho}\right]
\end{equation}
Inserting relation \eqref{bL and L y} into \eqref{up lo inc exp 1}, and invoking the bound in \eqref{F inc bound} as well as the regularity of $y^\prime$ and $y$, we obtain that 
\begin{equation}
\| y^{\prime,\cdot_2}f\left(y^{\cdot_1}\right)\|_{\left(\alpha,\gamma\right),1,2,<} \leq \|f\|_{\mathcal{C}^2_b}\left(\|y^{\prime,\cdot_1,\cdot_2}\|_{(\alpha,\gamma)}+\|y\|_{(\alpha,\gamma)}\right). 
\end{equation} 
A similar argument can now be used to also show that $\| y^{\prime,\cdot_2}f\left(y^{\cdot_1}\right)\|_{\left(\alpha,\gamma\right),1,2,>}<\infty$, and thus it follows by  \eqref{upper vb norm} that 
\begin{equation}
\| y^{\prime,\cdot_2}f\left(y^{\cdot_1}\right)\|_{\left(\alpha,\gamma\right),1,2}\leq \|f\|_{\mathcal{C}^2_b}\left(\|y^{\prime,\cdot_1,\cdot_2}\|_{(\alpha,\gamma)}+\|y\|_{(\alpha,\gamma)}\right).
\end{equation}

Let us now handle the term $\|R^{f(y)}\|_{(2\rho+\gamma,\gamma)}$ in equation \eqref{eq:controlled norm relation}. More precisely  recalling that the norm
 $\|\cdot\|_{(2\rho+\gamma,\gamma)}$ is given by \eqref{eq:three variable gen norm},
  let us first bound the quantity $\|R^{f(y)}\|_{(2\rho+\gamma,\gamma),1}$. 
 Towards this aim, we go back to the definition \eqref{eq:Rf(y) relation} of $R^{f(y)}$  and apply Taylor's expansion in a standard way.
That is, define $c_{t,s}^{\tau}(a)=ay_{s}^\tau +(1-a)y_t^\tau$, and observe that  
\begin{equation}\label{Rem f(y)}
R_{ts}^{f\left(y\right),\tau}=R_{ts}^{y,\tau} f^{\prime}\left(y_{s}^{\tau}\right)+\frac{1}{2}\left(y_{ts}^{\tau}\right)^{\otimes2}\int_0^1 f^{\prime\prime}\left(c_{t,s}^{\tau}(a)\right)da. 
\end{equation}  
The regularity of $R^{f(y)}$ for the $\|\cdot\|_{(2\rho+\gamma,\gamma),1}$ norm now follows from the boundedness of the second derivative of $f$, the squared regularity of the increment of $y$ and the regularity of $R^{y}$. 

Next, we will compute the regularity in the upper argument for  $R^{f(y)}$, which corresponds to the semi-norm $\|\cdot\|_{(2\rho+\gamma,\gamma),1,2}$ in \eqref{eq: two variable  function norm change}. In particular, we will consider the increment 
\begin{multline}\label{upper arg Rf(y)}
R^{qp,f(y)}_{ts}
=R_{ts}^{p,y}f^{\prime}\left(y_{s}\right)^{qp}+R_{ts}^{qp,y}f^{\prime}\left(y_{s}^{p}\right) \\
+\frac{1}{2}\left(y_{ts}^{p}\right)^{\otimes2}\int_0^1 f^{\prime\prime}(c_{t,s}(a))^{qp}da+\frac{1}{2}(\left(y_{ts}^{q}\right)^{\otimes2}-\left(y_{ts}^{p}\right)^{\otimes2})\int_0^1 f^{\prime\prime}(c^{p}_{t,s}(a)).
\end{multline}

Using that $f\in C^{3}_{b}$ and $a^{2}-b^{2}=(a+b)(a-b)$, it follows from a combination of  (\ref{upper arg Rf(y)})  and (\ref{Rem f(y)}) that 
\begin{equation}\nonumber
\| R^{f\left(y\right)}\|_{\left(2\rho+\gamma,\gamma\right)}=\| R^{f\left(y\right)}\|_{\left(2\rho+\gamma,\gamma\right),1}+\| R^{f\left(y\right)}\|_{\left(2\rho+\gamma,\gamma\right)1,2}\leq\| f\|_{\mathcal{C}_{b}^{3}}\left(\| R^{y}\|_{\left(2\rho+\gamma,\gamma\right)}+\| y\|_{\left(\rho+\gamma,\gamma\right)}^{2}\right).
\end{equation}
We now use the fact that the regularity of the controlled Volterra path is inherited by the noise, as discussed in Remark \ref{rem: inherited reg} and see that 
\begin{equation}\label{regualrity of cont path}
\| y\|_{\left(\alpha,\gamma\right)}\leq\left(|y_{0}^{\prime}|+\| y^{\prime,\cdot_2}\|_{\left(\alpha,\gamma\right)}\right)(\| z\|_{\left(\alpha,\gamma\right)}+\| R^{y}\|_{\left(2\rho+\gamma,\gamma\right)}).
\end{equation}
Combining the information from (\ref{regualrity of cont path}), (\ref{Rem f(y)}), and (\ref{control first f}) yields (\ref{bound for function composed with controlled path}).  Namely, it follows that
\begin{align}\nonumber
&\| f\left(y\right),f\left(y\right)y^{\prime}\|_{z,\left(\alpha,\gamma\right)}
\\\nonumber
&\leq C_{\|f\|_{C^{3}_{b}},\alpha,\gamma}\left(1+\| z\|_{\left(\alpha,\gamma\right)}\right)^{2}\left[\left(|y_{0}^{\prime}|+\| y,y^{\prime}\|_{z,\left(\alpha,\gamma\right)}\right)^{2}\vee\left(|y_{0}^{\prime}|+\| y,y^{\prime}\|_{z,\left(\alpha,\gamma\right)}\right)\right].
\end{align}
\end{proof}

\begin{rem}
We point out that we require $f\in C^{3}_{b}$ in order to compose $f$  with a controlled Volterra path 
$(y,y^{\prime})\in \hat{\mathscr{D}}_{z}^{(\alpha,\gamma)}$. 
This requirement is one degree of differentiation more than what is standard in classical rough path theory (see e.g. \cite[Section 7]{FriHai}). The reason for this comes from the fact that we also need regularity in the upper argument of the controlled Volterra paths, and thus we see that in order to bound the term $f^{\prime\prime}(c)^{pq}$ in \eqref{upper arg Rf(y)},  we need $f\in C^{3}_{b}$.
\end{rem}

\subsection{Rough Volterra Equations}\label{sec:rough volt eq}

Based on the concept of controlled Volterra paths and Volterra integration introduced in Section \ref{non-linear setting}, we are now ready to  prove existence and uniqueness of non-linear Volterra equations. As we have seen so far, the results that we obtain are directly comparable to those known from the classical setting under substitution of the tensor product with the convolution product. 

\begin{thm}\label{existence and uniq}
Let $\mathbf{z}\in\mathscr{V}^{\left(\alpha,\gamma\right)}\left(E\right)$
with $\alpha-\gamma>\frac{1}{3}$. Assume that $\bz$ satisfies the same hypothesis as in Theorem~\ref{thm:Volterra integral of controlled path is controlled path} and suppose $f\in\mathcal{C}_{b}^{4}\left(V;\mathcal{L}\left(E,V\right)\right)$.
Then there exists a unique Volterra solution in $\hat{\mathscr{D}}_{\bz^{1}}^{\left(\alpha,\gamma\right)}\left(V\right)$
to the equation 
\begin{equation}\label{volterra eq in thm}
y_{t}^{\tau}=y_{0}+\int_{0}^{t}k\left(\tau,r\right)f\left(y_{r}^{r}\right)dx_{r},\,\,\,\left(t,\tau\right)\in\Delta_{2}\left(\left[0,T\right]\right),\,\,\,\,y_{0}\in E,
\end{equation}
where the integral is understood as a rough Volterra integral given in Theorem \ref{thm:Volterra integral of controlled path is controlled path}.
\end{thm}

\begin{proof}
The parameter $(s,\tau)$ we consider in this proof sits in a small variation of the simplex $\Delta_2$ defined by \eqref{simplex}. Namely we define the trapezoid
\begin{equation}\label{trapezoid}
\Delta_{2}^{T}\left(\left[a,b\right]\right)=\left\{ \left(s,\tau\right)\in\left[a,b\right]\times\left[0,T\right]|a\leq s\leq\tau\leq T\right\} ,
\end{equation}
and note that the first component of $\left(s,\tau\right)\in \Delta^T_2([a,b])$
is restricted to $\left[a,b\right]$ and the second component to $\left[0,T\right]$. For simplicity, assume that $\|\bz\|_{(\alpha,\gamma)}  \leq M\in \mathbb{R}_ {+}$. Furthermore, throughout the proof we will consider a subset of $\hat{\mathscr{D}}_{\bz^{1}}^{\left(\beta,\gamma\right)}\left(\Delta_{2}^{T}\left(\left[0,\bar{T}\right]\right);V\right) $ of paths $(y,y^\prime)$ starting in $(y_0,f(y_0))$. With a slight abuse of notation we still denote this subset by  $\hat{\mathscr{D}}_{\bz^{1}}^{\left(\beta,\gamma\right)}\left(\Delta_{2}^{T}\left(\left[0,\bar{T}\right]\right);V\right) $.

We start by considering $\bar{T},\beta$ such that $0<\bar{T}\leq T$ and $\beta<\alpha$ and $\beta-\gamma>\frac{1}{3}$ (note that this is made possible thanks to the fact that $\alpha-\gamma>\frac{1}{3}$ ). With Definition   \ref{Volterra controlled path} and our notation \eqref{trapezoid} in mind, we introduce a mapping
\begin{equation}
\mathcal{M}_{\bar{T}}:\hat{\mathscr{D}}_{\bz^{1}}^{\left(\beta,\gamma\right)}\left(\Delta_{2}^{T}\left(\left[0,\bar{T}\right]\right);V\right)\rightarrow \hat{\mathscr{D}}_{\bz^{1}}^{\left(\beta,\gamma\right)}\left(\Delta_{2}^{T}\left(\left[0,\bar{T}\right]\right);V\right)
\end{equation}
such that for all $\left(y,y^{\prime}\right)\in\hat{\mathscr{D}}_{\bz^{1}}^{\left(\beta,\gamma\right)}\left(V\right)$
we have 
\[
\mathcal{M}_{\bar{T}}\left(y,y^{\prime}\right)=\left\{ 
\lp y_{0}+\int_{0}^{t}k\left(\tau,r\right)f\left(y_{r}^{r}\right)dx_{r}, \, f\left(y_{t}^{\tau}\right) \rp
\Big| \left(t,\tau\right)\in\Delta_{2}^{T}\left(\left[0,\bar{T}\right]\right)\right\}. 
\]
Our aim is to prove that if $\bar{T}$ is chosen to be small enough, then $\mathcal{M}_{\bar{T}}$ is a contraction. A first step in this direction is obtained by a direct application of Theorem \ref{thm:Volterra integral of controlled path is controlled path}, where the norms are restricted to $\Delta_2^T([0,\bar{T}])$.
 With the additional notation 
\begin{equation}\label{ww map}
\left(s,t,\tau\right)\mapsto\left(w_{ts}^{\tau},w_{ts}^{\prime,\tau}\right)=\mathcal{M}_{\bar{T}}\left(y,y^{\prime}\right)_{ts}^{\tau},
\end{equation}
we easily get 
\[
\| w,w^{\prime}\|_{\bz^{1};\left(\beta,\gamma\right)}\leq\| f\left(y\right),f\left(y\right)f^{\prime}\left(y\right)\|_{\bz^{1},\left(\beta,\gamma\right)}\| \bz\|_{\left(\alpha,\gamma\right)}\bar{T}^{\beta-\gamma},
\]
where we recall our notation \eqref{eq:bz notation} for $\| \bz\|_{\left(\alpha,\gamma\right)}$. Furthermore,  it follows from the fact that any composition of a $C^{3}_{b}$ function with a controlled Volterra path in $\hat{\mathscr{D}}^{(\beta,\gamma)}_{\bz^1}$ is again a Volterra path (see Proposition \ref{prop: composition with reg function}) that
\begin{equation}
\| w,w^{\prime}\|_{\left(\beta,\gamma\right)}\leq C \left[\left(|y^\prime_{0}|+\| y,y^{\prime}\|_{\left(\beta,\gamma\right),\bz^{1}}\right)^{2}\vee\left(|y^\prime_{0}|+\| y,y^{\prime}\|_{\left(\beta,\gamma\right),\bz^{1}}\right)\right]\|\mathbf{z}\|_{\left(\alpha,\gamma\right)}\bar{T}^{\beta-\alpha},\label{eq:invariance of ball inequality}
\end{equation}
where we recall that we assume $\|\mathbf{z}\|_{\left(\alpha,\gamma\right)}\leq M$. 

Next we will show that there exists a ball of radius $1$ centred at a trivial element in $\hat{\mathscr{D}}_{\bz^{1}}^{(\beta,\gamma)}(\Delta_{2}^{T}([0,\bar{T}]);V)$, which is left invariant by $\mathcal{M}_{\bar{T}}$, provided that $\bar{T}$ is small enough. Namely  consider the trivial path $(t,\tau)\mapsto (c_t^\tau,c^{\prime,\tau,\cdot}_t)$ defined in the following way
\begin{equation}\nonumber
\left(c_{t}^{\tau},c_{t}^{\prime,\tau,\cdot}\right)=\left(y_{0}+\bz_{t0}^{1,\tau}f\left(y_{0}\right),f(y_{0})\right),
\end{equation}
where we recall that $y_0$ is the element in $V$ such that $y^\tau_0=y_0$ for all $\tau\in [0,T]$. Note that this element satisfies $\|c,c^\prime\|_{\bz^1,(\beta,\gamma)}=0$, due to invariance of Hölder norms to translations by constants, and that $R^{c,\tau}_{ts}=0$ for all $(s,t,\tau)\in \Delta_3([0,T])$. 
 Next consider the unit ball $\mathcal{B}_{\bar{T}} $ centred at the element $(c,c^\prime)$ of $ \hat{\mathscr{D}}_{\bz^{1}}^{(\beta,\gamma)}(\Delta_{2}^{T}([0,\bar{T}]);V)$ defined by 
\begin{multline}\label{eq:ball}
\mathcal{B}_{\bar{T}}= 
\Big\{ \left(y,y^{\prime}\right)\in\hat{\mathscr{D}}_{\bz^{1}}^{\left(\beta,\gamma\right)}\left(\Delta_{2}^{T}\left(\left[0,\bar{T}\right]\right);V\right)
\Big| \, y_{0}^{\tau}=y_{0},\,\, 
\text{and } y^{\prime,\tau,\cdot}_0=f(y_0), \\
 \text{with}\,\,\,\,\| y-c,y^{\prime}-c^{\prime}\|_{\bz^{1},\left(\beta,\gamma\right)}\leq1
  \Big\} .
\end{multline}
Again we observe that, thanks to the invariance of H\"older norms by translations by constants and according to the fact that $R^{c,\tau}_{ts}=0$ for all $(s,t,\tau)\in \Delta_3([0,T])$, we have 
\begin{equation}\nonumber
\|y,y^\prime\|_{\bz^1,(\beta,\gamma)}=\|y-c,y^\prime-c^\prime\|_{\bz^1,(\beta,\gamma)}
\end{equation}
 for all $(y,y^\prime)\in \mathcal{B}_{\bar{T}}$ defined as in \eqref{eq:ball}.

Consider now $(y,y^\prime)\in \mathcal{B}_{\bar{T}}$ and define $(w,w^\prime)$ as in \eqref{ww map}. Thanks to the fact that $y_0^{\prime,\cdot}=f(y_0)$, together with the assumption that $f$ is bounded (recall that $f\in C^4_b$), relation~\eqref{eq:invariance of ball inequality} can be read as 
\begin{equation}\label{five thirty nine}
\|w,w^\prime\|_{\bz^1,(\beta,\gamma)}\leq C\left(1+\|y,y^\prime\|_{\bz^1,(\beta,\gamma)} \right)^2 
\|\bz\|_{(\alpha,\gamma)} \, \bar{T}^{\alpha-\beta}.
\end{equation}
Moreover, since $\|y-c,y^\prime-c^\prime\|_{\bz^1,(\beta,\gamma)} \le 1$, we easily get 
\[
\|\mathcal{M}_{\bar{T}}\left(y,y^{\prime}\right)\|_{\bz^{1},\left(\beta,\gamma\right);\Delta_{2}^{T}\left(\left[0,\bar{T}\right]\right)}\leq C \|\bz\|_{(\alpha,\gamma)} \,\bar{T}^{\alpha-\beta}.
\]
We now choose $\bar{T}$ satisfying $C\|\bz\|_{(\al,\gamma)}\bar{T}^{\alpha-\beta}=\frac{1}{2}$, and we obtain that 
$(w,w^\prime)$ is an element of $\mathcal{B}_{\bar{T}}$.  Summarizing our considerations so far, we end up with the relation 
\begin{equation}\label{implies rel}
C\|\bz\|_{(\al,\gamma)} \, \bar{T}^{\alpha-\beta}=\frac{1}{2} 
\qquad \implies \qquad 
\mathcal{B}_{\bar{T}} \, \textrm{ is left invariant by } \mathcal{M}_{\bar{T}}.
\end{equation}
Notice the condition on $\bar{T}$  in relation  \eqref{implies rel} does not depend on the initial condition $y_0$.

Next, we will prove that $\mathcal{M}_{\bar{T}}$ is a contraction
on $\hat{\mathscr{D}}_{\bz^{1}}^{(\alpha,\gamma)}(\Delta_{2}^{T}([0,\bar{T}]);V)$,
i.e. we will prove that for two controlled Volterra paths $(y,y^{\prime})$
and $(\tilde{y},\tilde{y}^{\prime})$ in 
$\hat{\mathscr{D}}_{\bz^{1}}^{(\beta,\gamma)}(\Delta_{2}^{T}([0,\bar{T}]);V)$
there exists a $q\in\left(0,1\right)$ such that 
\begin{equation}\label{contraction}
\|\mathcal{M}_{\bar{T}}\left(y-\tilde{y},y^{\prime}-\tilde{y}^{\prime}\right)\|_{\bz^{1},\left(\beta,\gamma\right);\Delta_{2}^{T}\left(\left[0,\bar{T}\right]\right)}\leq q\| y-\tilde{y},y^{\prime}-\tilde{y}^{\prime}\|_{\bz^{1},\left(\beta,\gamma\right);\Delta_{2}^{T}\left(\left[0,\bar{T}\right]\right)}.
\end{equation}
Without loss of generality, and with a slight abuse of notation,  we will from now denote by  $\mathscr{D}_{\bz^{1}}^{(\beta,\gamma)}(\Delta_{2}^{T}([0,\bar{T}]);V)$  the space of controlled  Volterra paths starting from the point $y_{0}\in V$. Thus, the two paths $(y,y^{\prime})$ and $(\tilde{y},\tilde{y}^{\prime})\in\mathscr{D}_{\bz^{1}}^{(\beta,\gamma)}(\Delta_{2}^{T}([0,\bar{T}]);V)$ share the same  initial value. 
Since $\mathscr{D}_{\bz^{1}}^{(\beta,\gamma)}(\Delta_{2}^{T}([0,\bar{T}]);V)$
is a linear space, we may define 
\begin{equation}\label{eq:F Fprime}
(F,F^{\prime})
=
\left(f\left(y\right)-f\left(\tilde{y}\right),
f^{\prime}\left(y^{\cdot_{2}}\right)f\left(y^{\cdot_{1}}\right)
-f^{\prime}\left(\tilde{y}^{\cdot_{2}}\right)f\left(\tilde{y}^{\cdot_{1}}\right)
\right),
\end{equation}
where $(F,F^{\prime})$ has to be seen as an element of 
$\hat{\mathscr{D}}_{\bz^{1}}^{(\beta,\gamma)}(\Delta_{2}^{T}([0,\bar{T}]);V)$.
Thus we have 
\begin{equation}\label{eq:MbarT map}
\mathcal{M}_{\bar{T}}(y-\tilde{y},y^\prime-\tilde{y}^\prime)_t^{\tau} = \int_0^t k(\tau,r)F_r^r dx_r,
\end{equation}
where we observe that the initial condition is now $0$. In order to bound the right hand side of \eqref{eq:MbarT map} we now apply Theorem \ref{thm:Volterra integral of controlled path is controlled path} (in particular equation \eqref{eq:bound volterra ointegration}), which yields
\begin{multline}\label{eq:contraction first ineq}
\|\mathcal{M}_{\bar{T}}\left(y-\tilde{y},y^{\prime}-\tilde{y}^{\prime}\right)\|_{\bz^{1}, \left(\beta,\gamma\right)}\leq  
\|F\|_{\left(\beta,\gamma\right)}+\|F^{\prime}\|_{\infty}\|\bz^{2}\|_{(2\rho+\gamma,\gamma)}\bar{T}^{2(\alpha-\beta)}
\\+
C\|F,F^{\prime}\|_{\bz^{1},\left(\beta,\gamma\right)}
\left(\|\bz^{1}\|_{(\alpha,\gamma)} +\|\bz^{2}\|_{(2\rho+\gamma,\gamma)}\right) \bar{T}^{3\alpha-\gamma-2\beta},
\end{multline}
where we have used that $\rho=\alpha-\gamma$. In \eqref{eq:contraction first ineq} notice that the quantity $\|F\|_{(\beta,\gamma)}$ comes from the term $\|y^{\prime,\cdot_1,\cdot_2}\|_{(\alpha,\gamma)}$ in the definition \eqref{eq:controlled norm def} of the norm $\|y,y^\prime\|_{\bz^1;(\alpha,\gamma)}$, together with the fact that 
\begin{equation}\label{eq:MbarT prime map}
\left[\mathcal{M}_{\bar{T}}(y-\tilde{y},y^\prime-\tilde{y}^\prime)\right]^{\prime,\cdot_1,\cdot_2} = F^{\cdot_2}.
\end{equation}
Also observe that the other terms in the right hand side of \eqref{eq:contraction first ineq} correspond to the evaluation of the remainder for $\mathcal{M}_{\bar{T}} (y-\tilde{y},y^\prime-\tilde{y}^\prime)$, which is obtained by invoking relation~\eqref{eq:bound volterra ointegration}.

Let us now describe how to  get the contraction term $\bar{T}^{\alpha-\beta}$ in front of the $\|F\|_{(\beta,\gamma)}$ term in \eqref{eq:contraction first ineq}. Indeed, even though we consider $(y,y^\prime)$ and $(\tilde{y},\tilde{y}^\prime)$ as elements of $\hat{\mathscr{D}}^{(\beta,\gamma)}_{\bz^1}$, our decomposition \eqref{controlled rel} reveals that their H\"older regularity is dictated by  $\bz^1$ (see also Remark \ref{rem: inherited reg} for a similar observation). Therefore using the expression \eqref{eq:F Fprime} for $F$ and arguments similar to Proposition \ref{prop: composition with reg function}, we get 
\begin{align}\label{eq:contraction second ineq}
\|F^{\cdot_2}\|_{(\beta,\gamma)}
\leq 
C\|y-\tilde{y}\|_{(\beta,\gamma)}\|\bz^{1}\|_{(\alpha,\gamma)}\bar{T}^{\alpha-\beta}
\end{align}
Combining (\ref{eq:contraction first ineq}) and (\ref{eq:contraction second ineq}) we can see that 
\begin{equation}\label{contraction F,F}
\|\mathcal{M}_{\bar{T}}\left(y-\tilde{y},y^{\prime}-\tilde{y}^{\prime}\right)\|_{\bz^{1}, \left(\beta,\gamma\right)}
\leq C_{M,\alpha,\beta,\gamma}\left[\|y-\tilde{y}\|_{(\beta,\gamma)}+\|F,F^{\prime}\|_{\bz^{1},\left(\beta,\gamma\right)}\right]\bar{T}^{\alpha-\beta}. 
\end{equation}
The dependence on $\bar{T} $ on the left hand side will later allow us to use this parameter to create a constant $q\in (0,1)$  such that (\ref{contraction}) holds, similar to the argument for the invariance property of the unit ball.  Next we will prove that 
\begin{equation}\label{eq:F Fprime bound}
\|F,F^{\prime}\|_{\bz^{1},\left(\beta,\gamma\right)}\lesssim \|y-\tilde{y},y^{\prime}-\tilde{y}^{\prime}\|_{\bz^{1},(\alpha,\gamma)}. 
\end{equation} 

We will mainly focus on the term $\|F^{\prime,\cdot_1,\cdot_2}\|_{(\beta,\gamma)}$ , the remainder $R^F$ being treated similarly. 
Now recall from \eqref{eq:F Fprime} that $F^{\prime,\cdot_1,\cdot_2}(y,y^\prime)=f^{\prime}\left(y^{\cdot_2}\right)f\left(y^{\cdot_1}\right)-f^{\prime}\left(\tilde{y}^{\cdot_2}\right)f\left(\tilde{y}^{\cdot_1}\right)$. To be able to treat the fact that we have two upper variables to take care of, we do a simple addition and subtraction to see that 
\begin{equation}\label{mid add sub}
F^{\prime,\cdot_1,\cdot_2}(y,y^\prime)=\left(f^\prime(y^{\cdot_2})-f^\prime(\tilde{y}^{\cdot_2})\right)f(y^{\cdot_1})+f^\prime(\tilde{y}^{\cdot_2})\left(f(y^{\cdot_1})-f(\tilde{y}^{\cdot_1})\right).
\end{equation}
By invoking the fact that $f\in C^4_b$, let $g$ and $ h$ denote the remainders from a first order Taylor expansion of the differences $f(y)-f(\tilde{y})$ and $f^\prime(y)-f^\prime(\tilde{y})$.  Note in particular that this implies that $g\in C^3_b$ and $h\in C^2_b$, and we have that $\|g\|_{C^3_b}\vee \|h\|_{C^2_b} \leq  \|f\|_{C^4_b}$. Then it follows from~\eqref{mid add sub} that for any $t\in [0,\bar{T}]$ we have 
 \begin{align}\nonumber
 F^{\prime,\cdot_1,\cdot_2}(y,y^\prime)_t&=g(y_t,\tilde{y}_t)\left(y^{\cdot_2}_t-\tilde{y}^{\cdot_2}_t\right)f(y^{\cdot_1}_t)+f^\prime(\tilde{y}^{\cdot_2}_t)h(y^{\cdot_1}_t,\tilde{y}^{\cdot_1}_t)\left(y^{\cdot_1}_t-\tilde{y}^{\cdot_1}_t\right)
 \\\label{I1 and I2}
 &=:I_t^{1,\cdot_1,\cdot_2}+I_t^{2,\cdot_1,\cdot_2}.
 \end{align}
 Let us now consider the increment $I_{ts}^{1,\tau,\tau}$. By  elementary addition of subtraction of terms coming from $g$ and $f$, we obtain that 
 \begin{equation}
 |I_{ts}^{1,\tau,\tau}|\leq C_{\|f\|_{C^3_b}} \|y-\tilde{y}\|_{(\beta,\gamma),1}\left(|\tau-t|^{-\gamma}|t-s|^\beta \wedge |\tau-s|^{\beta-\gamma}\right),
\end{equation} 
  from which it follows that $\|I^1\|_{(\beta,\gamma),1}<\infty$. A similar argument can be used to show that also $\|I^2\|_{(\beta,\gamma),1}<\infty$, however in this case we get dependence on the norm $\|f\|_{C^4_b}$ in the bounding constant. Putting the two terms together, and invoking the relation in \eqref{I1 and I2}, we observe that 
  \begin{equation}
  \|F^{\prime,\cdot_1,\cdot_2}\|_{(\beta,\gamma),1} \lesssim \|y-\tilde{y},y^\prime-\tilde{y}^\prime\|^2_{\bz^1,(\beta,\gamma)} \lesssim \|y-\tilde{y},y^\prime-\tilde{y}^\prime\|_{\bz^1,(\beta,\gamma)},
  \end{equation}
where we have invoked the fact that $(y,y^\prime),(\tilde{y},\tilde{y}^\prime)\in \mathcal{B}_{\bar{T}}$ for the second inequality. The quantity $\|F^{\prime,\cdot_1,\cdot_2}\|_{(\beta,\gamma),1,2}$ can be bounded using a similar argument, and we leave this component for the patient reader, for conciseness of the proof. It follows that 
\begin{equation}
  \|F^{\prime,\cdot_1,\cdot_2}\|_{(\beta,\gamma)}  \lesssim \|y-\tilde{y},y^\prime-\tilde{y}^\prime\|_{\bz^1,(\beta,\gamma)},
\end{equation}
 and our claim~\eqref{eq:F Fprime bound} is now proved. 

In conclusion of this step, we are ready to state the desired contraction property on a small interval $[0,\bar{T}]$. Indeed, plugging \eqref{eq:F Fprime bound} into \eqref{contraction F,F} we obtain 
 \begin{equation}\nonumber
 \|\mathcal{M}_{\bar{T}}\left(y-\tilde{y},y^{\prime}-\tilde{y}^{\prime}\right)\|_{\bz^{1}, \left(\beta,\gamma\right)}
\leq C\|y-\tilde{y},y^{\prime}-\tilde{y}^{\prime}\|_{\bz^{1},\left(\beta,\gamma\right)}\bar{T}^{\alpha-\beta}.
 \end{equation}
By choosing $\bar{T}$ small enough,  it is clear that there exists a $q\in(0,1)$  such that (\ref{contraction}) holds. It follows that $\mathcal{M}_{\bar{T}}$ 
admits fixed point in $ \hat{\mathscr{D}}_{\bz^{1}}^{(\beta,\gamma)}(\Delta_{2}^{T}([0,\bar{T}]);V)$, 
and thus existence and uniqueness of Equation (\ref{volterra eq in thm}) on $\Delta_{2}^{T}([0,\bar{T}])$ is established.
 Next we want to extend the solution to all of $\Delta_{2}$, 
 which we do by constructing a solution on all intervals of length $ \bar{T}$.
  That is,  we construct a solution to (\ref{volterra eq in thm}) 
  on $\Delta_{2}^{T}\left([\bar{T},2\bar{T}]\right)$ using the terminal value of the solution created on $\Delta_{2}^{T}([0,\bar{T}])$.   Note that for any
   $(t,\tau)\in \Delta_{2}^{T}([k\bar{T},(k+1)\bar{T}])\subset \Delta_{2}$  for some $k\geq 1$ 
   we formally have that 
\begin{equation}\nonumber
y_{t}^{\tau}=y_{k\bar{T}}^{\tau}+\int_{k\bar{T}}^{t}k(\tau,r)f(y_{r}^{r})dx_{r}. 
\end{equation}
It follows, similarly as in the classical results on existence and uniqueness of SDEs, that there exists a solution on all subintervals of length $\bar{T}$, 
i.e. all intervals  $[a,a+\bar{T}]\subset [0,T]$ for some $a\geq 0$. All these solutions are connected on the boundaries, and thus we use that a function which is H\"older on any subinterval $[a,a\bar{T}]\subset [0,T]$  
of length $\bar{T}$ is also H\"older continuous on $[0,T]$ (see e.g. \cite{FriHai}, exercise 4.24), which applies to the H\"older continuity in both variables.
Here notice that the time step $\bar{T}$ can be made constant thanks to the fact that $f$ is a bounded function (see relation \eqref{five thirty nine}). 
 
 We can conclude that there exists a unique global  solution to Equation (\ref{volterra eq in thm}) in the space $\hat{\mathscr{D}}_{\bz^{1}}^{(\beta,\gamma)}(\Delta_{2};V)$ for $\beta<\alpha$. 
Actually, by (\ref{inherit reg of controled path}) it is clear that the solution inherits the regularity of the controlling noise, and thus, the solution is in $\hat{\mathscr{D}}_{\bz^{1}}^{(\alpha,\gamma)}(\Delta_{2};V)$.
\end{proof}

\begin{rem}\label{rmk:comparison-with-promel}
We would like to point out that the existence and uniqueness of Equation~(\ref{volterra eq in thm}) requires one more degree of regularity on the diffusion coefficient $f$ than what is standard for regular Rough differential equations (see e.g.  \cite{FriHai} section 8). This higher regularity requirement comes from the fact that we need control of the H\"older regularity of the upper argument when composing a function with a controlled Volterra path, as seen in (\ref{upper arg Rf(y)}). This is in contrast to \cite{TrabsPromel} where the authors only need a $C^{3}_{b}$ diffusion coefficients. However, \cite{TrabsPromel} is restricted to the case of a coefficient $f$ such that $f(0)=0$ and to Volterra equations with kernels which can be written as $k(t,s)=k(t-s)$.
\end{rem}

\begin{rem}
Although Equation (\ref{volterra eq in thm}) is a two parameter object, we can study the solution on the diagonal of $\Delta_{2}$ to obtain the classical type of one parameter Volterra equations. The H\"older continuity on the diagonal is already guaranteed by the H\"older topologies used on the space of controlled paths. In particular, there exists a unique solution to the equation
\begin{equation}\nonumber
y_{t}\equiv y_{t}^{t}=y_{0}+\int_{0}^{t}k(t,r)f(y_{r}^{r})dx_{r},\,\,y_{0}\in V.
\end{equation}
One can easily check that $t\mapsto y_{t}\in\mathcal{C}^{\rho} $ for $\rho=\alpha-\gamma$, where $\mathcal{C}^{\rho}$ denotes the classical H\"older spaces of order $\rho$. 
\end{rem}

\subsection{Discussion}
Theorem \ref{existence and uniq} tells us that for any $T>0$ there exists a solution to Equation~\eqref{volterra eq in thm} on $[0,T]$ for any singular Volterra kernel satisfying ($\mathbf{H}$) (in particular, as mentioned in Remark~\ref{rmk:comparison-with-promel}, we do not require a convolutional type of kernel like in~\cite{TrabsPromel}). 
Furthermore, since the extension developed here is fully based on the framework of classical rough path, one can also construct solutions to equations driven 
by lower regularity noise 
(i.e. with $\rho=\alpha-\gamma$ positive but lower than $1/3$). In fact, let $n$ be the whole number part of $1/\rho$. One can extend  Definition \ref{Volterra controlled 
path} to any regularity $\alpha$ by considering a formal expansion of a path to degree $n$ such that the
 $j$-th Volterra-Gubinelli derivative is  convoluted with the $(j+1)$-th term in the
  Volterra rough path, namely
\begin{equation}\nonumber
y_{ts}^{\tau}=\sum_{j=1}^{n-1} \bz^{j,\tau}_{ts}\ast y_{s}^{j,\tau,\cdot_{j},...\cdot_{1}}+R_{ts}^{\tau},
\end{equation}  
where $R\in \mathcal{V}^{(n-1)\rho+\gamma,\gamma}_{2}$, and each derivative
 $y^{j}\in \mathcal{V}^{(\alpha,\gamma)}$ for $j=1,\ldots,n-1$. We will perform this construction more explicitly in the forthcoming paper \cite{HTW}.


\begin{thebibliography}{99}


\bibitem{BFGMS19}
 Christian Bayer, Peter K. Friz, Paul Gassiat, Jorg Martin and Benjamin Stemper
 \newblock {\em A regularity structure for rough volatility}
 \newblock Mathematical Finance. An International Journal of Mathematics,
              Statistics and Financial Economics, Vol. 30, No. 3, pp. 782--832. 2020



\bibitem{BergMiz}
M. A. Berger and V. J. Mizel,
\newblock {\em  Volterra Equations with It\^{o} Integrals-I}.
\newblock   Journal of Integral Equations
Vol. 2, No. 3,  1980, pp. 187-245.

\bibitem{Chen}
K.T. Chen,
\newblock {\em Integration of paths -a faithful representation of paths by
  non-commutative formal power series}.
\newblock  Transactions of the American Mathematical Society, Vol. 89, No. 2, 1958, pp. 395-407.

\bibitem{CoGeLeSo}
 G. W. Cochran,  J. S. Lee and J. Potthoff, 
\newblock {\em Stochastic Volterra equations with singular kernel}.
\newblock  Stochastic Processes and  Applications, Vol. 56, No. 2,  1995, pp. 337-349.


\bibitem{Crom}
T. L. Cromer,
\newblock {\em Asymptotically periodic 7solutions to volterra integral equations in
  epidemic models}.
\newblock Journal of Mathematical analysis and Applications, Vol. 110, No. 2,  1985, pp. 483-494.

\bibitem{DietFred}
A.D. Freed and K. Diethelm,
\newblock {\em On the solution of non-linear fractional-order differential equations
  used in the modelling of viscoplasticity}.
\newblock  Keil F., Mackens W., Voss H., Werther J. (eds) Scientific
  Computing in Chemical Engineering II,Springer, Berlin, Heidelberg, 1999.

\bibitem{DD}
F. Delarue and  R. Diel,
\newblock {\em Rough paths and 1d sde with a time dependent distributional drift. Application to polymers}.
\newblock Probability Theory and Related Fields, 2016, No. 1-2, pp. 1-63. 

\bibitem{DeyaTIndel1}
A. Deya and S. Tindel,
\newblock {\em Rough volterra equations 1: The algebraic integration setting}.
\newblock  Stochastics and Dynamics, Vol. 09, No. 03, 2009, pp. 437-477.

\bibitem{DeyaTIndel2}
A. Deya and S. Tindel,
\newblock {\em Rough volterra equations 2: Convolutional generalized integrals}.
\newblock Stochastic Processes and Applications, Volume 121, Issue 8, 2011, pp. 1864-1899.

\bibitem{ER19}
Omar El Euch and Mathieu Rosenbaum 
\newblock {\em The characteristic function of rough {H}eston models}
\newblock Mathematical Finance, Vol. 29, No. 1, pp. 3--38, 2019



\bibitem{Fell}
W. Feller,
\newblock {\em On the integral equations of renewal theory}.
\newblock  Annals of Mathematical Statistics, Vol. 12, No. 3,  1941, pp. 243-267.







  
 \bibitem{FGGR}
 P.K. Friz, B. Gess, A. Gulisashvili and S. Riedel,
 \newblock {\em Jain-Monrad criterion for rough paths}.
\newblock  Annals of Probability, Vol. 44, No. 1, 2016, pp. 684-738.

\bibitem{FriHai}
M.~Hairer P.~Friz,
\newblock {\em A Course on Rough Paths with an introduction to regularity
  structures}.
\newblock Springer, 2014.

\bibitem{FriVic}
 P.~Friz and N.~Victoir,
\newblock {\em Multidimensional Stochastic Processes as Rough Paths}.
\newblock Cambridge Studies in Advanced Mathematics, 2009.

 \bibitem{Gubinelli}
M. Gubinelli,
\newblock {\em Controlling rough paths}.
\newblock Journal of Functional Analysis, Vol. 216, No. 3, 2003, pp. 86-140.

\bibitem{GT} M. Gubinelli, S. Tindel, 
\newblock {\em Rough evolution equations}. 
\newblock The Annals of Probability, Vol. 38, No. 1, 2010, pp. 1-75.

\bibitem{GubPerImk}
M.~Gubinelli, P.~Imkeller and N.~Perkowski,
\newblock {\em Paracontrolled distributions and singular pdes}.
\newblock  Forum of Mathematics. Pi, Vol. 3, 2015, e6.

 
  \bibitem{GOT}
  B. Gess, C. Ouyang and S. Tindel,
  \newblock {\em Density bounds for solutions  to differential equations driven by Gaussian rough paths}.
 \newblock To appear in Journal of  Theoretical Probability.
  
  \bibitem{Hairer}
  M. Hairer,
  \newblock {\em A theory of Regularity Structures}.
  \newblock Inventiones Mathematicae, Vol. 198, Issue 2,  2014, p. 269-504. 

\bibitem{HTW}
  F. Harang, S. Tindel. X. Wang,
  \newblock {\em Volterra equations driven by rough signals 2:
higher order expansions}.
  \newblock In preparation.

  
 \bibitem{KS}
 I. Karatzas and S. Shreve,
 \newblock {\em Brownian Motion and
Stochastic Calculus}.
\newblock Springer, 1991.

 \bibitem{LyonsLevy}
M.~Caruana,  T.~L\'evy and T.~Lyons,
\newblock {\em Differential Equations driven by Rough Paths}.
\newblock  Springer lecture series, École d'Été de Probabilités de Saint-Flour XXXIV, 2004.

\bibitem{LQZ} 
 M. Ledoux, Z. Qian and T. Zhang, 
\newblock {\em Large deviations and support theorem for diffusion processes via rough paths}. 
\newblock Stochastic Processes and  Applications, Vol. 102, No. 2, 2002, pp. 265-283. 
 
 \bibitem{LY}
 T. Lyons, 
 \newblock {\em Differential equations driven by rough signals}.
 \newblock Revista Matem\'{a}tica Iberoamericana, Vol. 14, No. 2, 1998, pp. 215-310.

\bibitem{NNT}
A. Neuenkirch, I. Nourdin and S.  Tindel,  
\newblock {\em Delay equations driven by rough paths.} 
 \newblock Electronic Journal of  Probability, Vol. 13, No. 67,  2008, pp. 2031-2068.
  
\bibitem{NuaTind}
D. Nualart and S. Tindel,
\newblock {\em A construction of the rough path above fractional Brownian motion using Volterra's representation}. 
\newblock Annals of  Probability, Vol. 39, No. 3, 2011,  pp. 1061-1096. 
  
  \bibitem{OksZha}
B.~Oksendal and T. S. Zhang,
\newblock {\em The stochastic Volterra equation}.
\newblock  D. Nualart, M. Sanz-Sol\'e, eds. The Barcelona Seminar on
  Stochastic Analysis, Basel, Vol. 32, 1993, pp. 168-202.


\bibitem{Pr}
P. Protter,
\newblock {\em Volterra equations driven by semi-martingales}.  
\newblock Annals of  Probability, Vol. 13, No. 2, 1985, pp. 519-530.


\bibitem{SKM} S. Samko, A. Kilbas and O. Marichev,
\newblock {\em Fractional Integrals and Derivatives}. 
\newblock {Gordon and Breach Science Publishers, Amsterdam, 1993.}

\bibitem{TrabsPromel}
D. J. Pr\"omel and  M. Trabs,
\newblock {\em Paracontrolled distribution approach to stochastic Volterra equations}.
\newblock arXiv:1812.05456, V2, September 2019.

\bibitem{Young}
L.C. Young,
\newblock {\em An inequality of the holder type, connected with Stieltjes 
  integration}.
\newblock  Acta Mathematica, Vol. 67, No. 1, 1936, pp. 251-282.

\bibitem{Zhang}
X. Zhang,
\newblock {\em Stochastic Volterra equations in Banach spaces and stochastic
  partial differential equations}.
\newblock Journal Of Functional Analysis, Vol. 258, No. 4,  2010, pp. 1361-1425.



 
\end{thebibliography}

\end{document}